\documentclass[12pt,a4paper,leqno]{article}
\usepackage{latexsym}
\usepackage{amsmath}
\usepackage{amssymb,amscd,amsthm,epic,url}
\usepackage[all]{xy}
\usepackage[dvips]{graphicx}
\usepackage[english,francais]{babel}
\newtheorem{theo}{{Theorem}}[section]
\newtheorem{coro}[theo]{{Corollary}}
\newtheorem{lemma}[theo]{{Lemma}}
\newtheorem{prop}[theo]{Proposition}

\newtheorem*{claim}{{Claim}}
\newtheorem*{ntheo}{{Theorem}}
\theoremstyle{definition} \newtheorem{remark}[theo]{\textbf{Remark}}
\newtheorem{defn}[theo]{Definition}
\newtheorem{example}[theo]{Example}
\newcommand{\ra}{\rightarrow}

\newcommand{\ol}{\overline}
\newcommand{\ul}{\underline}

\newcommand{\rA}{\mathrm{A}}
\newcommand{\rB}{\mathrm{B}}
\newcommand{\rC}{\mathrm{C}}
\newcommand{\rD}{\mathrm{D}}
\newcommand{\rS}{\mathrm{S}}
\newcommand{\rG}{\mathrm{G}}
\newcommand{\rH}{\mathrm{H}}
\newcommand{\rI}{\mathrm{I}}
\newcommand{\rJ}{\mathrm{J}}
\newcommand{\rT}{\mathrm{T}}
\newcommand{\rM}{\mathrm{M}}
\newcommand{\rN}{\mathrm{N}}
\newcommand{\rL}{\mathrm{L}}
\newcommand{\rP}{\mathrm{P}}
\newcommand{\rQ}{\mathrm{Q}}
\newcommand{\rR}{\mathrm{R}}
\newcommand{\rE}{\mathrm{E}}
\newcommand{\rF}{\mathrm{F}}

\newcommand{\rU}{\mathrm{U}}

\newcommand{\Nr}{\mathrm{Nr}}
\newcommand{\rW}{\mathrm{W}}
\newcommand{\rX}{\mathrm{X}}
\newcommand{\rY}{\mathrm{Y}}
\newcommand{\rZ}{\mathrm{Z}}
\newcommand{\Cl}{\mathrm{Cl}}

\newcommand{\sE}{\mathcal{E}}
\newcommand{\sF}{\mathcal{F}}
\newcommand{\sG}{\mathcal{G}}
\newcommand{\sH}{\mathcal{H}}
\newcommand{\sM}{\mathcal{M}}
\newcommand{\sO}{\mathcal{O}}
\newcommand{\sP}{\mathcal{P}}
\newcommand{\sR}{\mathcal{R}}
\newcommand{\sV}{\mathcal{V}}

\newcommand{\Hom}{\mathrm{Hom}}
\newcommand{\Aut}{\mathrm{Aut}}
\newcommand{\Autext}{\mathrm{Autext}}
\newcommand{\Stab}{\mathrm{Stab}}
\newcommand{\Centr}{\mathrm{Centr}}
\newcommand{\Isom}{\mathrm{Isom}}
\newcommand{\Isomext}{\mathrm{Isomext}}
\newcommand{\Isomint}{\mathrm{Isomint}}
\newcommand{\Norm}{\mathrm{Norm}}
\newcommand{\Transt}{\mathrm{Transt}}
\newcommand{\Trans}{\mathrm{Trans}}
\newcommand{\Par}{\mathrm{Par}}
\newcommand{\Dyn}{\mathrm{Dyn}}
\newcommand{\Of}{\mathrm{Of}}
\newcommand{\diag}{\mathrm{diag}}
\newcommand{\inter}{\mathrm{int}}
\newcommand{\id}{\mathrm{id}}
\newcommand{\ad}{\mathrm{ad}}
\newcommand{\sesi}{\mathrm{ss}}
\newcommand{\der}{\mathrm{der}}
\newcommand{\sico}{\mathrm{sc}}

\newcommand{\rad}{\mathrm{rad}}
\newcommand{\corad}{\mathrm{corad}}
\newcommand{\rk}{\mathrm{rank}}
\newcommand{\Gal}{\mathrm{Gal}}
\newcommand{\spec}{\mathrm{Spec}}
\newcommand{\spin}{\mathrm{Spin}}
\newcommand{\pso}{\mathrm{PSO}}

\newcommand{\Lie}{\mathrm{Lie}}
\newcommand{\lieg}{\mathfrak{g}}
\newcommand{\liet}{\mathfrak{t}}

\newcommand{\gm}{\mathbb{G}}

\newcommand{\fE}{\mathfrak{E}}
\newcommand{\fG}{\mathfrak{G}}
\newcommand{\fH}{\mathfrak{H}}
\newcommand{\fX}{\mathfrak{X}}
\newcommand{\fY}{\mathfrak{Y}}
\newcommand{\fZ}{\mathfrak{Z}}

\newcommand{\bt}{\mathbf{t}}
\newcommand{\bM}{\mathbf{M}}
\newcommand{\bT}{\mathbf{T}}
\newcommand{\bv}{\mathbf{v}}
\newcommand{\bD}{\mathbf{D}}

\newcommand{\rat}{\mathbb{Q}}
\newcommand{\ent}{\mathbb{Z}}
\newcommand{\re}{\mathbb{R}}
\newcommand{\cpx}{\mathbb{C}}
\DeclareFontEncoding{OT2}{}{} 
  \newcommand{\textcyr}[1]{%
    {\fontencoding{OT2}\fontfamily{wncyr}\fontseries{m}\fontshape{n}%
     \selectfont #1}}
\newcommand{\sha}{{\mbox{\textcyr{Sh}}}}

\begin{document}
\tolerance 400 \pretolerance 200 \selectlanguage{english}
\title{Embedding functors and their arithmetic properties}
\date{\today}
\author{}
\maketitle

\begin{center}
{\bf Abstract}
\end{center}
\small{ In this article, we focus on how to embed a torus $\rT$ into
a reductive group $\rG$ with respect to a given root datum $\Psi$
over a scheme $\rS$. This problem is also related to embedding an
\'etale algebra with involution into a central simple algebra with
involution (cf.~\cite{PR1}). We approach this problem by defining
the embedding functor, which is representable and is a left
homogeneous space over $\rS$ under the automorphism group of $\rG$.
In order to fix a connected component of the embedding functor, we
define an orientation $u$ of $\Psi$ with respect to $\rG$. We  show
that the oriented embedding functor is also representable and is a
homogeneous space under the adjoint action of $\rG$. Over a local
field, the orientation $u$ and the Tits index of $\rG$ determine the
existence of an embedding of $\rT$ into $\rG$ with respect to the
given root datum $\Psi$. We also use the techniques developed in
Borovoi's paper~\cite{Bo} to prove that the local-global principle
holds for oriented embedding functors in certain cases. Actually,
the Brauer-Manin obstruction is the only obstruction to the
local-global principle for the oriented embedding functor. Finally,
we apply the results on oriented embedding functors to give an
alternative proof of Prasad and Rapinchuk's Theorem, and to improve
Theorem 7.3 in~\cite{PR1}.

\noindent {\bf Key words:} torus, root datum, reductive group,
central simple algebra, \'etale algebra, local-global principle, \
Tits index

\noindent{\bf   MSC 2010:} 11E57, 14L15, 14L30, 14L35, 20G30. }
\normalsize

\section*{Introduction}
Let $K$ be a field, $\rA$ be a central simple algebra over $K$ with
involution $\tau$, and $\rE$ be an \'etale algebra over $K$ with
involution $\sigma$. Suppose that $\tau|_K=\sigma|_K$. Let $k$ be
the field of invariants $K^{\tau}$, which is a global field.
Motivated by the weak commensurability and length-commensurability
between locally symmetric spaces (\cite{PR2}), Prasad and Rapinchuk
discuss in \cite{PR1} the local-global principle for embeddings of
$(\rE,\sigma)$ into $(\rA,\tau)$ over $K$. This problem is also
related to studying the condition under which  the isomorphism
classes of simple groups are determined by their isomorphism classes
of maximal tori over a number field (cf.~\cite{Ga} and~\cite{PR2}
Thm 7.5).

Motivated by the work of Prasad and Rapinchuk, we consider the
embedding problem of a twisted root datum. Loosely speaking, a
twisted root datum is a torus equipped with some extra data related
to the roots. In this article, we transform such embedding problem
of algebras into a embedding problem of algebraic groups, in order
to work in a more conceptual framework. Moreover, in this framework,
our criteria can be applied not only to the classical groups but
also to the exceptional groups. Instead of  a global field, we work
 over an arbitrary scheme.

Let $\rS$ be a scheme and $\rG$ be a reductive group scheme over
$\rS$. Given an $\rS$-torus $\rT$ and a twisted root datum $\Psi$
associated to $\rT$, we want to know when it is possible to embed
$\rT$ in $\rG$ so that the corresponding twisted root datum
$\Phi(\rG,\rT)$ is isomorphic to $\Psi$.  To approach this problem,
we first define the \emph{embedding functor} $\fE(\rG,\Psi)$.
Roughly speaking, each point of the embedding functor is a closed
immersion $f$ from $\rT$ to $\rG$ such that the twisted root datum
$\Phi(\rG,f(\rT))$ is isomorphic to $\Psi$. For the formal
definition, we refer to Section 1.1. Then our problem can be
reformulated as: when is the set $\fE(\rG,\rT)(\rS)$ nonempty?

We first prove that the embedding functor is a sheaf for the \'etale
topology (in the sense of big \'etale site). To be more precise, the
embedding functor is a homogeneous sheaf under the action of the
automorphism group $\ul{\Aut}_{\rS-gr}(\rG)$ over $\rS$, and it is a
principal homogeneous space under the automorphism group
$\ul{\Aut}(\Psi)$ over the scheme of maximal tori of $\rG$. Then by
the result in~\cite{SGA3} Exp. X, 5.5, we conclude that the
embedding functor $\fE(\rG,\Psi)$ is representable.

However, the embedding functor $\fE(\rG,\Psi)$ can be disconnected
if $\ul{\Aut}_{\rS-gr}(\rG)$ is. Therefore, instead of dealing with
$\fE(\rG,\Psi)$, we fix a particular connected component of
$\fE(\rG,\Psi)$ which will be called an \emph{oriented embedding
functor}. The way we fix a connected component is to fix an
\emph{orientation} of $\Psi$ with respect to $\rG$. An orientation
between semisimple $\rS$-groups was previously defined by Petrov and
Stavrova (\cite{PS}). Here, we generalize it to an orientation
between a twisted root datum $\Psi$ and a reductive $\rS$-group
$\rG$, which is an element $u$ in $\ul{\Isomext}(\Psi,\rG)(\rS)$
(Section. 1.2.1). We show that the oriented embedding functor
$\fE(\rG,\Psi,u)$ is homogeneous under the adjoint action of $\rG$
over $\rS$ and is principal homogeneous under the action of the Weyl
group $\rW(\Psi)$. Hence, $\fE(\rG,\Psi,u)$ is also representable
(ref. \cite{SGA3} Exp. X, 5.5). Moreover, in Theorem~\ref{1.10}, we
show that over a local field $L$, the orientation together with the
Tits index of the given group determine the existence of $L$-points
of the oriented embedding functor.

The main application of embedding functors is to the embedding
problem of Azumaya algebras with involutions. Let $\tilde{R}$ be a
commutative ring, where $2$ is invertible in $\tilde{R}$. Let $\rE$
be an \'etale algebra over $\tilde{R}$ with involution $\sigma$ and
$\rA$ be an Azumaya algebra with involution $\tau$. Suppose
$\sigma|_{\tilde{R}}=\tau|_{\tilde{R}}$. We ask when $(\rE,\sigma)$
can be embedded into $(\rA,\tau)$. The case where $\tilde{R}$ is a
global field is  discussed in Prasad and Rapinchuk's paper
(\cite{PR1}).

Over a commutative ring $\tilde{R}$ in which $2$ is an invertible
element, we let $R$ be the ring of invariants $\tilde{R}^{\tau}$. If
$\tilde{R}$ is equal to $R$, then $\tau$ is said to be of the first
kind. If $\tilde{R}$ is a quadratic extension of $R$, then $\tau$ is
said to be of the second kind.  We consider the reductive group
$\rG=\rU(\rA,\tau)^{\circ}$, the torus $\rT=\rU(\rE,\sigma)^{\circ}$
over $R$, and a twisted root datum $\Psi$ attached to $\rT$. There
is a nice correspondence between the $R$-points of the embedding
functor $\fE(\rG,\Psi)$ and the embeddings
$\iota:(\rE,\sigma)\ra(\rA,\tau)$, namely:
\begin{ntheo}
Keep all the notation defined above.
  The set of $k$-embeddings from $(\rE,\sigma)$ into $(\rA,\tau)$ is in one-to-one correspondence with the set of $R$-points of
  $\fE(\rG,\Psi)$, except for $\rG$ of type $D_4$ or $\rA$ of degree
  2 with $\tau$ orthogonal.
\end{ntheo}
For $\rG$ of type $D_4$, we have a finer treatment and we refer to
Proposition~\ref{1.15}. Moreover, for the involution $\tau$ of the
second kind, we prove that there is an orientation $u$ such that all
$R$-points on the connected component $\fE(\rG,\Psi,u)$ are in
one-to-one correspondence with $R$-embeddings from $(\rE,\sigma)$
into $(\rA,\tau)$ (see Remark~\ref{1.12}, Lemma~\ref{2.11}).

The second part of  this article is devoted to the arithmetic
properties of the embedding functor. In particular, we want to know
if the Hasse principle holds for the existence of $k$-points of the
oriented embedding functor $\fE(\rG,\Psi,u)$, when $k$ is a global
field. Since $\fE(\rG,\Psi,u)$ is a homogeneous space under the
group $\rG$ whose stabilizer is a torus, we use the technique
developed by Borovoi to solve this problem, cf.~\cite{Bo}. Actually,
in \cite{Bo}, Borovoi proved that the Brauer-Manin obstruction to
the Hasse principle is the only obstruction in this case. He also
computed the obstruction using the Galois hypercohomology. We apply
his result to show the following:
\begin{ntheo}
Let $\rG$, $\Psi$ be as above, and $\rT$ be the torus determined by
$\Psi$. Let $u\in\ul{\Isomext}(\Psi,\rG)(k)$ be an orientation.
Suppose that $\Psi$ satisfies one of the following conditions:
\begin{itemize}
  \item [1.] all connected components of $\ul{\Dyn}(\Psi)({k}^{s})$ are  of type $C$, where $k_s$ is a separable closure of $k$.
  \item [2.] $\rT$ is anisotropic  at some place $v\in\Omega_k$.
\end{itemize}
Then the local-global principle holds for the existence of a
$k$-point of the oriented embedding functor $\fE(\rG,\Psi,u)$. In
particular, when $\Psi$ is generic, the local-global principle
holds.
\end{ntheo}

Finally, for the global field $k$ of characteristic different from
2, we combine these techniques and the correspondence established in
Theorem~\ref{1.4} and Proposition~\ref{1.15} to give an alternative
proof of Theorem A and Theorem 6.7 in ~\cite{PR1}. Besides, we
provide an example (\ref{2.13}) to show that the Hasse principle
fails in some cases when the involution $\tau$ is orthogonal and
$\rA$ is $\bM_{2m}(\rD)$, where $\rD$ is a division algebra over
$K$. The main reason for the failure is that the embedding functor
$\fE(\rG,\Psi)$ is disconnected in this case. Let
$\fE(\rG,\Psi)=\rX_1\coprod\rX_2$. Then it may happen that the
embedding functor has a $k_v$ point at each place $v$, but only
$\rX_1$ has a $k_{v_1}$-point, at some place $v_1$, and only $\rX_2$
has a $k_{v_2}$ point, at another place $v_2$. This explains the
failure of the Hasse principle.

\section{Some general facts and notations}
In this section, we briefly recall the notation and definitions
which will be used later. We also state some well-known theorems
which are necessary for the development of the main results about
the embedding functor. Most of the material here can be found
in~\cite{SGA3}, and in the Appendix A of the book by Conrad, Gabber,
and Prasad~\cite{CGP}.

\subsection{Notation and conventions}
Let $\rS$ be a scheme and $\rS'$ an $\rS$-scheme. For an
$\rS$-scheme $\rX$, we let $\rX_{\rS'}$ be the scheme
$\rX\underset{\rS}{\times}\rS'$ over $\rS'$. For a set $\Lambda$, we
let $\Lambda_{\rS}$ denote the disjoint union of the schemes
$\rS_{i}$, where $i\in\Lambda$ and each $\rS_i$ is isomorphic to
$\rS$, i.e. $\Lambda_{\rS}=\underset{i\in\Lambda}{\coprod}\rS_i$. We
call $\Lambda_{\rS}$ the constant scheme over $\rS$ of type
$\Lambda$. (ref. ~\cite{SGA3}, Exp. I, 1.8).

Let $Sch/\rS$ be the category of all $\rS$-schemes. Throughout this
article,  the \'etale site of $\rS$ means the big \'etale site.
Namely, we equip the category $Sch/\rS$ with the following topology:
for an $\rS$-scheme $\rU$, $\{\rU_i\xrightarrow{f_i}\rU\}_i$ is a
covering of $\rU$ if for each $i$, $f_i$ is an \'etale morphism and
$\rU=\cup_i f_i(\rU_i)$. For a detailed introduction to Grothendieck
topology, we refer to the lecture notes by Brochard~\cite{Bro}.

\subsection{Torsors and homogeneous spaces} Let $\rG$ be an
$\rS$-group sheaf for \'etale topology. Let $\sF$ and $\rX$ be
$\rS$-sheaves. Let $p:\sF\ra\rX$ be a morphism between
$\rS$-sheaves. Then $\sF$ is called a right (resp. left) $\rG$-sheaf
over $\rX$ with respect to $p$ if $\sF$ is equipped with an
$\rG$-action satisfying $p(fg)=p(f)$ (resp. $p(gf)=p(f)$ ) for all
$(f,g)\in(\sF{\times}\rG)(\rS')$ and for all $\rS$-schemes $\rS'$.
Note that $\rX{\times}\rG$ can be equipped with a right $\rG$-action
as $(x,g)\sigma=(x,g\sigma)$. Let $p_{\rX}$ be the projection from
$\rX{\times}\rG$ to $\rX$. Then $\rX{\times}\rG$ is a right
$\rG$-sheaf over $\rX$ with respect to $p_{\rX}$. A $\rG$-sheaf
$\sF$ over $\rX$ with respect to $p$ is \emph{trivial} if it is
isomorphic to $\rX{\times}\rG$ with respect to $p_{\rX}$ as
$\rG$-sheaves over $\rX$. A right $\rG$-sheaf $\sF$ over $\rX$ with
respect to $p$ is called a \emph{$\rG$-torsor} over $\rX$ if there
is an epimorphism of $\rS$-sheaves $\pi:\rY\ra\rX$ such that
$\rY\underset{\rX}{\times}\sF$ over $\rY$ is a trivial $\rG$-sheaf.
\begin{prop}\label{0:1}
Let $\mathrm{G}$ be an $\rS$-group sheaf, $\rX$ be a sheaf over
$\rS$. Let $\sF$ be a
$\rG$-sheaf over $\rX$ with respect to an $\rS$-sheaf morphism
$p:\sF\ra\rX$. Then $\sF$ is a torsor over $\rX$ if and only if $p$
is an epimorphism of $\rS$-sheaves and the morphism
$i:\sF{\times}\rG\ra\sF\underset{\rX}{\times}\sF$ defined as
$i(x,h)=(x,xh)$ is invertible.
\end{prop}
\begin{proof}
~\cite{DG}, Chap. III, \S 4, Corollary 1.7.
\end{proof}

Let $\rG$ be an $\rS$-group sheaf. Let $\sF$ and $\rX$ be
$\rS$-sheaves. Let $p:\sF\ra\rX$ be a morphism between
$\rS$-sheaves. A $\rG$-sheaf $\sF$ over $\rX$ with respect to $p$ is
called a \emph{$\rG$-homogeneous space} if $p$ is an epimorphism of
$\rS$-sheaves and the morphism
$i:\sF{\times}\rG\ra\sF\underset{\rX}{\times}\sF$ defined as
$i(x,h)=(x,xh)$ is an epimorphism between sheaves .

\begin{subsection}{Root data and twisted root datum}

Let $\psi=(\rM,\rM^{\vee},\rR,\rR^{\vee})$ be a root datum (ref.
\cite{SGA3}, Exp. XXI, 1.1.1).  Let $\Delta\subseteq\rR$ be a system
of simple roots of $\rR$. The root datum $\psi$ plus a system of
simple roots $\Delta$ of $\rR$ is called a \emph{pinning root
datum}, and we denote it as
$(\rM,\rM^{\vee},\rR,\rR^{\vee},\Delta)$.

The subgroup of the automorphism group of $\rM$ generated by the
reflections  $\{s_\alpha\}_{\alpha\in\rR}$ is called the Weyl group
of $\psi$, and we denote it as $\rW(\psi)$.

For a finite subset $\rR$ (resp. $\rR^{\vee}$)of $\rM$ (resp.
$\rM^{\vee}$), we let $\Gamma _0(\rR)$ (resp. $\Gamma
_0(\rR^{\vee})$) be the subgroup generated by $\rR$ (resp.
$\rR^{\vee}$) and we let $\sV(\rR)$ (resp. $\sV(\rR^{\vee}$) be the
vector space defined by $\Gamma_0(\rR)\otimes_\ent\rat$ (resp.
$\Gamma_0(\rR^{\vee})\otimes_\ent\rat$).

A root datum is called \emph{reduced} if for all $\alpha\in\rR$, we
have $2\alpha\notin\rR$. A root datum is called \emph{semisimple} if
$\rk(\Gamma_0(\rR))=\rk(\rM)$. A root datum
$(\rM,\rM^{\vee},\rR,\rR^{\vee})$ is called \emph{adjoint} (resp.
\emph{simply connected}) if $\rM=\Gamma_0(\rR)$ (resp.
$\rM^{\vee}=\Gamma_0(\rR^{\vee})$).

We define the dual root datum of $\psi$ to be
$(\rM^{\vee},\rM,\rR^{\vee},\rR)$, and denote it as $\psi^{\vee}$.

\begin{subsubsection}{Radical and coradical of root data}
Let
$$\rN=\{x\in\rM|\ \alpha^{\vee}(x)=0,\ \forall \alpha^{\vee}\in\rR^{\vee}\}.$$
Then the dual of $\rN$ can be identified with
$\rM^{\vee}/\sV(\rR^{\vee})\cap\rM^{\vee}$ (ref.~\cite{SGA3},
Exp.XXI, 6.3.1).

Define the coradical of $\psi$ to be the root datum
$(\rN,\rN^{\vee},\emptyset,\emptyset)$ and denote it as
$\corad(\psi)$. We define the radical of $\psi$ to be
$\corad(\psi^{\vee})^{\vee}$, and denote it as $\rad(\psi)$.
\end{subsubsection}
\begin{subsubsection}{Induced and coinduced root data}
Given a root datum $\psi=(\rM,\rM^{\vee},\rR,\rR^{\vee})$, and a
subgroup $\rN$ of $\rM$ which contains $\Gamma_0(\rR)$, let
$i_\rN:\rN\ra\rM$ be the natural inclusion, and
$i_\rN^{\vee}:\rM^{\vee}\ra\rN^{\vee}$ be the corresponding map on
$\rM^{\vee}$. Let $\rR_\rN=\rR$ and
$\rR_\rN^{\vee}=i_\rN^{\vee}(\rR^{\vee})$. We define the root datum
$\psi_{\rN}$ as $(\rN,\rN^{\vee},\rR_\rN,\rR_\rN^{\vee})$, which is
called the \emph{induced root datum} of $\psi$ respect to $\rN$. If
$\rN=\Gamma_0(\rR)$, then $\psi_\rN$ is an adjoint root datum, and
we denote it as $\ad(\psi)$. If $\rN=\sV(\rR)\bigcap\rM$, then
$\psi_\rN$ is a semisimple root datum, and we denote $\psi_\rN$ as
$\sesi(\psi)$. We let $\der(\psi)=\sesi(\psi^{\vee})^{\vee}$, and
$\sico(\psi)=\ad(\psi^{\vee})^{\vee}$.
\end{subsubsection}
\begin{subsubsection}{Morphisms between root data}
Let $\psi_1=(\rM_1,\rM_1^{\vee},\rR_1,\rR_1^{\vee})$,
$\psi_2=(\rM_2,\rM_2^{\vee},\rR_2,\rR_2^{\vee})$ be two root data. A
module morphism $f:\rM_1\ra\rM_2$ is a \emph{morphism between
$\psi_1$ and $\psi_2$} if $f$ induces a bijection between $\rR_1$
and $\rR_2$ and the transpose map
$^{t}f:\rM_2^{\vee}\ra\rM_1^{\vee}$ is a bijection between
$\rR_2^{\vee}$ and $\rR_1^{\vee}$.
\begin{prop}\label{0.10}
Keep all the notations above. If $f:\rM_1\ra\rM_2$ is a morphism
between $\psi_1$ and $\psi_2$, then
$^{t}f(f(\alpha)^{\vee})=\alpha^{\vee}$, and the map $s_\alpha\ra
s_{f(\alpha)}$ for $\alpha\in\rR_1$ extends to an isomorphism
between $\rW(\psi_1)$ and $\rW(\psi_2)$,
\end{prop}
\begin{proof}
~\cite{SGA3}, Exp. XXI, 6.1.1 and 6.2.2.
\end{proof}

Let $\Aut(\psi)$ be the automorphism group of $\psi$, and fix a
system of simple roots $\Delta$ of $\rR$. Define the abstract group
$$\rE_\Delta(\psi)=\{u\in\Aut(\psi)|\ u(\Delta)=\Delta\}.$$ Then we
have the following:
\begin{prop}\label{0.5}
$\rW(\psi)$ is a normal subgroup of $\Aut({\psi})$, and $\Aut(\psi)$
is a semi-direct product of $\rW(\psi)$ by $\rE_\Delta(\psi)$.
\end{prop}
\begin{proof}
~\cite{SGA3}, Exp. XXI, 6.7.1 and 6.7.2.
\end{proof}
\end{subsubsection}
\end{subsection}
\begin{subsubsection}{Twisted root data}
Let $\rT$ be an $\rS$-torus.  Let $\sM$ be the \emph{character group
scheme} associated to $\rT$, i. e.
$\sM(\rS')=\Hom_{\rS'-gr}(\rT_{\rS'}, \gm_{m,\ \rS'})$. Let
$\Psi=(\sM,\sM^{\vee},\sR,\sR^{\vee})$ be a twisted root datum
associated to $\rT$ (ref.~\cite{SGA3}, Exp. 22, Def. 1.9).

The root datum $\Psi$ is \emph{split} if $\rT$ is split. A twisted
root datum is called \emph{reduced} if for all $\rS$-schemes $\rS'$
and all $\alpha\in\sR(\rS')$, we have $2\alpha\notin\sR(\rS')$.

Let $\psi=(\rM,\rM^{\vee},\rR,\rR^{\vee})$ be a root datum. A
twisted root datum $\Psi$ is said to have \emph{type} $\psi$ at the point $s$
of $\rS$ if
$\Psi_{\ol{s}}\simeq(\rM_{\ol{s}},\rM^{\vee}_{\ol{s}},\rR_{\ol{s}},\rR^{\vee}_{\ol{s}})$.

Let $\Psi=(\sM,\sM^{\vee},\sR,\sR^{\vee})$ be a twisted root datum.
Since at each $s\in\rS$, there is an \'{e}tale neighborhood such
that $\rT$ splits, we can define $\ad(\Psi)$, $\sico(\Psi)$,
$\sesi(\Psi)$, $\der(\Psi)$ \'{e}tale locally, and by the
functoriality  of induced root data, define them over $\rS$ by
descent (~\cite{SGA3}, Exp. XXI, 6.5).

Let $\Psi_1=(\sM_1,\sM_1^{\vee},\sR_1,\sR_1^{\vee})$,
$\Psi_2=(\sM_2,\sM_2^{\vee},\sR_2,\sR_2^{\vee})$ be two twisted root
data. Let $\rT_1$, $\rT_2$ be the tori determined by $\Psi_1$ and
$\Psi_2$ respectively. An $\rS$-group morphism $f:\rT_2\ra\rT_1$ is
a \emph{morphism from $\Psi_1$ to $\Psi_2$} if $f$ induces an
isomorphism from $\sR_1$ to $\sR_2$ and an isomorphism from
$\sR_2^{\vee}$ to $\sR_1^{\vee}$. We can also define the induced
twisted root data by \'etale descent, and define $\ad(\Psi)$,
$\sesi(\Psi)$, $\der(\Psi)$ and $\sico(\Psi)$ as we have done for
the root data.

\end{subsubsection}
\begin{subsubsection}{Weyl groups, Isom, Isomext and Isomint for twisted root data}
Let $\Psi$ be a twisted root datum. Suppose $\Psi$ is split and
$\Psi=(\rM_\rS,\rM^{\vee}_\rS,\rR_\rS,\rR^{\vee}_\rS)$. Let $\psi$
be the root datum $(\rM,\rM^{\vee},\rR,\rR^{\vee})$. Let $\rW$ be
the Weyl group of $\psi$, and define $\rW(\Psi)=\rW_\rS$. Suppose that
$\Psi$ is not split. Then we can find an \'etale covering
$\{\rS_i\ra\rS\}$ such that $\Psi_{\rS_i}$ is split. By
Proposition~\ref{0.4}, the canonical isomorphism between
$(\Psi_{\rS_i})_{\rS_j}$ and $(\Psi_{\rS_j})_{\rS_i}$ gives a
canonical isomorphism between $\rW(\Psi_{\rS_i})_{\rS_j}$ and
$\rW(\Psi_{\rS_j})_{\rS_i}$ and hence gives descent data for
$\{\rW(\Psi_{\rS_i})\}_i$, which allows us to define $\rW(\Psi)$.

Let $\ul{\Aut}(\Psi)$ be the automorphism functor of $\Psi$. By
descent, we can define the Weyl group $\rW(\Psi)$. Then from
Proposition~\ref{0.5}, we can define the following exact sequence by
\'etale descent: $$\xymatrix@C=0.5cm{
  1 \ar[r] & \rW(\Psi) \ar[r] &  \ul{\Aut}(\Psi)\ar[r] & \ul{\Autext}(\Psi)
  \ar[r] & 1 }$$
Then we have the following proposition:
\begin{prop}\label{0.6}
Keep all the notations above. The automorphism group
$\ul{\Aut}(\Psi)$ is representable by a twisted constant
$\rS$-scheme, and $\rW(\Psi)$ is normal in $\ul{\Aut}(\Psi)$. Let
$\ul{\Autext}(\Psi)$ be the quotient group of $\ul{\Aut}(\Psi)$ by
$\rW(\Psi)$. Then $\ul{\Autext}(\Psi)$ is also representable by a
twisted constant $\rS$-scheme.
\end{prop}
\begin{proof}
Note that we can find an \'etale covering $\{\rS_i\ra\rS\}_i$ such
that $\Psi_{\rS_i}$ is split (~\cite{SGA3}, Exp. X, 4.5). By
Proposition~\ref{0.5} , $\ul{\Aut}(\Psi_{\rS_i})$ and
$\ul{\Autext}(\Psi_{\rS_i})$ are constant group schemes over
$\rS_i$.  By~\cite{SGA3}, Exp. X, 5.5, $\{\rS_i\ra\rS\}_i$ gives an
effective descent datum, so $\ul{\Aut}(\Psi_{\rS_i})$ and
$\ul{\Autext}(\Psi_{\rS_i})$ are representable.
\end{proof}

Let $\Psi_1$, $\Psi_2$ be two twisted root data. Suppose $\Psi_2$ is
a twisted form of $\Psi_1$. Let $\ul{\Isom}(\Psi_1,\Psi_2)$ be the
isomorphism functor between $\Psi_1$ and $\Psi_2$. Then
$\ul{\Isom}(\Psi_1,\Psi_2)$ is a right principal homogeneous space
of $\ul{\Aut}(\Psi_1)$ and a left principle homogeneous of
$\ul{\Aut}(\Psi_2)$. Since $\ul{\Aut}(\Psi_2)$ is representable,
$\ul{\Isom}(\Psi_1,\Psi_2)$ is also representable.

Define
$\ul{\Isomext}(\Psi_1,\Psi_2)=\rW(\Psi_2)\setminus\ul{\Isom}(\Psi_1,\Psi_2)$.

Note that for $f\in\ul{\Isom}(\Psi_1,\Psi_2)(\rS)$,
$f^{-1}\circ\rW(\Psi_2)\circ f=\rW(\Psi_1)$ by
Proposition~\ref{0.10}. Therefore we have a natural isomorphism from
$\ul{\Isomext}(\Psi_1,\Psi_2)$ to
$\ul{\Isom}(\Psi_1,\Psi_2)/\rW(\Psi_1)$. Then
$\ul{\Isomext}(\Psi_1,\Psi_2)$ is a left
$\ul{\Autext}(\Psi_2)$-principal homogeneous space and a right
$\ul{\Autext}(\Psi_1)$-principal homogeneous space. An
\emph{orientation} of $\Psi_1$ with respect to $\Psi_2$ is an
$\rS$-point of $\ul{\Isomext}(\Psi_1,\Psi_2)$.

Suppose that there is $u\in\ul{\Isomext}(\Psi_1,\Psi_2)(\rS)$. Then
we can regard $\rS$ as an $\ul{\Isomext}(\Psi_1,\Psi_2)$-scheme
through $u$ and define
$$\ul{\Isomint}_u(\Psi_1,\Psi_2):=\rS\underset{\ul{\Isomext}(\Psi_1,\Psi_2)}{\times}\ul{\Isom}(\Psi_1,\Psi_2).$$
\end{subsubsection}


\begin{subsection}{Reductive groups}\label{s0.3}
An $\rS$-group scheme $\rG$ is called \emph{reductive (resp.
semi-simple)} if it is affine and smooth over $\rS$, and all the
geometrical fibers are connected and reductive (resp. semisimple)
(ref.~\cite{SGA3}, Exp. XIX, Def. 2.7).

Let $\rG$ be a reductive $\rS$-group scheme and suppose that $\rT$
is a maximal torus in $\rG$.  We let $\Phi(\rG,\rT)$ be the twisted
root datum of $\rG$ with respect to $\rT$ (ref.~\cite{SGA3} Exp.
XXII, 1.10).

For a point $s\in\rS$, let $\kappa(s)$ be the residue field of $s$
and $\ol{\kappa(s)}$ be the algebraic closure of $\kappa(s)$. Let
$\ol{s}$ be the scheme $\spec(\ol{\kappa(s)})$. The \emph{type} of
$\rG$ at $s$ is the type of $\Phi(\rG_{\ol{s}},\rT_{0})$, where
$\rT_{0}$ is a maximal torus of $\rG_{\ol{s}}$ (ref.~\cite{SGA3},
Exp. 22, Def. 2.6.1, 2.7). Note that the type of $\rG$ is locally
constant over $\rS$ (ref~\cite{SGA3}, Exp. 22, Prop. 2.8).

A reductive $\rS$-group $\rG$ is \emph{split} if there is a maximal
torus $\rT$ of $\rG$ and a root datum
$(\rM,\rM^{\vee},\rR,\rR^{\vee})$ such that
$\Phi(\rG,\rT)\simeq(\rM_{\rS},\rM^{\vee}_{\rS},\rR_{\rS},\rR^{\vee}_{\rS})$
and satisfying the following:
\begin{itemize}
\item[1]
$\rS$ is nonempty and each root $\alpha\in\rS$ (resp.
$\alpha^{\vee}\in\rR^{\vee}$) can be identified as a constant map
from $\rS$ to $\rM$ (resp. $\rM^{\vee}$).
\item[2]
Let $\lieg=\Lie(\rG/\rS)$ and $\liet=\Lie(\rT/\rS)$. Under the
adjoint action of $\rT$,
$\lieg=\liet\oplus\underset{\alpha\in\rR}{\coprod}\lieg^{\alpha}$,
where the $\lieg^\alpha$'s are free $\sO_\rS$-modules.
\end{itemize}
In this case, we say that $\rG$ is \emph{split relatively to $\rT$}
(ref.~\cite{SGA3}, Exp. XXII, 1.13 and 2.7).

Let us endow $\rS$ with the \'{e}tale topology. Let $\rS'$ be an
$\rS$-scheme, and $\rG$ be an $\rS$-group scheme. Let
$\underline{\Aut}_{\rS-gr}(\rG)$ be the sheaf of group automorphisms
of $\rG$. Then we can define the group homomorphism
\begin{center}
$\ad:\rG\ra\underline{\Aut}_{\rS-gr}(\rG)$
\end{center}
which maps an element $g$ of $\rG(\rS')$ to an automorphism of
$\rG_{\rS'}$ defined by the conjugation by $g$. Let
$\underline{\Centr}(\rG)$ be the center of $\rG$. Then the image
sheaf of $\ad$ is isomorphic to $\rG / \underline{\Centr}(\rG)$ and
$\ad(\rG)$ is normal in $\underline{\Aut}_{\rS-gr}(\rG)$. So we have
the exact sequence of $\rS$-group sheaves:
\begin{center}
$1\ra\ad(\rG)\ra\underline{\Aut}_{\rS-gr}(\rG)\ra\underline{\Autext}(\rG)\ra
1$.
\end{center}
\begin{theo}\label{0.3}
Let $\rS$ be a scheme and $\rG$ be a reductive $\rS$-group scheme.
For the exact sequence of $\rS$-sheaves:
\begin{center}
$\xymatrix{
1\ar[r]&\ad(\rG)\ar[r]&\underline{\Aut}_{\rS-gr}(\rG)\ar[r]^{p}&\underline{\Autext}(\rG)\ar[r]&
1}$,
\end{center}
we have the following:
\begin{itemize}
\item[(i)]
$\underline{\Aut}_{\rS-gr}(\rG)$ is represented by a separated,
smooth $\rS$-scheme.
\item[(ii)]
$\underline{\Autext}(\rG)$ is represented by a twisted finitely
generated constant scheme.
\item[(iii)]
Suppose that $\rG$ splits relatively to $\rT$, and
$\Phi(\rG,\rT)\simeq(\rM_{\rS},\rM^{\vee}_{\rS},\rR_{\rS},\rR^{\vee}_{\rS})$.
Let $(\psi,\Delta)=(\rM,\rM^{\vee},\rR,\rR^{\vee},\Delta)$ be a
pinning root datum. Then there is a monomorphism between sheaves
$a:\rE_{\Delta}(\psi)_{\rS}\ra\ul{\Aut}_{\rS-gr}(\rG)$ such that
 $$p\circ a:\rE_{\Delta}(\psi)_{\rS}\ra\underline{\Autext}(\rG)$$
 is an isomorphism.
\end{itemize}
\end{theo}
\begin{proof}
~\cite{SGA3}, Exp. XXIV, Theorem 1.3.
\end{proof}
For a subgroup scheme $\rH$ of $\rG$, we let
$\ul{\Aut}_{\rS-gr}(\rG,\rH)$ be the subsheaf of
$\ul{\Aut}_{\rS-gr}(\rG)$ which normalizes $\rH$, i. e.
$\ul{\Aut}_{\rS-gr}(\rG,\rH)=\ul{\Norm}_{\ul{\Aut}_{\rS-gr}(\rG)}(\rH)$
(cf.~\cite{SGA3} Exp. $\mathrm{VI_{B}}$, Def. 6.1 (iii)). We let
$\ul{\Autext}(\rG,\rH)$ be the quotient sheaf
$\ul{\Aut}_{\rS-gr}(\rG,\rH)/\ul{\Norm}_{\ad(\rG)}(\rH)$.
\end{subsection}
\begin{subsection}{Dynkin diagrams}

For each reductive $\rS$-group $\rG$, we can associate a
\emph{Dynkin diagram scheme} $\underline{\Dyn}(\rG)$ to $\rG$
(ref.~\cite{SGA3}, Exp. XXIV, 3.2 and 3.3). Moreover we have the
following:
\begin{prop}\label{0.4}
If $\rG$ is semisimple (resp. adjoint or simply connected), then the
morphism
$$\underline{\Autext}(\rG)\ra\underline{\Aut}_{\Dyn}(\underline{\Dyn}(\rG))$$
is a monomorphism (resp. isomorphism).
\end{prop}
\begin{proof}
~\cite{SGA3}, Exp. XXIV, 3.6.
\end{proof}

Given a twisted root datum $\Psi$ over $\rS$,  we can also define
the Dynkin scheme of $\Psi$ in a similar way and denote it by
$\underline{\Dyn}(\Psi)$. We also have a natural morphism from
$\ul{\Autext}(\Psi)$ to
$\underline{\Aut}_{\Dyn}(\underline{\Dyn}(\Psi))$, which will be a
monomorphism (resp. isomorphism) if $\Psi$ is reduced semisimple
(resp. reduced adjoint or reduced simply connected).

For a root datum $\psi$, we can associate to each connected
component of its Dynkin diagram $\Dyn(\psi)$ a type according to the
classification of Dynkin diagrams (ref.~\cite{SGA3} Exp. XXI,
7.4.6). Let $\bT$ be the set of all types of Dynkin diagram.
Similarly, for each Dynkin scheme $\bD$ over $\rS$, we can associate
the scheme of connected components $\bD_0$ to $\rD$
(ref.~\cite{SGA3} Exp. XXIV, 5.2). We can also define a morphism
$$a:\bD_0\ra\bT_{\rS}.$$ Let $\bv\in\bT$. If $\bD_0=a^{-1}(\bv)$,
then we say $\bD$ is \emph{isotypical of type $\bv$}. If the Dynkin
scheme $\ul{\Dyn}(\Psi)$ is connected at each fiber over $\rS$ and
is of constant type $\bv$, then we say that $\Psi$ is \emph{simple}
of type $\bv$.
\end{subsection}
\subsection{Parabolic subgroups}
Let $\rS$ be a scheme and $\rG$ be a reductive $\rS$-group. A
subgroup scheme $\rP$ of $\rG$ over $\rS$ is called parabolic if
\begin{itemize}
  \item [1] $\rP$ is smooth over $\rS$.
  \item [2] For each $s\in\rS$, the quotient
$\rG_{\ol{s}}/\rP_{\ol{s}}$ is proper.
\end{itemize}

Let us keep the notations in Section~\ref{s0.3}. Let $\sE=\{\rG,\
\rT,\ \rR,\ \Delta,\ \{\rX_\alpha\}_{\alpha\in\Delta}\}$ be a
pinning of $\rG$ and $\rP$ be a parabolic subgroup. The pinning
$\rE$ is said to be adapted to $\rP$ if $\rP$ contains $\rT$ and
$\Lie(\rP/\rS)=\liet\oplus\coprod_{\alpha\in\rR'}\lieg^\alpha$,
where $\rR'$ is a subset of $\rR$ which contains all the positive
roots. In this case, we denote $\Delta(\rP)=\Delta\cap -\rR'$.

Let $\Of(\ul{\Dyn}(\rG))$ be the functor defined as the following:
for each $\rS$-scheme $\rS'$, $\Of(\ul{\Dyn}(\rG))(\rS')$ is the set
of all subschemes of $\ul{\Dyn}(\rG)_{\rS'}$ which are open and closed.
Then $\Of(\ul{\Dyn}(\rG))$ is a  twisted finite constant scheme. Let
$\ul{\Par}(\rG)$ be the functor defined by $\ul{\Par}(\rG)(\rS')$ is the
set of all parabolic subgroups of $\rG_\rS'$, for each $\rS$-scheme
$\rS'$. One can define a morphism
$$\bt:\ul{\Par}(\rG)\ra\Of(\ul{\Dyn}(\rG))$$
satisfying
\begin{itemize}
  \item [1] $\bt$ is functorial in $\rG$.
  \item [2] If $\sE$ is a pinning of $\rG$ adapted to the parabolic subgroup
$\rP$, then $\bt(\rP)=\Delta(\rP)_{\rS}$.
\end{itemize}

For a parabolic subgroup $\rP$ of $\rG$, we call $\bt(\rP)$ the type
of $\rP$.

\begin{prop}\label{0.9}
Let $\rS$, $\rG$ be as above. Let $\rP$ be a parabolic subgroup of
$\rG$. Let $t'$ be a section of $\Of(\ul{\Dyn}(\rG))$ over $\rS$ and
$t'\supseteq\bt(\rP)$. Then there is a unique parabolic subgroup
$\rP'$ of $\rG$ which contains $\rP$ and the type of $\rP'$ is $t'$.
\end{prop}
\begin{proof}
~\cite{SGA3} Exp. XXVI, Lemme 3.8.
\end{proof}

\begin{section}{Embedding functors}
Let $\rS$ be a scheme and $\rG$ be a reductive group over $\rS$. Let
$\rT$ be an $\rS$-torus and $\Psi$ be a root datum associated to
$\rT$. We would like to know if we can embed $\rT$ in $\rG$ as a
maximal torus such that the twisted root datum $\Phi(\rG,\rT)$ is
isomorphic to $\Psi$. To answer this question, we first define the
embedding functor $\fE(\rG,\Psi)$. The embedding functor is
representable and is a left $\ul{\Aut}_{\rS-gr}(\rG)$-homogeneous
space. Briefly speaking, each $\rS$-point of $\rE(\rG,\Psi)$
corresponds to an embedding from $\rT$ to $\rG$ with respect to
$\Psi$.

In the second part, we first define an orientation $v$ of $\Psi$
with respect to $\rG$. Once we can fix an orientation, we can fix a
connected component of $\rE(\rG,\Psi)$, which is called an oriented
embedding functor.  The oriented embedding functor $\fE(\rG,\Psi,v)$
is also representable and is a left $\rG$-homogeneous space.

In the end of this section, we show that the embedding functor has
an interpretation in the embedding problem of Azumaya algebras with
involution. Moreover, we show that there is a one-to-one
correspondence between the $k$-points of the embedding functor and
the $k$-embeddings from an \'etale $k$-algebra with involution into an Azumaya
algebra with involution.
\begin{subsection}{Embedding functors}
Let $\rS$ be a scheme, $\rG$ be a reductive $\rS$-group scheme. Let
$\rT$ be an $\rS$-torus. Let $\sM$ be the character group scheme
associated to $\rT$, and $\Psi=(\sM,\sM^{\vee},\sR,\sR^{\vee})$ be a
root datum associated to $\rT$. We define the \emph{embedding
functor} by:  \[ \fE(\rG,\Psi)(\rS')=\left\{\begin{array}{l}
\mbox{$f:\rT_{\rS'}\hookrightarrow\rG_{\rS'}$}\left|\begin{array}{l}\mbox{$
f$ is both a closed
immersion and a group }\\
 \mbox{homomorphism which induces an isomorphism
}\\
\mbox{$f^{\Psi}:\Psi_{\rS'}\xrightarrow{\sim}\Phi(\rG_{\rS'},f(\rT_{\rS'}))$
such that }\\
\mbox{ $f^{\Psi}(\alpha)=\alpha\circ f^{-1}|_{f(\rT_{\rS'})}$ for
all $\alpha\in\sM(\rS'')$,}\\\mbox{for each $\rS'$-scheme
$\rS''$.}\end{array}\right.\end{array}\right\}\]
 for
$\rS'$ a scheme over $\rS$. In this article, we always assume that
at each geometric point $\ol{s}\in\rS$, the root datum
$\Psi_{\ol{s}}$ is isomorphic to the root datum of $\rG_{\ol{s}}$.
Therefore, $\fE(\rG,\Psi)$ is not empty in our case.

The embedding functor $\fE(\rG,\Psi)$ is naturally equipped with a
left $\underline{\mathrm{Aut}}_{\rS-gr}(\rG)$-action defined as
compositions of functions. Namely define
$$l:\underline{\mathrm{Aut}}_{\rS-gr}(\rG)\times\fE(\rG,\Psi)\ra\fE(\rG,\Psi)$$ as
$l(\sigma,f)=\sigma\circ f$ for all
$\sigma\in\underline{\mathrm{Aut}}_{\rS-gr}(\rG)(\rS')$,
$f\in\fE(\rG,\Psi)(\rS')$ and $\rS'$ an $\rS$-scheme.

Since $\underline{\Aut}(\Psi)\subseteq\underline{\mathrm{Aut}}_{\rS-gr}(\sM)$
and
$\underline{\mathrm{Aut}}_{\rS-gr}(\rT)=\underline{\mathrm{Aut}}_{\rS-gr}(\sM)^{op}$,
we can regard $\ul{\Aut}(\Psi)$ as a subgroup of
$\ul{\Aut}_{\rS-gr}(\rT)$ through the inverse map between
$\underline{\mathrm{Aut}}_{\rS-gr}(\sM)$ and
$\underline{\mathrm{Aut}}_{\rS-gr}(\sM)^{op}$. We define a
right $\ul{\Aut}(\Psi)$-action on $\fE(\rG,\Psi)$ as a composition
of an automorphism of $\rT$ followed by a closed embedding from
$\fE(\rG,\Psi)$.

Now, let $\mathcal{T}$ be the scheme of maximal tori of $\rG$
(cf.~\cite{SGA3} XII, 1.10). We think about the morphism
$\pi:\fE(\rG,\Psi)\ra\mathcal{T}$ defined as $\pi(f)=f(\rT_{\rS'})$,
where $f\in\fE(\rG,\Psi)(\rS')$, and $\rS'$ is a scheme over $\rS$.
Then we have the following:
\begin{theo}\label{1:1}
In the sense of the \'{e}tale topology, $\fE(\rG,\Psi)$ is a homogeneous
space over $\rS$ under the left
$\underline{\mathrm{Aut}}_{\rS-gr}(\rG)$-action, and a torsor over
$\mathcal{T}$ under the right $\ul{\Aut}(\Psi)$-action. Moreover,
$\fE(\rG,\Psi)$ is representable by an $\rS$-scheme.
\end{theo}
\begin{proof}
We divide the argument into the following three parts:
\begin{claim}
$\fE(\rG,\Psi)$ is a sheaf for the \'{e}tale topology.
\end{claim}
\begin{proof}
Let $\{\rS_i\ra\rS\}$ be an \'etale covering. Since $\fE(\rG,\Psi)$
is a subfunctor of
$\underline{\mathrm{Hom}}_{\rS-\mathrm{gr}}(\rT,\rG)$ and
$\underline{\mathrm{Hom}}_{\rS-\mathrm{gr}}(\rT,\rG)$ is a sheaf, we
only need to prove that for
$f\in\underline{\mathrm{Hom}}_{\rS-\mathrm{gr}}(\rT,\rG)(\rS)$ if
$f_{\rS_i}\in\fE(\rG,\Psi)(\rS_i)$, then $f\in\fE(\rG,\Psi)(\rS)$.

We note that to verify that $f$ is a closed immersion, it is enough
to verify it \'etale locally. Hence $f$ is a closed immersion since
$f_{\rS_i}$ is. Since $f(\rT)$ is locally a maximal torus in $\rG$,
$f(\rT)$ is a maximal torus. Finally, $f^{\Psi}$ is an isomorphism
\'etale locally, so $f^{\Psi}$ is an isomorphism. We conclude that
$\fE(\rG,\Psi)$ is a sheaf.
\end{proof}

\begin{claim}
$\fE(\rG,\Psi)$ is homogeneous under the left
$\underline{\mathrm{Aut}}_{\rS-gr}(\rG)$-action, which is defined as
composition of functions.
\end{claim}
\begin{proof}
Let $\rS'$ be a scheme over $\rS$, and $f_{1}$, $f_{2}$ be two
elements in $\fE(\rG,\Psi)(\rS')$. Let $\rF_{i}=f_i(\rT_{\rS'})$,
$i=$1, 2 respectively. Then there exists an \'{e}tale neighborhood
$\rU$ of $\rS'$ where $\rF_1$ and $\rF_2$ are conjugated
(ref.~\cite{SGA3} Exp. 12, Theorem 1.7), so we can assume
$\rF_{1,\rU}=\rF_{2,\rU}$. Moreover, we can even assume $\rG_{\rU}$
is split relatively to $\rF_{1,\rU}$ (ref.~\cite{SGA3} Exp. 22,
2.3). By abuse of notation, we still use $f_{2}\circ
f_{1}^{-1}$ to denote the morphism from $\rF_{1,\rU}$ to
$\rF_{2,\rU}$. Then by the definition of the
$\fE(\rG,\Psi)$-functor, we know that $f_{2}\circ f_{1}^{-1}$
induces an automorphism on $\Phi(\rG_{\rU},f(\rF_{1,\rU}))$.
According to Theorem~\ref{0.3}, we can find $\sigma$, which is an
automorphism of $\rG_{\rU}$, such that $\sigma\circ f_2=f_1$, which
proves the claim.
\end{proof}

\begin{claim}
$\fE(\rG,\Psi)$ is a right $\ul{\Aut}(\Psi)$-torsor over
$\mathcal{T}$ for \'{e}tale topology.
\end{claim}
\begin{proof}
We first prove that $\pi:\fE(\rG,\Psi)\ra\mathcal{T}$ is surjective
as an $\rS$-sheaf morphism for the \'{e}tale topology. For an
$\rS$-scheme $\rS'$ and an element $\rF$ in $\mathcal{T}(\rS')$,
which means that $\rF$ is a maximal torus in $\rG_{\rS'}$, for each
$s'\in\rS'$, we can find an \'{e}tale open neighborhood
$\rU'\ra\rS'$ such that $\Psi_{\mathrm{U}'}$ splits and $\rG_{\rU'}$
splits relatively to $\rF_{\rU'}$. Therefore, $\Psi_{\mathrm{U}'}$
and $\Phi(\rG,\rF)_{\mathrm{U}'}$ are isomorphic as we assume that
both of them are with the same type at each geometric point. Hence,
there is $f\in\fE(\rG,\Psi)(\mathrm{U}')$ such that
$\pi(f)=\rF\underset{\mathrm{S}'}{\times}\mathrm{U}'$.

Next, let us show that
$\fE(\rG,\Psi)\underset{\rS}{\times}\ul{\Aut}(\Psi)\simeq\fE(\rG,\Psi)\underset{\mathcal{T}}{\times}\fE(\rG,\Psi)$
as $\ul{\Aut}(\Psi)$-space. By identifying $\ul{\Aut}(\Psi)$ with a
subgroup of $\ul{\Aut}(\rT)$, we regard
$\sigma\in\ul{\Aut}(\Psi)(\rS')$ as an element of
$\ul{\Aut}_{\rS-gr}(\rT)(\rS')$. Define
$$m:\fE(\rG,\Psi)\underset{\rS}{\times}\ul{\Aut}(\Psi)\ra\fE(\rG,\Psi)\underset{\mathcal{T}}{\times}\fE(\rG,\Psi)$$
as $m(f,\sigma)=(f, f\circ\sigma)$ for all $\rS'$ a scheme over
$\rS$.

Given
$(f_1,f_2)\in(\fE(\rG,\Psi)\underset{\mathcal{T}}{\times}\fE(\rG,\Psi))(\rS')$,
we let $\rF=f_1(\rT_{\rS'})=f_2(\rT_{\rS'})$ and
$\Phi=\Phi(\rG_{\rS'},\rF)$. Then both $f_1^{\Psi},f_2^{\Psi}$
induce isomorphisms from $\Psi_{\rS'}$ to $\Phi$, so
$({f_1^{\Psi}})^{-1}\circ {f_2^{\Psi}}$ is an automorphism of
$\Psi_{\rS'}$. So we can define
$$i:\fE(\rG,\Psi)\underset{\mathcal{T}}{\times}\fE(\rG,\Psi)\ra\fE(\rG,,\Psi)\underset{\rS}{\times}\ul{\Aut}(\Psi)$$
as $i(f_1,f_2)=(f_1,f_1^{-1}\circ f_2)$. Then we have
\begin{center}
$i\circ m(f,\sigma)=i(f,f\circ\sigma)=(f,f^{-1}\circ
f\circ\sigma)=(f, \sigma)$; $m\circ i(f_1,f_2)=m(f_1,{f_1}^{-1}\circ
f_2)=(f_1,f_2).$
\end{center}
Therefore $i$ is the inverse map of $m$ and the claim follows from
Proposition~\ref{0:1}.
\end{proof}

Now we want to show that $\fE(\rG,\Psi)$ is a scheme. As we have
mentioned in Proposition~\ref{0.10}, the group scheme
$\ul{\Aut}(\Psi)$ is \'etale locally constant. Therefore, the
$\ul{\Aut}(\Psi)$-torsor $\fE(\rG,\Psi)$ is representable
by~\cite{SGA3} Exp. X, 5.5.
\end{proof}
For a maximal torus $\rX$ of $\rG$, we let $\rX^{\ad}$ be the
corresponding torus in $\ad(\rG)$. Note that
$\rX/\underline{\Centr}(\rG)\xrightarrow{\sim}\rX^{\ad}$
(ref.~\cite{SGA3}, Exp. 24, Prop. 2.1). For
$f\in\fE(\rG,\Psi)(\rS')$, we define the stabilizer of $f$ under the
$\underline{\Aut}_{\rS'-gr}(\rG_{\rS'})$ as:
\begin{center}
\[\underline{\Stab}(f)(\rS'')=\left\{\begin{array}{l}\mbox{$x\in\Aut_{\rS''-gr}(\rG_{\rS''})|$
$x\circ f_{\mathrm{\rS''}}=f_{\mathrm{\rS''}}$ }
\end{array}\right\}\]
\end{center}

\begin{prop}\label{1:2}
Let $f\in\fE(\rG,\Psi)(\rS')$ and $\rX=f(\rT_{\rS'})$. Then
$\underline{\Stab}(f)$ is isomorphic to $\rX^{\ad}$.
\end{prop}
\begin{proof}
Let $\sigma\in{\Aut}_{\rS''-gr}(\rG_{\rS''})$. Then
$\sigma\in\underline{\Stab}(f)(\rS'')$ if and only if
$\sigma|_{\rX}$ is the identity map on $\rX$, which means
$\underline{\Stab}(f)=\underline{\Aut}_{\rS-gr}(\rG,\id_{\rX})$.
Since $\underline{\Aut}_{\rS-gr}(\rG,\id_{\rX})=\rX^{\ad}$
(ref.~\cite{SGA3}, Exp. 24, Prop. 2.11),
$\underline{\Stab}(f)=\rX^{\ad}$.
\end{proof}
\end{subsection}
\begin{subsection}{Oriented embedding functors}
\begin{subsubsection}{The definition of an orientation}
Let $\Psi=(\sM,\sM^{\vee},\sR,\sR^{\vee})$ be a twisted reduced root
datum over $\rS$, and $\rG$ be a reductive group over $\rS$. Suppose
that $\Psi$ and $\rG$ have the same type at each $s\in\rS$. From
Theorem~\ref{1:1}, we know that $\fE(\rG,\Psi)$ is a homogeneous
space under the action of $\ul{\Aut}_{\rS-gr}(\rG)$. However,
$\ul{\Aut}_{\rS-gr}(\rG)$ may be disconnected, so we would like to
fix an extra datum "v" to make our embedding functor together with
"v" to be a homogeneous space under the adjoint action of $\rG$. The "v" will be
called an orientation of $\Psi$ with respect to $\rG$.

First, we suppose that $\rG$ has a
maximal torus $\rT$.  Let $\Phi(\rG,\rT)$ be the twisted root datum
of $\rG$ with respect to $\rT$.

For an $\rS$-scheme $\rS'$, and for
$\sigma\in\ul{\Aut}_{\rS-gr}(\rG,\rT)(\rS')$, $\sigma$ induces an
automorphism on $\Phi(\rG,\rT)$, and induces a left action on
$f\in\ul{\Isom}(\Psi,\Phi(\rG,\rT))(\rS')$ which is defined as
$$(\sigma\cdot f)(x)=f(x)\circ\sigma^{-1},$$ for all $x\in\sM_{\rS'}(\rS'')$, where $\rS''$ is an $\rS'$-scheme.

Let $\rT'$ be another maximal torus of $\rG$, and
$\ul{\Transt}_{\rG}(\rT,\rT')$ be the strict transporter from $\rT$
to $\rT'$ (cf.~\cite{SGA3} Exp. $\mathrm{VI}_{\mathrm{B}}$, Def. 6.1
(ii)). Then we have a natural morphism (for the convention, we refer
to  \cite{Gir} Chap III, Def. 1.3.1.):
 $$\ul{\Transt}_{\rG}(\rT,\rT')\overset{\ul{\Norm}_{\rG}(\rT)}{\wedge}\ul{\Isom}(\Psi,\Phi(\rG,\rT))\ra\ul{\Isom}(\Psi,\Phi(\rG,\rT')).$$

Since $\ul{\Transt}_{\rG}(\rT,\rT')$ is a right principal
homogeneous space under $\ul{\Norm}_{\rG}(\rT)$ and
$\ul{\Norm}_{\rG}(\rT)$ acts on the left of
$\ul{\Isomext}(\Psi,\Phi(\rG,\rT))$ trivially, we have the following
canonical morphism:
$$\ul{\Isomext}(\Psi,\Phi(\rG,\rT))\simeq\ul{\Transt}_{\rG}(\rT,\rT')\overset{\ul{\Norm}_{\rG}(\rT)}{\wedge}\ul{\Isomext}(\Psi,\Phi(\rG,\rT))
\simeq\ul{\Isomext}(\Psi,\Phi(\rG,\rT')).$$

Therefore, for $\rG$ with a maximal torus $\rT$, we can define
$$\ul{\Isomext}(\Psi,\rG):=\ul{\Isomext}(\Psi,\Phi(\rG,\rT)).$$

In general, since $\rG$ has a maximal torus \'etale locally, we
can find an \'etale covering $\{\rS_i\ra\rS\}_i$ such that
$\rG_{\rS_i}$ has a maximal torus, and we can define
$\ul{\Isomext}(\Psi,\rG)$ by the descent data of
$\ul{\Isomext}(\Psi_{\rS_i},\rG_{\rS_i})$.

An \emph{orientation} of $\Psi$ with respect to $\rG$ is an
$\rS$-point of $\ul{\Isomext}(\Psi,\rG)$. A twisted root datum
$\Psi$ together with an orientation
$v\in\ul{\Isomext}(\Psi,\rG)(\rS)$ is called an oriented root datum
and we denote it as $(\Psi,v)$.

One can also define the functor $\ul{\Isomext}(\rG,\Psi)$ in the
same way. Suppose that $\rG$ is with a maximal torus $\rT$. Then
there is a natural isomorphism $\iota$ between
$\ul{\Isom}(\Psi,\Phi(\rG,\rT))$ and
$\ul{\Isom}(\Phi(\rG,\rT),\Psi)$ which sends $u$ to $u^{-1}$. This
isomorphism also induces an isomorphism between
$\ul{\Isomext}(\Psi,\Phi(\rG,\rT))$ and
$\ul{\Isomext}(\Phi(\rG,\rT),\Psi)$. Let $\rT'$ be another maximal
torus of $\rG$. We have the following commutative diagram:
\begin{center}
$\xymatrix{
  \ul{\Transt}_{\rG}(\rT,\rT')\overset{\ul{\Norm}_{\rG}(\rT)}{\wedge}\ul{\Isom}(\Psi,\Phi(\rG,\rT)) \ar[d] \ar[r] & \ul{\Isom}(\Phi(\rG,\rT),\Psi)\overset{\ul{\Norm}_{\rG}(\rT)}{\wedge} \ul{\Transt}_{\rG}(\rT',\rT)\ar[d] \\
  \ul{\Isom}(\Psi,\Phi(\rG,\rT')) \ar[r] & \ul{\Isom}(\Phi(\rG,\rT'),\Psi)
  }$
\end{center}
Therefore, the morphism $\iota$ defines an isomorphism between
$\ul{\Isomext}(\Psi,\rG)$ and $\ul{\Isomext}(\rG,\Psi)$ and we can
define $\iota$ for an arbitrary reductive group $\rG$ by descent.
\begin{remark}
Actually, in our case, there is no difference between the
transporter $\ul{\Trans}_{\rG}(\rT,\rT')$ and strict transporter
$\ul{\Transt}_{\rG}(\rT,\rT')$ since both $\rT$ and $\rT'$ are
maximal tori.
\end{remark}
\begin{prop}\label{1.14}
Let $\rG'$ be another reductive group over $\rS$. Suppose that
$\rG'$ and $\Psi$ have the same type at each fibre over $\rS$. Then
we will have the following map:
$$\ul{\Isomext}(\Psi,\rG')\times\ul{\Isomext}(\rG,\Psi)\ra\ul{\Isomext}(\rG,\rG').$$
\end{prop}
\begin{proof}
To see this, we first suppose that both $\rG$ and $\rG'$ have
maximal tori. Let $\rT$ and $\rT'$ be the maximal tori of $\rG$ and
$\rG'$ respectively. Then the natural map from
$\ul{\Isom}(\Psi,\Phi(\rG',\rT'))\times\ul{\Isom}(\Phi(\rG,\rT),\Psi)$
to $\ul{\Isom}(\Phi(\rG,\rT),\Phi(\rG',\rT')),$ induces the map
$$\ul{\Isomext}(\Psi,\Phi(\rG',\rT'))\times\ul{\Isomext}(\Phi(\rG,\rT),\Psi)\ra\ul{\Isomext}(\Phi(\rG,\rT),\Phi(\rG',\rT')).$$

We now want to show
$\ul{\Isomext}(\Phi(\rG,\rT),\Phi(\rG',\rT'))\simeq\ul{\Isomext}(\rG,\rG')$.
Note that we have natural morphisms from
$\ul{\Isom}_{\rS-gr}(\rG,\rT;\rG',\rT')/\ul{\Norm}_{\ad(\rG)}(\ad(\rT))$
to $\ul{\Isomext}(\Phi(\rG,\rT),\Phi(\rG',\rT'))$. By~\cite{SGA3}
Exp.XXIV, 2.2,
$$\ul{\Isom}_{\rS-gr}(\rG,\rT;\rG',\rT')/\ul{\Norm}_{\ad(\rG)}(\ad(\rT))\xrightarrow{\sim}\ul{\Isomext}(\rG,\rG')$$
So we have a map from
$$\imath_1:\ul{\Isomext}(\rG,\rG')\ra\ul{\Isomext}(\Phi(\rG,\rT),\Phi(\rG',\rT')).$$

Note that $\ul{\Isomext}(\rG;\rG')$ and
$\ul{\Isomext}(\Phi(\rG,\rT),\Phi(\rG',\rT'))$ are principal
homogeneous spaces under $\ul{\Autext}(\rG)$ and
$\ul{\Autext}(\Phi(\rG,\rT))$ respectively.

By~\cite{SGA3} Exp. XXIV, 2.1, we have
$\ul{\Autext}(\rG,\rT)\simeq\ul{\Autext}(\rG)$. By Theorem~\ref{0.3}
and Proposition~\ref{0.5}, the natural map between
$\ul{\Autext}(\rG,\rT)$ and $\ul{\Autext}(\Phi(\rG,\rT))$ is an
isomorphism on each geometric fiber, so
$$\ul{\Autext}(\rG)\simeq\ul{\Autext}(\rG,\rT)\simeq\ul{\Autext}(\Phi(\rG,\rT)).$$
Under these identifications, the map $\imath_1$ is a morphism
between $\ul{\Autext}(\rG)$-principal homogeneous spaces. So it is
an isomorphism.

For general reductive groups $\rG$ and $\rG'$, we can define this
map by descent.

\end{proof}
\begin{remark}
For a semisimple group $\rG$, there is also a definition of an
orientation of $\rG$ (ref.~\cite{PS}). Let $\rG^{qs}$ be a
quasi-split form of $\rG$, and $\rT'$ be a maximal torus of
$\rG^{qs}$. If we replace $\Psi$ above by $\Phi(\rG^{qs},\rT')$,
then an orientation of $\Psi$ with respect to $\rG$ is called an
orientation of $\rG$ in~\cite{PS} \S 2.
\end{remark}
\end{subsubsection}

\begin{subsubsection}{Oriented embedding functors}
Given an oriented twisted root datum $(\Psi, v)$ of $\Psi$ with
respect to $\rG$, we define the \emph{oriented embedding functor}
as:

\[\fE(\rG,\Psi,v)(\rS')=\left\{\begin{array}{l}\mbox{$f:\rT_{\rS'}\hookrightarrow\rG_{\rS'}$}\left|
\begin{array}{l}\mbox{$f\in\fE(\rG,\Psi)(\rS')$,
and the image of $f^{\Psi}$ } \\
\mbox{ in $\underline{\Isomext}(\Psi,\rG)(\rS')$ is
$v$.}\end{array}\right.\end{array}\right\}\]

With all the notation defined above, we have the following result
similar to Theorem~\ref{1:1}:
\begin{theo}\label{0.7}
Suppose that $\rG$ is reductive. Then in the sense of the \'{e}tale
topology, $\fE(\rG,\Psi,v)$ is a left homogeneous space under the
adjoint action of $\rG$, and a torsor over $\mathcal{T}$ under the
right $\rW(\Psi)$-action. Moreover, $\fE(\rG,\Psi,v)$ is
representable by an affine $\rS$-scheme.
\end{theo}
\begin{proof}
Since $\fE(\rG,\Psi)$ and $\ul{\Isomext}(\Psi,\rG)$ are sheaves,
$\fE(\rG,\Psi,v)$ is an $\rS$-sheaf.

Let $\rS'$ be an $\rS$-scheme, $f_1,f_2\in\fE(\rG,\Psi,v)(\rS')$ and
$\rT_i=f_i(\rT_{\rS'}),$ for $i=1,2$ respectively. There is an
\'{e}tale neighborhood $\rU$ of $\rS'$ such that $\rG$ splits
relatively  to $\rT_{i,\rU}$'s and hence there is
$g\in\ul{\Transt}(\rT_1,\rT_2)(\rU)$ (ref.~\cite{SGA3}, Exp. XXIV,
1.5). By the definition of $\fE(\rG,\Psi,v)$, we know that
$f_1^{\Psi}$ and $f_2^{\Psi}$ have the same image in
$\ul{\Isomext}(\Psi,\rG)(\rS')$, so $g\cdot f_1^{\Psi}$ and
$f_2^{\Psi}$ has the same image in
$\ul{\Isomext}(\Psi,\Phi(\rG,\rT_2))(\rS')$. Since $\rG_{\rU}$
splits relatively to $\rT_{2,\rU}$, we can find
$n\in\ul{\Norm}_{\rG}(\rT)(\rU)$ such that $n\cdot g\cdot
f_1^{\Psi}=f_2^{\Psi}$, which proves that $\fE(\rG,\Psi,v)$ is a
homogeneous space under the adjoint action of $\rG$.

Next, we show that $\pi:\fE(\rG,\Psi,v)\ra\mathcal{T}$ is surjective
as a morphism of sheaves. As we have seen in the proof of
Theorem~\ref{1:1}, $\pi:\fE(\rG,\Psi)\ra\mathcal{T}$ is surjective,
so for an $\rS$-scheme $\rS'$ and $\rX\in\mathcal{T}(\rS')$, there
is an \'etale covering  $\{\rS_i'\ra\rS'\}$ such that for each $i$,
there is $f_i\in\fE(\rG,\Psi)(\rS_i')$ with $\pi(f)=\rX_{\rS_i'}$.
Moreover, we can assume $\rG_{\rS_i'}$ is split relatively to
$\rX_{\rS_i'}$. Then $\underline{\Aut}_{\rS-gr}(\rG,\rT)(\rS_i')$ is mapped
surjectively to $\underline{\Autext}(\rG)(\rS_i')$ (ref.~\cite{SGA3},
Exp. XXIV, 2.1), which allows us to find
$\sigma_i\in\underline{\Aut}_{\rS-gr}(\rG,\rT)(\rS_i')$ such that
$\sigma_i\circ f_i\in\fE(\rG,\Psi,v)(\rS_i')$. Therefore,
$\pi:\fE(\rG,\Psi,v)\ra\mathcal{T}$ is surjective.

Finally, we want to prove that $\fE(\rG,\Psi,v)$ is a right
$\rW(\Psi)$-torsor over $\mathcal{T}$. We identify $\rW(\Psi)$ with
a subgroup of $\ul{\Aut}_{\rS-gr}(\rT)$. So for
$w\in\rW(\Psi)(\rS')$, we can regard it as an element in
$\ul{\Aut}_{\rS-gr}(\rT)(\rS')$. By the  definition of $\underline{\Isomext}(\Psi,\rG)$,
$\rW(\Psi)$ acts trivially on  $\underline{\Isomext}(\Psi,\rG)$.
Therefore, we can consider the map
$$m_v:\fE(\rG,\Psi,v)\underset{\rS}{\times}\rW(\Psi)\ra\fE(\rG,\Psi,v)\underset{\mathcal{T}}{\times}\fE(\rG,\Psi,v)$$
defined as $m_{v}(f,w)=(f,f\circ w),$ for
$f\in\fE(\rG,\Psi,v)(\rS')$, $w\in\rW(\Psi)(\rS')$, where $\rS'$ is
an $\rS$-scheme.

On the other hand, given $f_1$, $f_2\in\fE(\rG,\Psi,v)(\rS')$ with
$f_1(\rT)=f_2(\rT)$, $f_1^{\Psi}$ and $f_2^{\Psi}$ have the same
image in $\ul{\Isomext}(\Psi,\Phi(\rG,f_1(\rT)))$, so
$f_{1}^{-1}\circ f_2$ defines an element in $\rW(\Psi)(\rS')$.

Then we can define the map
$$i_v:\fE(\rG,\Psi,v)\underset{\mathcal{T}}{\times}\fE(\rG,\Psi,v)\ra\fE(\rG,\Psi,v)\underset{\rS}{\times}\rW(\rT)$$ as
$i_v(f_1,f_2)=(f_1,f_{1}^{-1}\circ f_2)$ for
$(f_1,f_2)\in\fE(\rG,\Psi,v)(\rS')\underset{\mathcal{T}}{\times}\fE(\rG,\Psi,v)(\rS')$,
$\rS'$ an $\rS$-scheme. As what we have verified in the proof of
Theorem~\ref{1:1}, we have that $i_v$, $m_v$ are the inverse maps of
each other. Again, by Proposition~\ref{0:1}, we conclude that
$\fE(\rG,\Psi,v)$ is a $\rW(\Psi)$-torsor over $\mathcal{T}$ and
by~\cite{SGA3} Exp. X, 5.5, $\fE(\rG,\Psi,v)$ is representable.
Since $\mathcal{T}$ is affine and $\rW(\Psi)$ is finite,
$\fE(\rG,\Psi,v)$ is represented by an affine $\rS$-scheme.
\end{proof}

For a reductive group $\rG$, we let $\der(\rG)$ be the derived group
of $\rG$ and  $\sesi(\rG)$ be the semisimple group associated to
$\rG$. Let $\sico(\rG)$ be the simply connected group associated to
$\der(\rG)$. The following corollary allows us to reduce the
oriented embedding problem of reductive groups to that of semisimple
simply connected groups, which is useful for arithmetic
purposes.

\begin{coro}\label{0.8}
Let
$v\in\ul{\Isomext}(\Psi,\rG)(\rS).$ Then $v$ induces an orientation
$v_{\der}\in\ul{\Isomext}(\der(\Psi),\der(\rG))(\rS)$. Moreover, we
have a natural isomorphism
$\fE(\rG,\Psi,v)\xrightarrow{\sim}\fE(\der(\rG),\der(\Psi),v_\der)$.

One can also replace $\der(\Psi)$ and $\der(\rG)$ by $\ad(\Psi)$ and
$\ad(\rG)$, $\sesi(\Psi)$ and $\sesi(\rG)$, $\sico(\Psi)$ and
$\sico(\rG)$ respectively.
\end{coro}
\begin{proof}
The key point lies in the functoriality of the induced and coinduced
operation on the root data and the one-to-one correspondence between
the maximal tori of $\rG$ and the maximal tori of $\der(\rG)$ (resp.
$\ad(\rG)$,  $\sico(\rG)$), which gives us a natural isomorphism
from $\ul{\Isomext}(\Psi,\rG)$ to
$\ul{\Isomext}(\der(\Psi),\der(\rG))$ (resp.
$\ul{\Isomext}(\ad(\Psi),\ad(\rG))$,
$\ul{\Isomext}(\sesi(\Psi),\sesi(\rG))$,
$\ul{\Isomext}(\sico(\Psi),\sico(\rG))$). Hence, we only prove the case
for $\der(\Psi)$ and $\der(\rG)$ in detail, all the other cases can
be proved similarly.

Suppose that $\rG$ has a maximal torus $\rT$. Let
$\rT'=\rT\cap\der(\rG)$. Then $\rT'$ is a maximal torus of
$\der(\rG)$ and $\der(\Phi(\rG,\rT))=\Phi(\der(\rG),\rT')$.
Moreover, the scheme of maximal tori of $\rG$ is isomorphic to the
scheme of maximal tori of $\der(\rG)$ (ref.~\cite{SGA3}, Exp. XXII,
6.2.7, 6.2.8). Therefore, there is a natural morphism $i_{\der}$
from $\ul{\Isom}(\Psi,\Phi(\rG,\rT))$ to
$\ul{\Isom}(\der(\Psi),\Phi(\der(\rG),\rT'))$. Moreover, by
Proposition~\ref{0.4}, the natural morphism from $\Psi$ to
$\der(\Psi)$ induces an isomorphism from $\rW(\Psi)$ to
$\rW(\der(\Psi))$, so $i_\der$ induces a natural morphism from
$\ul{\Isomext}(\Psi,\Phi(\rG,\rT))$ to
$\ul{\Isomext}(\der(\Psi),\Phi(\der(\rG),\rT'))$, and hence a
natural morphism from $\ul{\Isomext}(\Psi,\rG)$ to
$\ul{\Isomext}(\der(\Psi),\der(\rG))$ by descent. Therefore, given
$v\in\ul{\Isomext}(\Psi,\rG)(\rS)$, we can have
$v_\der\in\ul{\Isomext}(\der(\Psi),\der(\rG))$ induced by $v$.

Since $\rW(\Psi)$ is isomorphic to $\rW(\der(\Psi))$, by Theorem~\ref{0.7}, both
$\fE(\rG,\Psi,v)$ and $\fE(\der(\rG),\der(\Psi),v_\der)$ are $\rW(\Psi)$-torsors
 over the scheme of
maximal tori. Thus, the natural morphism
$$\fE(\rG,\Psi,v)\xrightarrow{\sim}\fE(\der(\rG),\der(\Psi),v_\der)$$
is an isomorphism.
\end{proof}
\end{subsubsection}
\end{subsection}
\begin{subsection}{Examples--Embedding functors and embedding problems of Azumaya algebras with
involution}\label{s1.3} In this section, we want to show the
relations between the embedding functors and embedding problems of
Azumaya algebras with involution. For the background of Azumaya
algebras, we refer to the book by Knus~\cite{KN} Chap. III, \S 5 and
the paper~\cite{KPS}.

Let $K$ be a commutative ring and suppose that 2 is invertible in
$K$. Let $\rA$ be an Azumaya algebra over $K$ of degree $n$ equipped
with an involution $\tau$. Let $k=K^{\tau}$, which are the elements
in $K$ fixed by $\tau$. If $k=K$, then $\tau$ is said to be of the
first kind. If $K$ is an \'etale quadratic extension over $k$, then
$\tau$ is said to be of the second kind.  Let $\rE$ be a commutative
\'etale algebra over $K$ of rank $n$ equipped with an involution
$\sigma$. Assume $\sigma|K=\tau|K$.

Let $\rU(\rE,\sigma)$ and $\rU(\rA,\tau)$ be two algebraic
$k$-groups defined as follows: for any commutative $k$-algebra $\rC$,
$$\rU(\rE,\sigma)(\rC)=\{x\in\rE\otimes_k\rC|\ x\sigma(x)=1\},$$ and
$$\rU(\rA,\tau)(\rC)=\{x\in\rA\otimes_k\rC|\ x\tau(x)=1\}.$$
Let $\rT=\rU(\rE,\sigma)^{\circ}$, the identity component of $\rU(\rE,\sigma)$,
 and $\rG=\rU(\rA,\tau)^{\circ}$.
Since 2 is invertible in $K$,  $\rG$ is smooth at each fiber.

Then we  associate a root datum $\Psi$ to $\rT$. The idea is to
associate a "split form" $(\rA_0,\tau_0)$ (resp. $(\rE_0,\sigma_0)$)
to each ($\rA,\tau$) (resp. $(\rE,\sigma)$). From the split form
$(\rA_0,\tau_0)$, we get a group $\rG_0$ with a split maximal torus
$\rT_0$. Let $\Phi(\rG_0,\rT_0)$ be the root datum of $\rG_0$ with
respect to $\rT_0$. Then we use the isomorphism between
$\ul{\Aut}(\rE_0,\sigma_0)$ and $\ul{\Aut}(\Psi_0)$ to associate a twisted
root datum $\Psi$ to $(\rE,\sigma)$. This allows us to transfer a
$k$-embedding from $(\rE,\sigma)$ to
$(\rA,\tau)$ to a $k$-point of the embedding functor
$\fE(\rG,\Psi)$. Moreover, we will show that the $k$-points of
$\fE(\rG,\Psi)$ are in one-to-one correspondence
with the $k$-embeddings from $(\rE,\sigma)$ to $(\rA,\tau)$. To
simplify things, we always assume that $\rA$ and $\rE$ have constant rank
over $K$.
\begin{subsubsection}{The root datum associated to $\rT$}\label{s1.3.1}
\begin{paragraph}{Notations for the case where the involution $\tau$ is of the second
kind} If $\tau$ is an involution of the second kind, then we let
$\rA_0$ be $\bM_{n,k}\times\bM_{n,k}^{op}$, where $\bM_{n,k}$ stands
for the $n\times n$-matrix algebra defined over $k$, and let $\rE_0$
be $k^n\times k^n$, which is viewed as an \'etale algebra over
$k\times k$. In this case, we let $\tau_0$ be the exchange
involution of $\rA_0$ defined by $\tau_0(\rM,\rN)=(\rN,\rM)$. Let
$\iota_0:\rE_0\ra\rA_0$ be defined as
$\iota_0(x_1,...,x_n,y_1,...,y_n)=(\diag(x_1,...,x_n),\diag(y_1,...,y_n))$,
where $\diag(x_1,...,x_{n})$ stands for the diagonal matrix with the
$(i,i)$-th entry $x_i$. Clearly it is a $k\times k$-homomorphism and
the image of $\iota_0$ is invariant under $\tau_0$. Let $\sigma_0$ be
the exchange involution on $\rE_0$ induced by $\tau_0$. Let
$\rT_0=\rU(\rE_0,\sigma_0)$ and $\rG_0=\rU(\rA_0,\tau_0)$ and
$f_0:\rT_0\ra\rG_0$ be the embedding induced by $\iota_0$. Let
$\Psi_0$ be the root datum associated to $\rT_0$ defined as
$$\Psi_0(\rC)=\Phi(\rG_0,f_0(\rT_0))(\rC)\circ f_0$$ for any
$k$-algebra $\rC$.

In this case, we let $i_{\rT_0}:\gm_{m,k}\ra\rT_0$ (resp.
$i_{\rG_0}:\gm_{m,K}\ra\rG_0$) denote the embedding defined by the
$k\times k$-structure morphism of $\rE_0$ (resp. $\rA_0$), and let
$i_{\rT}:\rR_{K/k}^{(1)}(\gm_{m,K})\ra\rT$ (resp.
$i_{\rG}:\rR_{K/k}^{(1)}(\gm_{m,K})\ra\rG$) denote the embedding
defined by the $K$-structure morphism of $\rE$ (resp. $\rA$).

An isomorphism between $(\rE_0,\sigma_0,k\times k)$ and
$(\rE,\sigma,K)$ is a $k$-isomorphism between $\rE_0$ and $\rE$ commuting with
the involutions, and sends $k\times k$ to $K$. Let $\fX=\ul{\Isom}((\rE_0,\sigma_0,k\times
k),(\rE,\sigma,K))$ be the isomorphism functor between
$(\rE_0,\sigma_0,k\times k)$ and $(\rE,\sigma,K)$. Note that
$\fX(\ol{s})$ is not empty for each geometric point $\ol{s}$ of
$\spec(k)$, if and only if $\rk_k\rE^{\sigma}=\rk_K\rE$. Throughout
this article, we assume that $\fX$ is non-empty. Then $\fX$ is a
right $\ul{\Aut}(\rE_0,\sigma_0,k\times k)$-torsor.
\end{paragraph}
\begin{paragraph}{Notations for the case where the involution $\tau$ is of the first
kind}\label{s1.3.1.2} For $\tau$ an involution of the first kind, we
let $\rA_0=\bM_{n,k}$, and $\rE_0=k^n$. Let $\iota_0:\rE_0\ra\rA_0$
be defined as $\iota(x_1,...,x_{n})=\diag(x_1,...,x_{n})$.

If $\tau$ is an orthogonal involution and $n$ is odd, we let $n=2m+1$ and
\\$\rB=(b_{ij})_{0\leq i,j\leq 2m}$, where
$$b_{i,j}=\left\{
         \begin{array}{ll}
             1, & \hbox{if i=j=0,} \\
             1, & \hbox{if i=j$\pm$ m, with i,j$\geq$ 1} \\
             0, & \hbox{otherwise.}
           \end{array}
         \right.$$

For $\tau$ an orthogonal involution and $n$ even, we let $n=2m$ and
\\$\rB=(b_{i,j})_{1\leq i,j\leq 2m}$, where
$$b_{i,j}=\left\{
         \begin{array}{ll}
             1, & \hbox{if i=j$\pm$ m, with i,j$\geq$ 1} \\
             0, & \hbox{otherwise.}
           \end{array}
         \right.$$

For $\tau$ a symplectic involution, we let $n=2m$ and
\\$\rB=(b_{i,j})_{1\leq i,j\leq 2m}$, where
$$b_{i,j}=\left\{
         \begin{array}{ll}
             1, & \hbox{if j=i+m,} \\
             -1, & \hbox{if j=i-m,} \\
             0, & \hbox{otherwise.}
           \end{array}
         \right.$$

Let $\tau_0$ be the involution on $\rA_0$ defined by
$\tau_0(\rM)=\rB\rM^t\rB^{-1}$, and let $\sigma_0$ be the involution
on $\rE_0$ induced by $\tau_0$. Let
$\rT_0=\rU(\rE_0,\sigma_0)^{\circ}$ and
$\rG_0=\rU(\rA_0,\tau_0)^{\circ}$. Let $f_0:\rT_0\ra\rG_0$ be the
embedding induced by $\iota_0$.

Let $\Psi_0$ be the root datum associated to $\rT_0$ defined as
$$\Psi_0(\rC)=\Phi(\rG_0,f_0(\rT_0))(\rC)\circ f_0$$ for any
$k$-algebra $\rC$. For $\tau$ of the first kind, let
$\fX=\ul{\Isom}((\rE_0,\sigma_0),(\rE,\sigma))$.

Note that $\fX(\ol{s})$ is non-empty for each geometric point
$\ol{s}$ of $\spec(k)$ if and only if
$\rk_K\rE^{\sigma}=\lceil\frac{1}{2}\rk_K\rE\rceil$. Throughout this
article, we assume that $\fX$ is non-empty.

\end{paragraph}
\paragraph{The definition of the twisted root datum $\Psi$}
\begin{defn}~\label{1.6}
\begin{itemize}
\item [1]
Suppose that $\tau$ is of the first kind. A \emph{$k$-embedding}
$\iota:(\rE,\sigma)\ra(\rA,\tau)$ is an injective $k$-homomorphism
commuting with the involutions.
\item [2]Suppose $\tau$ is of the second kind.
A \emph{$k$-embedding} $\iota:(\rE,\sigma)\ra(\rA,\tau)$ is an
injective $k$-homomorphism commuting with  the involutions and sends
$K$ to $K$.
\item [3]
Let $\iota$ be a $k$-embedding from $(\rE,\sigma)$ to $(\rA,\tau)$.
Define the isomorphisms between $(\rE_0,\rA_0,\iota_0)$ and
$(\rE,\rA,\iota)$ to be pairs $(\alpha,\beta)$, where $\alpha$
is an isomorphism between $(\rE_0,\sigma_0)$ and $(\rE,\sigma)$,
$\beta$ is an isomorphism between $(\rA_0,\tau_0)$ and $(\rA,\tau)$,
and $\alpha$, $\beta$ satisfy  $\iota\circ\alpha=\beta\circ\iota_0$.
Let $\ul{\Isom}((\rE_0,\rA_0,\iota_0),(\rE,\rA,\iota))$ be the
isomorphism functor between $(\rE_0,\rA_0,\iota_0)$ and
$(\rE,\rA,\iota)$.
\item[4]
A morphism $f:\rT\ra\rG$ is called an embedding if it is a closed
immersion and a group homomorphism.  Let $f:\rT\ra\rG$ be an
embedding. Define the isomorphisms between $(\rG_0,\rT_0,f_0)$ and
$(\rG,\rT,f)$ to be pairs $(h,g)$, where $h$ is an isomorphism
from $\rT_0$ to $\rT$ and $g$ is an isomorphism from $\rG_0$ to
$\rG$, and $h$, $g$ satisfy $f\circ h=g\circ f_0$. Let
$\ul{\Isom}((\rG_0,\rT_0,f_0),(\rG,\rT,f))$ be the isomorphism
functor between  $(\rG_0,\rT_0,f_0)$ and $(\rG,\rT,f)$.
\end{itemize}
\end{defn}
\begin{remark}~\label{1.8}
Suppose that $\tau$ is of the second kind, and $\iota$ is an embedding
from $(\rE,\sigma)$ to $(\rA,\tau)$. For a $k$-algebra $\rC$
and
$(\alpha,\beta)\in\ul{\Isom}((\rE_0,\rA_0,\iota_0),(\rE,\rA,\iota))(\rC)$,
$\alpha$ will automatically be in $\ul{\Isom}((\rE_0,\sigma_0,
k\times k),(\rE,\sigma,K))(\rC)$, because $\beta$ sends the center
of $\rA_0$ to the center of $\rA$ and
$\iota\circ\alpha=\beta\circ\iota_0$.
\end{remark}

\begin{remark}~\label{1.9}
Let f be a $k$-point of $\fE(\rG,\Psi)$. Since $f$ induces an
isomorphism between $\Psi$ and $\Phi(\rG,f(\rT))$, f induces an
isomorphism between $\rad(\Psi)$ and $\rad(\Phi(\rG,f(\rT)))$. As
$i_\rT(\rR_{K/k}^{(1)}(\gm_{m,K}))$ (resp.
$i_\rG(\rR_{K/k}^{(1)}(\gm_{m,K}))$) is the torus associated to
$\rad(\Psi)$ (resp. $\rad(\Phi(\rG,f(\rT)))$), f maps
$i_\rT(\rR_{K/k}^{(1)}(\gm_{m,K}))$ to
$i_\rG(\rR_{K/k}^{(1)}(\gm_{m,K}))$.
\end{remark}

The following lemma enables us to attach a twisted root datum to the
torus $\rT$.
\begin{lemma}\label{1.13}
Let $\rS=\spec(k)$. Then we
have the following:
\begin{itemize}
  \item [1] The canonical homomorphism from
  $\ul{\Aut}(\rE_0,\rA_0,\iota_0)$ to
  $\ul{\Aut}_{\rS-gr}(\rG_0,f_0(\rT_0))$ is an isomorphism except for $\rG_0$ of type $D_4$ or $\rA$ of degree 2 with $\tau$ orthogonal.
  \item [2] For the involution $\tau_0$ of the first kind, there is a canonical monomorphism $j_{\rE_0}$ from
  $\ul{\Aut}(\rE_0,\sigma_0)$ to
  $\ul{\Aut}(\Psi_0)$. In particular,  if $\Psi_0$ is not of type
  $D_4$, then the homomorphism $j_{\rE_0}$ is an isomorphism.
  \item [3] For the involution $\tau_0$ of the second kind, there is a canonical isomorphism from
  $\ul{\Aut}(\rE_0, \sigma_0,k\times k)$ to
  $\ul{\Aut}(\Psi)$.
\end{itemize}
\end{lemma}
\begin{proof}
 To verify that the canonical homomorphism
$j_{\rA_0}$ from $\ul{\Aut}(\rE_0,\rA_0,\iota_0)$ to
$\ul{\Aut}_{\rS-gr}(\rG_0,f_0(\rT_0))$ is an isomorphism,  it
suffices to verify that the natural morphism from
$\ul{\Aut}(\rA_0,\tau_0)$ to $\ul{\Aut}_{\rS-gr}(\rG_0)$ is an
isomorphism, since the automorphism preserves $\iota_0(\rE_0)$ if
and only if it preserves $f_0(\rT_0)$. To see that the natural
morphism from $\ul{\Aut}(\rA_0,\tau_0)$ to
$\ul{\Aut}_{\rS-gr}(\rG_0)$ is an isomorphism, we check it case by
case. For $\rG_0$ of type $A_n$, let $\sigma$ be the automorphism of
$(\rA_0,\tau_0)$ which maps $(\rM,\rN)\in\rA_0$ to
$(\rN^{t},\rM^t)$, where $\rN^{t}$ denotes the transpose of $\rN$.
Then $\ul{\Aut}(\rA_0,\tau_0)$ is $\mathrm{PGL}_{n+1}\rtimes
\ent/2\ent$, where $\ent/2\ent$ is generated by $\sigma$. Note that
$\sigma$ induces the outer automorphism of
$\rG_0=\mathrm{GL}_{n+1}$, and we have
$\ul{\Aut}_{\rS-gr}(\rG_0)=\ul{\Aut}_{\rS-gr}(\mathrm{GL}_{n+1})=\mathrm{PGL}_{n+1}\rtimes
\ent/2\ent$. For $\rG_0$ of type $B_n$, we have
$$\ul{\Aut}(\rA_0,\tau_0)=\mathrm{PGO}(\rA_0,\tau_0)\cong\rG_0$$ (cf.~\cite{KMRT}, Thm. 12.15 and Prop. 12.4).
Since $\rG_0$ is adjoint of type $B_n$ in this case, we have
$\rG_0=\ul{\Aut}_{\rS-gr}(\rG_0)$ and hence
$$\ul{\Aut}(\rA_0,\tau_0)\cong\rG_0=\ul{\Aut}_{\rS-gr}(\rG_0).$$
Similar calculation can be done for $\rG_0$ of type $C_n$ or $\rG_0$
of type $D_n$ with $n\geq2$ and $n\neq4$, and we refer to
\cite{KMRT}, Theorem 26.14 and Theorem 26.15.

To prove (2), we first note that there is a natural isomorphism
$j_{\rE_0}$ from  $\ul{\Aut}(\rE_0,\sigma_0)$ to
$\ul{\Aut}_{\rS-gr}(\rT_0)$. To see that the image of $j_{\rE_0}$ is
contained in $\ul{\Aut}(\Psi)$, we  verify it case by case. For
example, for $\sigma_0$ orthogonal and $\rE$ of degree $2m$, the
automorphism group $\ul{\Aut}(\rE_0,\sigma_0)$ is isomorphic to the
constant group scheme $((\ent/2\ent)^{m}\rtimes\rS_m)_{\rS}$. We can
check that the corresponding action of
$((\ent/2\ent)^{m}\rtimes\rS_m)_{\rS}$ on $\rT_0$ actually preserves
the root datum $\Psi_0$. Moreover, by~\cite{Bou} Plan. IV, we know
that $((\ent/2\ent)^{m}\rtimes\rS_m)_{\rS}$ is exactly the
automorphism group of $\Psi_0$ for $m\neq 4$ and is a subgroup of
$\ul{\Aut}(\Psi_0)$ for $m=4$. From this, we conclude that
$j_{\rE_0}$ maps $\ul{\Aut}(\rE_0,\sigma_0)$ isomorphically to
$\ul{\Aut}(\Psi)$ for $\Psi_0$ not of type $D_4$. One can check the
other cases in the same way, which allows us to conclude the
statement (2). One can prove (3) in the same way.
\end{proof}
Since $\rT_0$ is a maximal torus of $\rG_0$, we have the following
exact sequence:
$$\xymatrix@C=0.5cm{
  0 \ar[r] & \ad(\rT_0) \ar[r] &\ul{\Aut}_{\rS-gr}(\rG_0,f_0(\rT_0))  \ar[r] & \ul{\Aut}(\Psi_0) \ar[r] & 0 }$$
  by Proposition~\ref{0.5} and~\cite{SGA3} Exp.
  XXIV, Proposition 2.6. Therefore, we can summarize the above lemma
  as the following diagram:
 $$\xymatrix{
 0 \ar[r] & \ad(\rT_0) \ar[d]\ar[r] &\ul{\Aut}(\rA_0,\rE_0,\iota_0)  \ar[d]^{\alpha}\ar[r] &
 \ul{\Aut}(\rE_0,\sigma)\ar[d]^{\beta}\ar[r] &0\\
 0 \ar[r] & \ad(\rT_0)  \ar[r] &\ul{\Aut}_{\rS-gr}(\rG_0,f_0(\rT_0)) \ar[r] & \ul{\Aut}(\Psi_0) \ar[r] & 0
  },$$ where $\alpha$, $\beta$ are isomorphisms if $\tau_0$ is of  the first kind and $\rG_0$ is not of
  type $D_4$. For $\tau_0$  of  the second kind, we just replace
  $\ul{\Aut}(\rE_0,\sigma)$ by $\ul{\Aut}(\rE_0,\sigma,k\times k)$
  and $\alpha$, $\beta$ are isomorphisms.
Now for the involution $\tau_0$ of the first kind, we  define the
twisted root datum $\Psi$ related to $\rT$ as
$$\Psi:=\fX\wedge^{\Aut(\rE_0,\sigma_0)}\Psi_0.$$
For the involution $\tau_0$ of the second kind, we define the
twisted root datum $\Psi$ related to $\rT$ as
$$\Psi:=\fX\wedge^{\Aut(\rE_0,\sigma_0,k\times k)}\Psi_0.$$
\begin{remark}~\label{1.17}
If we regard $\Psi_0$ as a set of combinatorial data satisfying the
axioms of root data, the canonical morphism $\beta$ between
$\ul{\Aut}(\rE_0,\sigma_0)$ and $\ul{\Aut}(\Psi_0)$ is defined over
any arbitrary base. However, for the involution $\tau_0$ of the
first kind, the group $\rG_0$ is not reductive over arbitrary base.
Hence, we ask  $2$ to be invertible over the base so that  $\Psi_0$
can be regarded as a root datum of $\rG_0$.
\end{remark}

\begin{remark}~\label{1.5}
We have a canonical morphism from $\fX$ to $\ul{\Isom}(\Psi_0,\Psi)$
which is an isomorphism except $\Psi_0$ is of type $D_4$ by
Lemma~\ref{1.13}. The natural morphism from
$\ul{\Isom}{((\rA_0,\tau_0),(\rA,\tau))}$ to
$\ul{\Isom}{(\rG_0,,\rG)}$ over $k$ is a canonical monomorphism
which is an isomorphism except $\rG$ is of type $D_4$, since
$\ul{\Aut}(\rA_0,\tau_0)=\ul{\Aut}(\rG_0)$ except for $\rG_0$ of
type $D_4$ (~\cite{KMRT} Chap. IV, \S 23 and \S 26).
\end{remark}
\begin{remark}\label{1.16}
For $\rA$  of degree $2$ with $\tau$ orthogonal, the corresponding
split group $\rG_0$ is actually the one dimensional split torus.
Therefore, we have $\rG_0$ acts trivially on itself but non
trivially on $\rA_0$. However, the conjugation by $\begin{pmatrix}
0 & 1 \\
1 & 0 \\
\end{pmatrix}$
induces a nontrivial isomorphism of $\rG_0$. Hence the natural
morphism from
  $\ul{\Aut}(\rA_0,\rE_0,\iota_0)$ to
  $\ul{\Aut}_{\rS-gr}(\rG_0,f_0(\rT_0))$ is surjective but not
  injective.
However, in this case, we have
$$\ul{\Aut}(\rE_0,\sigma_0)\simeq(\ent/2\ent)_{K}\simeq\ul{\Aut}(\Psi_0),$$
so the natural map from $\ul{\Isom}(\rE_0,\rE)$ to
$\ul{\Isom}(\Psi_0,\Psi)$ is still an isomorphism.
\end{remark}
\end{subsubsection}
\begin{subsubsection}{Embedding functors and embedding problems for algebras with
involution}\label{s1.3.2}



\begin{theo}~\label{1.4}
Keep all the notations defined above. Then:
\begin{itemize}
  \item [1.] The set of $k$-embeddings from $(\rE,\sigma)$ into $(\rA,\tau)$ is in one-to-one correspondence with the set of $k$-points of
  $\fE(\rG,\Psi)$, except for $\rG$ of type $D_4$ or $\rA$
  of degree 2 with $\tau$ orthogonal.
  \item [2.] If $\tau$ is of the second kind,
  the set of $K$-algebra embeddings from $(\rE,\sigma)$ into $(\rA,\tau)$ is in one-to-one correspondence with the set of $k$-points $f$ of $\fE(\rG,\Psi)$ which satisfy $f\circ i_\rT=i_\rG$.
\end{itemize}
\end{theo}
\begin{proof}
The crucial ingredient of the proof is Lemma~\ref{1.13}. We prove
(1) first. Let $\iota$ be a $k$-embedding from $(\rE,\sigma)$ to
$(\rA,\tau)$. Clearly, $\iota$ induces an embedding $f:\rT\ra\rG$.
To see that $f$ is a $k$ point of $\fE(\rG,\Psi)$, we need to verify
that $f$ induces an isomorphism between $\Psi$ and
$\Phi(\rG,f(\rT))$.

Let $\fY=\ul{\Isom}((\rE_0,\rA_0,\iota_0),(\rE,\rA,\iota))$. By
Lemma~\ref{1.13}, we have
$$\fY\simeq\ul{\Isom}((\rG_0,f_0(\rT_0)),(\rG,f(\rT))).$$ This allows
us to define an isomorphism from
$\fY\wedge^{\ul{\Aut}(\rE_0,\rA_0,\iota_0)}(\rG_0,f_0(\rT_0))$ to
$(\rG,f(\rT))$, which induces an isomorphism from
$\fY\wedge^{\ul{\Aut}(\rE_0,\rA_0,\iota_0)}\Phi(\rG_0,f_0(\rT_0))$
to $\Phi(\rG,f(\rT))$.

Given a $k$-algebra $\rC$ and $(\alpha,\beta)\in\fY(\rC)$, we have a
natural map from $\fY$ to $\fX$ which maps $(\alpha,\beta)$ to
$\alpha$. By the definition of $\Psi_0$, $f_0$ induces an
isomorphism between $\Phi(\rG_0,f_0(\rT_0))$ and $\Psi_0$. Therefore
we have the following natural map
$$\fY\wedge^{\ul{\Aut}(\rG_0,f_0(\rT_0))}\Phi(\rG_0,f_0(\rT_0))\simeq\fY\wedge^{\ul{\Aut}(\rG_0,f_0(\rT_0))}\Psi_0\ra
\fX\wedge^{\ul{\Aut}(\Psi_0)}\Psi_0=\Psi.$$ Since
$\fY\wedge^{\ul{\Aut}(\rG_0,f_0(\rT_0))}\Phi(\rG_0,f_0(\rT_0))$ is
isomorphic to $\Phi(\rG,f(\rT))$, we conclude that $f$ induces a
natural map from $\Phi(\rG,f(\rT))$ to $\Psi$. Since this natural
map becomes an isomorphism at each geometric fiber, it is an
isomorphism and hence $f\in\fE(\rG,\Psi)(k)$, and we denote it as
$\rI_\rA(\iota)$.

Given $f\in\fE(\rG,\Psi)(k)$, now we want to define a $k$-embedding
$\iota$ from $f$. Note that $f$ induces an isomorphism from
$\Phi(\rG,f(\rT))$ to $\Psi$ by definition.

Let $\fZ=\ul{\Isom}((\rG_0,\rT_0,f_0),(\rG,\rT,f))$. Note that for a
$k$-algebra $\rC$ and $(h,g)\in\fZ(\rC)$, $g$ induces an isomorphism
between $\Phi(\rG_0,f_0(\rT_0))$ and $\Phi(\rG,f(\rT))$. Hence  h is
an element of $\ul{\Isom}(\Psi_0,\Psi)(\rC)$ and we have a natural
morphism from $\fZ$ to $\ul{\Isom}(\Psi_0,\Psi)$.

By Lemma~\ref{1.13}, there is a canonical isomorphism from
$\ul{\Isom}((\rE_0,\sigma_0),(\rE,\sigma))$ (resp.
$\ul{\Isom}((\rE_0,\sigma_0,k\times k),(\rE,\sigma,K))$, if $\tau$
is of the second kind) to $\ul{\Isom}(\Psi_0,\Psi)$, so we have a
canonical morphism from $\fZ$ to
$\ul{\Isom}((\rE_0,\sigma_0),(\rE,\sigma))$, and hence a canonical
map from $\fZ\wedge(\rE_0,\sigma_0)$ to $(\rE,\sigma)$. Similarly,
by Remark~\ref{1.5}, we have a canonical map from $\fZ$ to
$\ul{\Isom}((\rA_0,\tau_0),(\rA,\tau))$, and hence a canonical map
from $\fZ\wedge(\rA_0,\tau_0)$ to $(\rA,\tau)$. Therefore, we get a
$k$-embedding $\iota:(\rE,\sigma)\ra(\rA,\tau)$ from the map
$\fZ\wedge(\rE_0,\sigma_0)$ to $\fZ\wedge(\rA_0,\tau_0)$ induced by
$\iota_0$, and we denote $\iota$ as $\rJ_\rG(f)$.

Clearly, $\rI_\rA\rJ_\rG(f)=f$ since we construct $\iota$ from $\fZ$
and $\fZ\bigwedge(\rG_0,\rT_0,f_0)$ is canonically isomorphic to
$(\rG,\rT,f)$. On the other hand, we have
$$\rJ_\rG\circ\rI_\rA(\iota)=\iota$$ because of the canonical
isomorphism from $\ul{\Isom}((\rE_0,\rA_0,\iota_0),(\rE,\rA,\iota))$
to \\$\ul{\Isom}((\rG_0,\rT_0,f_0),(\rG,\rT,f))$, where $f$ is
induced by $\iota$. Hence, the first assertion follows.

We prove (2) now. Clearly, if $\iota:(\rE,\sigma)\ra(\rA,\tau)$ is a
$K$-embedding, then the corresponding $k$-embedding $f$ will be a
$k$-point of $\fE(\rG,\Psi)$ and $f$ satisfies $f\circ i_\rT=i_\rG$.

Now suppose $f\in\fE(\rG,\Psi)(k)$ and $f\circ i_\rT=i_\rG$. Then we
need to verify that the map $\rJ_\rG(f)$ from
$\fZ\wedge(\rE_0,\sigma_0)$ to $\fZ\wedge(\rA_0,\tau_0)$ is a
$K$-morphism. From the construction of $\rJ_\rG(f)$, it is enough to
prove that the two maps from
$\ul{\Isom}((\rG_0,\rT_0,f_0),(\rG,\rT,f))(\rR)$ to
$\ul{\Isom}(\gm_{m,k},\rR_{K/k}^{(1)}(\gm_{m,K}))(\rR)$, which map
$(h,g)$ to $i_\rT^{-1}\circ h\circ i_{0,\rT}$ and $i_\rG^{-1}\circ
g\circ i_{0,\rG}$ respectively, coincide. However, it is a direct
consequence from the fact that $f\circ i_\rT=i_\rG$, since
\begin{align*}
i_\rT^{-1}\circ h\circ i_{0,\rT}&=i_\rG^{-1}\circ f\circ h\circ i_{0,\rT}\\
&=i_\rG^{-1}\circ g\circ f_0\circ i_{0,\rT}\\
&=i_\rG^{-1}\circ g\circ i_{0,\rG}
\end{align*}
Therefore, $\rJ_\rG(f)$ is a $K$-algebra morphism.
\end{proof}


\begin{remark}~\label{1.12}
Let $\tau$ be of the second kind. Suppose that $\fE(\rG,\Psi)(k)$ is
nonempty and fix $f\in\fE(\rG,\Psi)(k)$. If $f\circ i_\rT\neq
i_\rG$, then $f\circ\sigma$ will satisfy $(f\circ\sigma)\circ i_\rT=
i_\rG$ since $\sigma$ acts on $\rR_{K/k}^{(1)}(\gm_{m,K})$ as -1.
Therefore, the existence of a $k$-embedding will imply the existence
of a $K$-embedding. Moreover, we will see that the condition $f\circ
i_\rT= i_\rG$ gives a particular orientation
$u\in\ul{\Isomext}(\Psi,\rG)(k)$.
\end{remark}

Now we want to consider the case where $\rG_0$ is of type $D_4$.
Note that since there is a natural monomorphism from
$\ul{\Aut}(\rA_0,\rE_0,\iota_0)$ to
$\ul{\Aut}_{k-gr}(\rG_0,f_0(\rT_0))$, we can still get a $k$-point
of the embedding functor $\fE(\rG,\Psi)$ from a $k$-embedding
$\iota:(\rE,\sigma)\ra(\rA,\tau)$. The problem is that given a
$k$-point $f$ of the embedding functor $\fE(\rG,\Psi)$, we can not
get a $k$-embedding from $f$ as we have done in the proof of
Theorem~\ref{1.4}, because the canonical map from
$\ul{\Aut}(\rA_0,\rE_0,\iota_0)$ to
$\ul{\Aut}_{k-gr}(\rG_0,f_0(\rT_0))$ is not an isomorphism.

To fix the problem, we first observe that  $\ad(\rG_0)$ (resp.
$\rW(\Psi_0)$) is in the image of the canonical morphism from
$\ul{\Aut}(\rA_0,\tau_0)$ (resp. $\ul{\Aut}(\rE_0,\sigma_0)$) to
$\ul{\Aut}_{k-gr}(\rG_0)$ (resp. $\ul{\Aut}(\Psi_0)$). So instead to
associate a split form $(\rA_0,\tau_0)$ to $(\rA,\tau)$, we consider
all "quasi-split" forms of $(\rA,\tau)$. Note that
$\ul{\Aut}(\rA_0,\tau_0)/\ad(\rG_0)$ is the constant group scheme
$(\ent/2\ent)_{k}$ and we can find a section from $(\ent/2\ent)_{k}$
to $\ul{\Aut}(\rA_0,\rE_0,\iota_0)$. For example  we can send $1$ in
$\ent/2\ent$ to the matrix
\begin{center}
$\begin{pmatrix}
0 & 0 & 0 & 0 & 1 & 0 & 0 & 0 \\
0 & 1 & 0 & 0 & 0 & 0 & 0 & 0\\
0 & 0 & 1 & 0 & 0 & 0 & 0 & 0\\
0 & 0 & 0 & 1 & 0 & 0 & 0 & 0\\
1 & 0 & 0 & 0 & 0 & 0 & 0 & 0\\
0 & 0 & 0 & 0 & 0 & 1 & 0 & 0\\
0 & 0 & 0 & 0 & 0 & 0 & 1 & 0\\
0 & 0 & 0 & 0 & 0 & 0 & 0 & 1
\end{pmatrix}$.
\end{center}
Let us fix the section from $(\ent/2\ent)_{k}$ to
$\ul{\Aut}(\rA_0,\rE_0,\iota_0)$ as above.  Let
$$\ul{\Isomext}(\rA_0,\tau_0;\rA,\tau):=\ul{\Isom}(\rA_0,\tau_0;\rA,\tau)/\ad(\rG_0).$$
Then for each $(\rA,\tau)$ we can associate a quasi-split form
$$(\rA_{q},\tau_q)=\ul{\Isomext}(\rA_0,\tau_0;\rA,\tau)\wedge^{(\ent/2\ent)_{k}}(\rA_0,\tau_0).$$
Moreover, since the section has image in
$\ul{\Aut}(\rA_0,\rE_0,\iota_0)$, we get an \'etale algebra with
involution $(\rE_q,\sigma_q)$ and a embedding
$\iota_q:(\rE_q,\sigma_q)\ra(\rA_{q},\tau_q)$ from the datum
$(\rA_0,\rE_0,\iota_0)$. Let $\rG_q=\rU(\rA_q,\tau_q)^{\circ}$ and
$\rT_q=\rU(\rE_q,\sigma_q)$ and $f_q:\rT_q\ra\rG_q$ be the morphism
induced by $\iota_q$.

From our construction,  the group $\rG$ in an inner form of $\rG_q$,
so we can always fix an orientation $v$ in
$\ul{\Isomext}(\rG_q,\rG)(k)$. Then we have the following result:
\begin{prop}\label{1.15}
Let $u$ be a $k$-point of $\ul{\Isomext}(\Psi,\rG)$. Then each $k$-point
of the oriented embedding functor $\fE(\rG,\Psi,u)$ corresponds to a
$k$-embedding $\iota$ from $(\rE,\sigma)$ to $(\rA,\tau)$.
\end{prop}
\begin{proof}
The way to prove it is exactly the same as in Theorem~\ref{1.4}. The
only different thing is that we stay in the inner case. First we fix
an orientation $v$ in $\ul{\Isomext}(\rG_q,\rG)(k).$ Let $u_q$ be
the orientation in $\ul{\Isomext}(\Psi_q,\rG_q)(k)$ which comes from
the map $f_q$. Then there is an orientation $v'=u^{-1}\circ v\circ
u_q$ in $\ul{\Isomext}(\Psi_q,\Psi)(k)$ by Proposition~\ref{1.14}.

For $f\in\fE(\rG,\Psi,u)(k)$, we consider the
$\ul{\Norm}_{\ad(\rG_q)}(\ad(f_q(\rT_q)))$-torsor\\
$\ul{\Isomint}_{v}(\rG_q,f_q(\rT_q);\rG,f(\rT))$. Let
$\fZ'=\ul{\Isomint}_{v}(\rG_q,f_q(\rT_q);\rG,f(\rT))$. Clearly, we
have a canonical morphism from $\fZ'$ to
$\ul{\Isomint}_v(\rG_q,\rG)$ which is an $\ad(\rG_q)$-torsor; and a
canonical morphism from $\fZ'$ to $\ul{\Isomint}_{v'}(\Psi_q,\Psi)$
which is a $\rW(\Psi_q)$-torsor. Note that $\ad(\rG_q)$ (resp.
$\rW(\Psi_q)$) are in the image of the canonical morphism from
$\ul{\Aut}(\rA_q,\tau_q)$ (resp. $\ul{\Aut}(\rE_q,\sigma_q)$) to
$\ul{\Aut}_{k-gr}(\rG_q)$ (resp. $\ul{\Aut}(\Psi_q)$) as they do in
the split case. So we get a $k$-embedding
$\iota:(\rE,\sigma)\ra(\rA,\tau)$ from the map
$\fZ'\wedge^{\rW(\Psi_q)}(\rE_q,\sigma_q)$ to
$\fZ'\wedge^{\ad(\rG_q)}(\rA_q,\tau_q)$ induced by $\iota_q$.
\end{proof}
\end{subsubsection}
\end{subsection}
\end{section}

\begin{section}{Arithmetic properties of embedding functors}
In this section, we focus on the arithmetic properties of the
embedding functor. The main arithmetic technique which we use here
has been developed by Borovoi~\cite{Bo}.

In the first part, we recall the main result in~\cite{Bo}. In the
second part, we give a criterion for an oriented embedding functor
to satisfy the local-global principle. Besides, over a local field
$\rL$, we use the Tits index to give a necessary and sufficient
condition for an oriented embedding functor to have an $\rL$-point.

For a field $k$ with characteristic different from $2$,  embedding
an \'etale algebra over $k$ into a central simple algebra over $k$
commuting with involutions is equivalent to finding a $k$-point of
the corresponding embedding functor. In Section 2.4, we use the
arithmetic properties of oriented embedding functors to give an
alternative proof of Theorem A, Theorem 6.7 and Theorem 7.3 in the
work of Prasad and Rapinchuk~\cite{PR1}.

Throughout this section, we let $k$ be a global field and $k^{s}$ be
a separable closure. Let $\Gamma$ be the absolute Galois group of
$k$, and $\Omega_{k}$ be the set of all places of $k$.

We start this section with some general facts of the local-global
principle of homogeneous spaces established in Borovoi's papers.
\begin{subsection}{The local-global principle for homogeneous spaces}
First, we let $k$ to be a number field. For a $k$-linear algebraic
group $\rG$, we let $\rG^{\circ}$ to be the connected component
containing the neutral element of $\rG$. Let $\rG^{u}$ be the
unipotent radical of $\rG^{\circ}$; $\rG^{red}=\rG^{\circ}/\rG^{u}$;
$\rG^{ss}$ be the derived group of $\rG^{red}$;
$\rG^{tor}$=$\rG^{red}/\rG^{ss}$. Let
$\rG^{ssu}=\ker[\rG^{\circ}\ra\rG^{tor}]$. If $\rG/\rG^{ssu}$ is
abelian, we let $\rG^{mult}=\rG/\rG^{ssu}$ which is a multiplicative
group.

Let $\rX$ be a left homogeneous space under a connected linear
algebraic group $\rG$ over $k$. Let $x\in\rX(k^s)$ and
$\ol{\rH}=\Stab_{\rG_{k^s}}(x)$ be the stabilizer of $x$.

Throughout this section, we will assume that $\rG^{ss}$ is simply
connected, and $\ol{\rH}/\ol{\rH}^{ssu}$ is abelian.

Since $\Gamma$ has
a natural action on $\rG(k^s)$, we can define $\fG$ to be the
semidirect product $\rG(k^s)\rtimes\Gamma$. We have a $\fG$- action
on $\rX(k^s)$ defined as $(g,\sigma)x=g\cdot\sigma(x)$. Let
$\fH=\Stab_{\fG}(x)$. Then we have the following exact sequence:
$$(\ast)\ \ \xymatrix {1\ar[r]& \ol{\rH}(k^s)\ar[r]^i& \fH \ar[r]^{p}&\Gamma\ar[r]& 1},$$
where $i(\ol{h})=(\ol{h},1)$, and $p$ is the projection to $\Gamma$.

Since $\ol{\rH}^{mult}$ is commutative, we can define a
$\Gamma$-action on $\ol{\rH}^{mult}$ by conjugation. To be precise,
for $\sigma\in\Gamma$, choose $(g_{\sigma},\sigma)\in
p^{-1}(\sigma)$. Then since $g_{\sigma}\cdot\sigma(x)=x$, we have
$\inter(g_{\sigma})^{\sigma}\ol{\rH}^{mult}=\ol{\rH}^{mult}$. Note
that, the above definition doesn't depend on the lifting of $\sigma$
in $\fH$ because $\ol{\rH}^{mult}$ is commutative. Hence we get a
$\Gamma$-action on $\ol{\rH}^{mult}$.  Since the point $x$ is
defined over some finite extension $L$ of $k$, $\ol{\rH}^{mult}$ is
defined over $L$. Moreover, for each $\sigma\in\Gal(k^s/L)$, we can
choose $g_{\sigma}=1$. Hence, there is a $k$-form $\rH^m$ of
$\ol{\rH}^{mult}$ defined by this $\Gamma$-action (cf.~\cite{BS}
2.12,~\cite{Ser1} Chap. V, \S 4, $\mathrm{n}^{\circ}$. 20,
and~\cite{FSS} 1.15). One can verify that the isomorphism class of
$\rH^{m}$  is independent of the choice of the geometric point $x$.
Therefore, given $\rG$ and $\rX$, the isomorphism class of $\rH^{m}$
is well-defined (\cite{Bo}, 1.2).

Let $\ol{j}:\ol{\rH}\ra\rG_{k^s}$ be the natural inclusion. Then
$\ol{j}$ induces a group morphism from $\ol{\rH}^{mult}$ to
$\rG^{tor}_{k^s}$, which descends to a group morphism
$j:\rH^{m}\ra\rG^{tor}$ over $k$.

Consider the complex $\xymatrix {0\ar[r]& \rH^{m}\ar[r]^j& \rG^{tor}
\ar[r]& 0}$, where $\rH^{m}$ is in degree $-1$ and $\rG^{tor}$ is in
degree 0. Let $\rH^{1}(k,\rH^{m}\ra\rG^{tor})$ be the first Galois
hypercohomology group of the above complex, and
$\sha^{1}(k,\rH^{m}\ra\rG^{tor})$ be the kernel of the localization
map
$\rH^{1}(k,\rH^{m}\ra\rG^{tor})\ra\underset{v\in\Omega_k}{\prod}\rH^{1}(k_v,\rH^{m}\ra\rG^{tor})$.

For $\sigma\in\Gamma$, let $(g_{\sigma},\sigma)\in\fH$. Let
$u_{\sigma,\tau}=g_{\sigma\tau}({g_\sigma}^{\sigma} g_\tau)^{-1}$,
$\hat{u}_{\sigma,\tau}$ be the image of $u_{\sigma,\tau}$ in
$\ol{\rH}^{mult}(k^s)$, and $\hat{g}_{\sigma}$ be the image of
$g_{\sigma}$ in $\rG^{tor}({k^s})$. Then $(\hat{u},\hat{g})$ is a
hypercocycle, and we let
$\eta(\rX)=\Cl(\hat{u},\hat{g})\in\rH^{1}(k,\rH^{m}\ra\rG^{tor})$.
Note that $\eta(\rX)$ is well-defined (see~\cite{Bo}, 1.4).

We will make use of the following two theorems later.

\begin{theo}~\label{2.1}
Let $k_v$ be a nonarchimedean local field of characteristic $0$. Let
$\rG$, $\rX$ be as above. If $\eta(\rX)=0$, then $\rX$ has a
$k_v$-point.
\end{theo}
\begin{proof}
\cite{Bo}, Thm. 2.1.
\end{proof}
\begin{theo}~\label{2.2}
Let $k$ be a number field, and let $\rG$, $\rX$ be as above. Assume
that $\rX(k_{v})$ is nonempty for every place $v$ of $k$ and
$\eta(\rX)=0$. Then $\rX$ has a rational point.
\end{theo}
\begin{proof}
\cite{Bo}, Thm. 2.2.
\end{proof}
\begin{remark}
Note that if $\rX$ has a $k_v$-point at all places $v\in\Omega_k$,
then  $\eta(\rX)$ lies in $\sha^{1}(k,\rH^{m}\ra\rG^{tor})$.
\end{remark}
\begin{remark}
In fact, up to a sign, $\eta(\rX)$ is the Brauer-Manin obstruction
of $\rX$ (ref.~\cite{Bo}, Thm. 4.5).
\end{remark}
Now, we let $k$ be a global function field, for example,
$k=\mathbb{F}_{q}(t)$, where $\mathbb{F}_{q}$ is a finite field with
$q$ elements. One may ask if Theorem~\ref{2.1} holds over $k$.
Indeed,  we have similar results  when $\rG$ is a connected
reductive group over $k$ and $\rX$ is a $\rG$-homogeneous space for
the \'etale topology.  Since $\rX$ is a $\rG$-homogeneous space for
the \'etale topology, the set $\rX(k^{s})$ is nonempty. Let
$x\in\rX(k^{s})$ and $\ol{\rH}=\Stab_{\rG_{k^{s}}}(x)$. Suppose that
$\ol{\rH}$ is connected reductive. Then we can define an $k$-torus
$\rH^{m}$, which is an $k$-form of $\ol{\rH}^{mult}$ as above. Let
$\fH$ and $\eta(\rX)$ be defined as above. Then we have the
following:
\begin{prop}~\label{2.17}
Let $k$ be a global function field. Let $\rG$ be a connected
reductive group over $k$ and $\rX$ be a $\rG$-homogeneous space for
the \'etale topology. Let $x\in\rX(k^s)$. Define $\ol{\rH}$ as
above. Suppose that $\rG^{ss}$ is simply connected and $\ol{\rH}$ is
a torus. If $\eta(\rX)=0$, then $\rX$ has a $k$-point. The same
result also holds over $k_v$ for $v\in\Omega_k$.
\end{prop}
\begin{proof}
The key point of the proof is that $\rH^{1}(k,\rG^{ss})=0$, for
$\rG^{ss}$ semisimple simply connected (\cite{H}, Satz A
and~\cite{Th}, Thm. A).

For $\sigma\in\Gamma$, let $(g_{\sigma},\sigma)\in\fH$. Let
$u_{\sigma,\tau}=g_{\sigma\tau}({g_\sigma}^{\sigma} g_\tau)^{-1}$.
As above, we have
$$\eta(\rX)=\Cl(\hat{u},\hat{g})\in\rH^{1}(k,\rH^{m}\xrightarrow{j}\rG^{tor}).$$
Since $\ol{\rH}$ is a torus, we have $\ol{\rH}^{mult}=\ol{\rH}$ and
$\rH^{m}$ is a $k$-form of $\ol{\rH}$. Suppose that $\eta(\rX)=0$.
Then we have $a_{\sigma}\in\rH^{m}(k^{s})$ and $s\in\rG(k^{s})$ such
that
$$(\hat{u}_{\sigma,\tau},\hat{g}_\sigma)=(-\partial(a_\sigma),j(a_\sigma)\partial
s),$$ i . e. $u_{\sigma,\tau}=a_{\sigma\tau}({a_\sigma}^{\sigma}
a_\tau)^{-1}$ and $g_{\sigma}=s^{-1}\cdot j(a_\sigma)\cdot\
^{\sigma}s$ (mod $\rG^{ss}$). After replacing $g_\sigma$ by
$a_\sigma^{-1} g_\sigma$, we can assume $u_{\sigma,\tau}=0$. We also
replace $x$ by $s\cdot x$, then we get $g_\sigma\in\rG^{ss}(k^s)$
and $u_{\sigma,\tau}=0$. Therefore, $(g_\sigma)$ is a cocycle of
$\Gamma$ with values in $\rG^{ss}$. Since $\rH^{1}(k,\rG^{ss})=0$
when $\rG^{ss}$ is semisimple simply connected, there is
$t\in\rG^{ss}(k^s)$ such that $s_{\sigma}=t^{-1}\ ^\sigma t$. Then
$t\cdot x$ is a $k$-point of $\rX$.
\end{proof}
\begin{remark}
For $k$ with positive characteristic, we ask $\rG$ to be reductive
because we want to ensure that $\rG^{red}$, $\rG^{ss}$ and
$\rG^{tor}$ are properly defined. Otherwise, it may happen that the
$k$-unipotent radical of $\rG$ is trivial but $\rG$ is not reductive
(see~\cite{CGP}, Example 1.1.3).
\end{remark}
\begin{remark}
The above proposition is also true over a totally imaginary number
field $k$, because in this case, $\rH^{1}(k,\rG^{ss})=0$ by Kneser's
Theorem (\cite{K} Chap IV, Thm. 1 and Chap. V, Thm. 1).
\end{remark}
\end{subsection}

\begin{subsection}{Local-Global principle for oriented embedding functors}
Let $k$ be a global field. Unless otherwise specified, $\rG$ is a
reductive $k$-group, $\Psi$ is a twisted root datum, and $\rT$ is
the $k$-torus determined by $\Psi$. In the following, we always
assume that $\rG$ and $\Psi$ have the same type.

Let $\sico(\rT)$ be the torus determined by the simply connected
root datum $\sico(\Psi)$. We will show that the only obstruction to
the local-global principal for the oriented functor
$\fE(\rG,\Psi,u)$ lies in the Shafarevich group
$\sha^{2}(k,\sico(\rT))$. Moreover, the local-global principle holds
for the oriented embedding functor $\fE(\rG,\Psi,u)$ if
$\ul{\Dyn}(\Psi)$ is of type $C$ or $\rT$ is anisotropic at some
place $v$.

Note that since $\rG$ is reductive,  we have $\rG^{ss}=\der(\rG)$. A
direct application of Theorem~\ref{2.2} is the following:
\begin{prop}~\label{2.3}
Let $u$ be an orientation of $\Psi$ with respect to $\rG$. Then the
only obstruction for $\fE(\rG,\Psi,u)$ to satisfy the local-global
principle lies in the group $\sha^{2}(k,\sico(\rT))$. In particular, if
$\rG$ has no outer automorphisms and $\sha^{2}(k,\sico(\rT))$
vanishes, then $\fE(\rG,\Psi)$ satisfies the local-global principle.
\end{prop}
Before proving the above proposition, we prove the following lemma:
\begin{lemma}\label{2.14}
Suppose furthermore that $\rG$ is a semisimple simply connected
group over $k$, and $\Psi$ is a twisted simply connected root datum.
Let $u\in\ul{\Isomext}(\Psi,\rG)(k)$. As we have shown in
Theorem~\ref{0.7}, the oriented embedding functor $\fE(\rG,\Psi,u)$
is a left homogeneous space under the adjoint $\rG$-action. Then
under this $\rG$-action, the corresponding $\rH^{m}$ (which is
defined in Section 2.1) is isomorphic to $\rT$.
\end{lemma}
\begin{proof}
Given $f\in\fE(\rG,\Psi,u)(k^{s})$, the stabilizer of $f$ in
$\rG_{k^{s}}$ is $f(\rT_{k^{s}})$. Since $f(\rT_{k^{s}})$ is a
torus, we have $f(\rT_{k^{s}})^{mult}=f(\rT_{k^{s}})$. For
$\sigma\in\Gamma$, $\sigma$ acts on $f$ as $^\sigma f=\sigma\circ
f\circ\sigma^{-1}$. Let $\fG=\rG(k^{s})\rtimes\Gamma,$ and
$\fH=\Stab_{\fG}(f)$. For $g\in\rG(k^{s})$, we let $\inter(g)$
denote the conjugation action of $g$ on $\rG$. Then for
$(g_\sigma,\sigma)\in\fH$, we have $\inter(g_\sigma)\circ\ ^\sigma
f=f$, which means $g_\sigma\cdot\ ^\sigma f(t)\cdot
g_\sigma^{-1}=f(t)$ for all $t\in\rT(k^{s})$. Therefore, we have
$$g_\sigma\cdot \sigma (f(t))\cdot g_\sigma^{-1}=f(\sigma(t)),$$
which means $f$ is a $k$-isomorphism between $\rT$ and
$f(\rT_{k^{s}})^{m}$. Therefore, the $\rH^{m}$ defined in Section
2.1 is isomorphic to $\rT$.
\end{proof}
Now, we are ready to prove Proposition~\ref{2.3}.
\begin{proof}[Proof of Proposition~\ref{2.3}]
By Corollary~\ref{0.8}, it suffices to prove this proposition for
$\fE(\sico(\rG),\sico(\Psi),u_\sico)$. Since
$\fE(\sico(\rG),\sico(\Psi),u_\sico)$ is a homogeneous space under
$\sico(\rG)$,  by Lemma \ref{2.14}, we know that the $\rH^{m}$
corresponding to this $\sico(\rG)$-action  is isomorphic to
$\sico(\rT)$.

Since $\sico(\rG)^{tor}$ is trivial, by Theorem~\ref{2.2} and
Proposition~\ref{2.17}, the only obstruction for
$\fE(\sico(\rG),\sico(\Psi),u_\sico)$ to satisfy the local-global
principle lies in the group $\sha^{2}(k,\sico(\rT))$. The rest of
the proposition then follows.
\end{proof}

Let $k_v$ be a nonarchimedean local field. Then by combining
previous results, we
get the following corollary:
\begin{coro}\label{2.15}
If the group $\rH^{2}(k_v,\sico(\rT))$ is trivial, then the oriented
embedding functor $\fE(\rG,\Psi,u)$ has a $k_v$-point.
\end{coro}
\begin{proof}
By Corollary~\ref{0.8}, it is enough to prove that
$\fE(\sico(\rG),\sico(\Psi),u_\sico)(k_v)$ is nonempty. By
Lemma~\ref{2.14}, the group $\rH^{m}$ for the
$\sico(\rG)$-homogeneous space $\fE(\sico(\rG),\sico(\Psi),
u_\sico)$ is isomorphic to $\sico(\rT)$. Since $\sico(\rG)^{tor}$ is
trivial, we have
$$\rH^{1}(k_v,\rH^{m}\ra\sico(\rG)^{tor})=\rH^{2}(k_v,\sico(\rT)).$$
By Theorem~\ref{2.1} and Proposition~\ref{2.17}, the set
$\fE(\sico(\rG),\sico(\Psi),u_\sico)(k_v)$ is nonempty if
$\rH^{2}(k_v,\sico(\rT))$ is trivial.
\end{proof}
For a twisted root datum $\Psi$, the Galois group $\Gamma$ has a
natural action on $\Psi_{k^{s}}$. Therefore, we have a group
homomorphism from $\Gamma$ to $\ul{\Aut}(\Psi)(k^{s})$. Recall that
$\Psi$ is said to be \emph{generic}, if the image of $\Gamma$ in
$\ul{\Aut}(\Psi)(k^{s})$ contains the Weyl group $\rW(\Psi)(k^{s})$.

\begin{theo}\label{2.12}
Let $\rG$ be a reductive group over $k$, $\Psi$ be a twisted root
datum over $k$, and $\rT$ be the torus determined by $\Psi$. Let
$u\in\ul{\Isomext}(\Psi,\rG)(k)$. Suppose that $\Psi$ satisfies one
of the following conditions:
\begin{itemize}
  \item [1.] all connected components of $\ul{\Dyn}(\Psi)(k^{s})$ are  of type $C$.
  \item [2.] $\rT$ is anisotropic at one place $v\in\Omega_k$.
\end{itemize}
Then the local-global principle holds for the existence of a
$k$-point of the oriented embedding functor $\fE(\rG,\Psi,u)$. In
particular, when $\Psi$ is generic, the local-global principle
holds.
\end{theo}
\begin{proof}
If $\Psi$ satisfies one of the above conditions, then $\sico(\Psi)$
also satisfies one of them. Therefore, we can assume that $\Psi$ and
$\rG$ are semisimple simply connected.

By Proposition~\ref{2.3}, the local-global principle holds for the
existence of  $k$-points of the oriented embedding functor
$\fE(\rG,\Psi,u)$ if $\sha^2(k,\rT)$ vanishes. Therefore, it is
enough to prove $\sha^2(k,\rT)=0$ for $\Psi$ satisfying either
condition.

Suppose that $\Psi$ satisfies condition (1).  Let $\Psi_0$ be the
split simple simply connected root datum of type $C_n$
(ref.~\cite{Bou} Plan. III). Let $\rE_0$ be the \'etale algebra
$k^{n}\times k^{n}$ and $\sigma_0$ be the involution which exchanges
the two copies of $k^{n}$. When $\Psi$ is simple simply connected of
type $C_n$, $\Psi$ corresponds to some twisted form $(\rE,\sigma)$
of $(\rE_0,\sigma_0)$ by Lemma~\ref{1.13}, and the torus $\rT$
determined by $\Psi$ is
$\rU(\rE,\sigma)=\rR_{\rE^{\sigma}/k}(\rR_{\rE/\rE^{\sigma}}^{(1)}(\gm_m))$.

Consider the exact sequence:
$$\xymatrix {1\ar[r]& \rR_{\rE/\rE^{\sigma}}^{(1)}(\gm_m)\ar[r]& \rR_{\rE/\rE^{\sigma}}(\gm_m) \ar[r] &\gm_m\ar[r]& 1}.$$
By Hilbert Theorem 90, we have $$\xymatrix
{0\ar[r]&\rH^{2}(\rE^{\sigma},\rR_{\rE/\rE^{\sigma}}^{(1)}(\gm_m))\ar[r]&\rH^{2}(\rE^{\sigma},\rR_{\rE/\rE^{\sigma}}(\gm_m))}.$$
By Shapiro's Lemma,
$\sha^{2}(\rE^{\sigma},\rR_{\rE/\rE^{\sigma}}(\gm_m))=\sha^{2}(\rE,\gm_m)$.
By Brauer-Hasse-Noether Theorem, $\sha^{2}(\rE,\gm_m)=0$. Therefore,
we have
$$\sha^2(k,\rT)=\sha^{2}(\rE^{\sigma},\rR_{\rE/\rE^{\sigma}}^{(1)}(\gm_m))=0.$$

For $\Psi$ not simple, since we can decompose $\Psi_{k^{s}}$ into a
product of isotypic root data  (~\cite{SGA3} Exp. XXI, 6.4.1 and
7.1.6), we can also decompose $\Psi$ into a product of isotypic
twisted root data $\Psi_i$ by descent. By the same reasoning as in
~\cite{SGA3} Exp. XXIV, 5.8 or in~\cite{CGP} Theorem A.5.14, we know
that there exists some \'etale algebra $\rF_i$ over $k$ such that
$\Psi_{i,{\rF_i}}$ is a product of copies of the twisted simple root
datum $\Psi_{i,0}$ over $\rF_i$, and the automorphism group
$\Aut(\rF_i/k)$ acts on $\Psi_{i,\rF_i}$ by permuting
$\Psi_{i,0}$'s. So we have $\Psi_i=\rR_{\rF_i/k}(\Psi_{i,0})$ and
the torus $\rT$ will take the form
$\underset{i}{\prod}\rR_{\rF_i/k}(\rT_{i,0})$, where $\rT_{i,0}$ is
the torus determined by the twisted root datum $\Psi_{i,0}$. Then we
know that $\Psi_{i,0}$ is a  twisted root datum defined by an
\'etale algebra with involution over $F_i$ (Section \ref{s1.3.1.2}).
As in the above discussion, we will have
$\rT_{i,0}=\rU(\rE_i,\sigma_i)$ where $\rE_i$ is an \'etale algebra
over $\rF_i$. By Shapiro's
Lemma,$$\sha^2(k,\rR_{\rF_i/k}(\rT_{i,0}))=\sha^2(\rF_i,\rT_{i,0})=\sha^2(\rE_i^{\sigma_i},\rR_{\rE_i/\rE_i^{\sigma_i}}^{(1)}(\gm_m))=0.$$
By Proposition~\ref{2.3}, the theorem holds when $\Psi$ satisfies
the first condition.

Now, suppose that $\rT$ is anisotropic at some place $v\in\Omega_k$.
Then by Kneser's Theorem (ref.~\cite{San} Lemma 1.9), we have
$\sha^2(k,\rT)=0$.

To complete the proof, we will show that if $\Psi$ is generic, then
$\rT$ is anisotropic at some place $v$. Suppose that $\Psi$ is
generic. Let $L$ be a finite Galois extension of $k$ which splits
$\rT$. Then there exists an element $\sigma\in\Gal(L/k)$ such that
$\sigma$ acts on $\Psi_{L}$ as the Coxeter element
$\omega\in\rW(\Psi)(L)$. Let $\rM$ be the character group of
$\rT_{L}$. Then  the set $\rM^{\omega}=0$, by Theorem 1
in~\cite{Bou} Chap. V, \S 6, and hence $\rM^{\sigma}=0$. By
\v{C}ebotarev Density Theorem, there exists a place $v$ such that
$\sigma$ generates the Frobenius map at $v$, so $\rT$ is anisotropic
at $v$.
\end{proof}
\begin{remark}
Note that for $\Psi$ generic without type $C_n$ (for $n\geq1$),
there is an even stronger result by Klyachko (cf.~\cite{Kl} I.5)
saying that $\rH^{1}(k,\rM)=0$, where $\rM$ is the character group
of $\sico(\rT).$
\end{remark}
\end{subsection}
\begin{subsection}{Oriented embedding functors over local fields}

Let $\rG$ be a reductive group over a local field $L$, and $\Psi$ be
a twisted root datum over $L$. Suppose that
 $\rG$ and $\Psi$ have the
same type and $\ul{\Isomext}(\Psi,\rG)(L)$ is not empty. Let
$u\in\ul{\Isomext}(\Psi,\rG)(L)$. In the following, we are going to
show that the existence of an $L$-point of the oriented embedding
functor is actually determined by the Tits indices of $\Psi$ and
$\rG$. Note that the existence of an orientation $u$ is important
here, because it gives a map between the Dynkin schemes
$\ul{\Dyn}(\rG)$ and $\ul{\Dyn}(\Psi)$, which allows us to compare
the Tits indices of $\rG$ and $\Psi$. An orientation also allows us
to replace the reductive group $\rG$ by the adjoint group $\ad(\rG)$
or simply connected group $\sico(\rG)$ as we have shown in
Corollary~\ref{0.8}.

\subsubsection{Tits indices}
We recall briefly the definition of the Tits index. For a detailed
introduction on Tits indices, we refer to Tits's
paper~\cite{Tit}. For the Tits indices of reductive groups over
connected semilocal rings, one can refer to~\cite{SGA3} Exp. XXVI,
\S5, \S6, and \S7. One can also look at  Petrov and Stavrova's
paper~\cite{PS} \S5. For the Tits indices of a twisted root datum,
we refer to Gille's paper~\cite{Gi} \S7.

Let $\rS$ be the spectrum of a semilocal ring. Let $\rG$ be an
$\rS$-reductive group. For each $\rG$, there exists a minimal
parabolic subgroup $\rP_{min}$ of $\rG$. Let $t_{min}$ be the type
of $\rP_{min}$. Note that given $\rG$, the type $t_{min}$ is well
defined (ref.~\cite{SGA3} Exp. XXVI, 5.7). Moreover, if $\rS$ is
connected, then we call the type $t_{min}$ the Tits index of $\rG$,
and denote it by $\Delta^{\circ}(\rG)$.

For a reduced root datum $\psi=(\rM,\rM^{\vee},\rR,\rR^{\vee})$, a
parabolic subset $P$ is a closed subset of $\rR$ which contains a
system of simple roots. For a reduced twisted root datum
$\Psi=(\sM,\sM^{\vee},\sR,\sR^{\vee})$ over $\rS$, a parabolic
subsheaf (fpqc) $\sP$ is a subsheaf of $\rR$ which is locally
isomorphic to a parabolic subset. Let $\ul{\Par}(\Psi)$ be the
functor such that for each $\rS$-scheme $\rS'$,
$\ul{\Par}(\Psi)(\rS')$ is the set of all the parabolic subsheaves
(fpqc) of $\Psi$ over $\rS'$. Similarly, we can define a type map
$\bt_{\Psi}$ from $\ul{\Par}(\Psi)$ to $\ul{\Dyn}(\Psi)$.

Let $t_{min}$ be the type of a minimal parabolic subsheaf of $\Psi$
(\cite{Gi}, Prop. 7.1). If $\rS$ is connected, then we call
$t_{min}$ the Tits index of $\Psi$, and denote it by
$\Delta^{\circ}(\Psi)$. Note that $\Delta^\circ(\rG)$ and
$\Delta^\circ(\Psi)$ only depend on the roots, so they are invariant
under the operations $\sico$, $\ad$,....

\subsubsection{A criterion for the existence of points of the oriented embedding functor over a local field L}
Let $L$ be a local field of arbitrary characteristic. We have the
following criterion for the existence of an $L$-point of the
oriented embedding functor:
\begin{theo}~\label{1.10} Let $\rG$ be a reductive group
over a local field $L$, and $\Psi$ be a twisted root datum over $L$.
Suppose that $\rG$ and $\Psi$ have the same type and
$\ul{\Isomext}(\Psi,\rG)(L)$ is not empty. Let
$u\in\ul{\Isomext}(\Psi,\rG)(L)$. Then
$\fE(\rG,\Psi,u)(L)\neq\emptyset$ if and only if
$u(\Delta^\circ(\Psi))\supseteq \Delta^\circ(\rG)$.
\end{theo}
\begin{proof}
First, we suppose that $\fE(\rG,\Psi,u)(L)\neq\emptyset$ and let
$f\in\fE(\rG,\Psi,u)(L)$. Since $\Psi\simeq\Phi(\rG,f(\rT))$, from
Prop. 7.2 in~\cite{Gi}, $u(\Delta^\circ(\Psi))\supseteq
\Delta^\circ(\rG)$.

Now, suppose $u(\Delta^\circ(\Psi))\supseteq \Delta^\circ(\rG)$, and
we want to show that $\fE(\rG,\Psi,u)(L)$ is nonempty. Again, by
Corollary~\ref{0.8}, we only need to consider the problem for
$\sico(\rG)$ and $\sico(\Psi)$. Therefore, we can assume $\rG$ is
simply connected, and $\Psi$ is reduced simply connected.

Let $\rT$ be the torus determined by $\Psi$ and
$\rI=\Delta^\circ(\Psi)$. We start with the case where $\rT$ is
anisotropic, i.e. $\rI=\ul{\Dyn}(\Psi)(L)$.

\noindent{\textbf{Case 1.}} $L$ is non-archimedean.  Since $\rT$ is
anisotropic, by Tate-Nakayama Theorem, we have $\rH^{2}(L,\rT)=0$
(cf.~\cite{K} 3.2, Thm. 5). Since $\rG$ and $\Psi$ are simply
connected, by Corollary~\ref{2.15}, the oriented embedding functor
$\fE(\rG,\Psi,u)$ has an $L$-point.

\noindent{\textbf{Case 2.}} $L=\re$. In this case, we consider the
oriented embedding functor $\fE(\ad(\rG),\ad(\Psi),u_\ad)$. Let
$\sigma$ be the nontrivial element of $\Gal(\cpx/\re)$. Let
$\ad(\rT)$ be the torus associated to the root datum $\ad(\Psi)$.
Since $\rT$ is anisotropic, the torus $\ad(\rT)$ is also anisotropic
and $\ad(\rT)\simeq(\rR_{\re/\cpx}^{(1)}(\gm_m))^r$.

Suppose that $\ad(\rG)$ is anisotropic and pick a maximal torus
$\rS$ of $\ad(\rG)$. Since $\ad(\Psi)$ and $\ad(\rG)$ have the same
type, there is a $\cpx$-point $f$ of $\fE(\ad(\rG),\ad(\Psi),u_\ad)$
which maps $\ad(\rT)$ to $\rS$. Since $\sigma$ acts on the character
group of the anisotropic torus by -1, $\sigma$ commutes with $f$.
Therefore, $f$ is an $\re$-point of the oriented embedding functor
$\fE(\ad(\rG),\ad(\Psi),u_\ad)$.

Suppose that $\ad(\rG)$ is not anisotropic. Then we can find an
anisotropic form $\tilde{\rG}$ of $\ad(\rG)$ by~\cite{Ge}, Corollary
7. Since $\tilde{\rG}$ has the same type with $\ad(\Psi)$, by the
above argument, we have an $\re$-point $f$ of
$\fE(\tilde{\rG},\ad(\Psi))$. Then f defines an $\re$-point $\tilde
u$ of $\ul{\Isomext}(\ad(\Psi),\tilde{\rG})$. The orientation
$u_\ad$ together with $\tilde{u}$ gives an orientation
$u_\ad\circ\tilde{u}^{-1}\in\ul{\Isomext}(\tilde{\rG},\ad(\rG))(\re)$.
Hence $\ad(\rG)$ is an inner form of $\tilde{\rG}$. However, the
natural inclusion from $\rH^1(\re,f(\ad(\rT)))$ to
$\rH^1(\re,\tilde{\rG})$ is surjective (~\cite{Ge} Thm. 3), so
$\ad(\rG)$ has an anisotropic torus $\rS$. Let $h$ belong to
$\fE(\ad(\rG),\ad(\Psi),u_\ad)(\cpx)$ and suppose that $h$ maps
$\ad(T)$ to $\rS$. Again, since $\sigma$ acts on the character group
of the anisotropic torus by -1, $\sigma$ commutes with $h$ and $h$
is  an $\re$-point of the oriented embedding functor
$\fE(\ad(\rG),\ad(\Psi),u_\ad)$. By Corollary~\ref{0.8}, the
oriented embedding functor $\fE(\rG,\Psi,u)$ has an $\re$-point.

Therefore, the proposition is true when $\rT$ is anisotropic.

Now, suppose that $\rT$ is arbitrary.  Since $u(\rI)\supseteq
\Delta^\circ(\rG)$, we can find a parabolic subgroup $\rP_\rI$ of
$\rG$ such that the type of $\rP$ is $u(\rI)$ by
Proposition~\ref{0.9}. Let $\rL_\rI$ be a Levi subgroup of $\rP_\rI$
and $\rT'$ be a maximal torus of $\rL_\rI$. Let
$\Psi'=\Phi(\rG,\rT')$, and $\Psi'_{\rI}=\Phi(\rL_\rI,\rT')$. Let
$\sP_\rI$ be the subsheaf of roots of $\Psi'$ which is determined by
$\rP_\rI$. Note that $u$ corresponds to an element in
$\ul{\Isomext}(\Psi,\Psi')(L)$, which we still denote as $u$.

Let $\Psi=(\sM,\sM^{\vee},\sR,\sR^{\vee})$. Let $\sP$ be a minimal
parabolic subsheaf of $\sR$. Then by definition, type $\sP=\rI$. Let
$\sR_\rI$ be the subsheaf of $\sP$ defined by the property: for any
$L$-scheme $\rX$,
\begin{center}
$x\in\sR_{\rI}(\rX)$, if and only if both $x$ and $-x$ are in
$\sP_{\rI}(\rX)$.
\end{center}

Let $\Psi_\rI$ be the root system given by
$(\sM,\sM^{\vee},\sR_\rI,\sR^{\vee}_\rI)$. Define
$$\rQ=\ul{\Isomint}_u(\Psi,\sP;\Psi',\sP_\rI)=\ul{\Isom}(\Psi,\sP;\Psi',\sP_\rI)\bigcap\ul{\Isomint}_u(\Psi,\Psi').$$
Note that $\rQ$ is a right $\rW(\Psi_\rI)$-torsor over $\spec(L)$
(for the \'etale topology), so $\rQ$ is representable. By the
definition of $\rQ$, each $h\in\rQ(\rX)$ will send the sheaf
$\sR_{\rI}$ to the sheaf of roots of $\rL_{\rI}$, because
$\rL_{\rI}$ is the unique Levi subgroup of $\rP_{\rI}$ which
contains $\rT'$. Therefore, we have a natural map
$$i_\rI:\rQ\ra\ul\Isom(\Psi_\rI,\Psi_\rI').$$  Let $L^{s}$ be a separable closure of $L$.
Let $x\in\rQ(L^s)$. By the definition of $\rQ$, the image of $x$ in
$\ul\Isomext(\Psi,\Psi')(L^s)$ is $u$. Moreover, since  $\rQ$ is a
right $\rW(\Psi_\rI)$-torsor and $\rW(\Psi_\rI)$ acts trivially on
$\ul\Isomext(\Psi_\rI,\Psi'_\rI)$, $i_\rI(x)$ defines an $L$-point
of $\ul\Isomext(\Psi_\rI,\Psi'_\rI)$ and hence an $L$-point of
$\ul\Isomext(\Psi_\rI,\L_\rI)$. We denote it by $u_\rI$. Note that
the definition of $u_\rI$ is independent of the choice of $\rT'$.

Now we consider the functor $\fE(\rL_\rI,\Psi_\rI,u_\rI)$. We claim
that if $\fE(\rL_\rI,\Psi_\rI,u_\rI)$ has an $L$-point, then
$\fE(\rG,\Psi,u)$ has an $L$-point.

Suppose that $\fE(\rL_\rI,\Psi_\rI,u_\rI)$ has an $L$-point. Let
$f\in\fE(\rL_\rI,\Psi_\rI,u_\rI)(L)$. Then we replace the torus
$\rT'$ above by $f(\rT)$. By the definition of $\rQ$ and $u_\rI$, we
have a natural morphism
$$j:\rQ\ra\ul{\Isomint}_{u_\rI}(\Psi_\rI,\Psi'_\rI).$$ Since both of
them are $\rW(\Psi_\rI)$-torsors, $j$ is an isomorphism. As
$\fE(\rL_\rI,\Psi_\rI,u_\rI)$ has an $L$-point,
$\ul{\Isomint}_{u_\rI}(\Psi_\rI,\Psi'_\rI)(L)$ is not empty, so
$\rQ$ has an $L$-point as well, which means
$\ul{\Isomint}_u(\Psi,\Psi')(L)\neq\emptyset$. Hence,
$\fE(\rG,\Psi,u)$ has an $L$-point.

Now, by Corollary~\ref{0.8}, it is enough to prove that
$\fE(\der(\rL_\rI),\der(\Psi_\rI),u_{\rI,\der})$ has an $L$-point.
Note that $\der(\Psi)$ is reduced simply connected as $\Psi$ is
(ref.~\cite{SGA3}, Exp. XXI, 6.5.11). Since the torus $\der(\rT)$
determined by $\der(\Psi)$ is anisotropic, it follows that
$\fE(\der(\rL_\rI),\der(\Psi_\rI),u_{\rI,\der})$ has an $L$-point as
we have seen above. This finishes the proof.
\end{proof}
\begin{example}
The above theorem does not hold over arbitrary fields. Here is an
example. Let $K=\rat(\sqrt{-1})$ and $\sigma$ be the conjugation on
$K$, and $k=\rat$. Let $\rT$ be the torus
$\rT=\rR^{(1)}_{K/k}(\gm_m)$. Since $\rT$ is of dimension 1, there
is only one semisimple simply connected root datum with respect to
$\rT$. Let $\Psi$ be this root datum. Let $v_1$, $v_2$ be two places
of $\rat$ of the form $4n+1$. Then $\Psi$ splits at $v_1$ and $v_2$.
Let $\rD$ be a quaternion algebra over $\rat$ corresponding to $1/2$
in $\rat/\ent\simeq\rH^{2}(\rat_{v_i},\gm_m)$ for $i=1,$ 2, and
corresponding to $0$ in the other places. Note that such a
quaternion exists by Brauer-Hasse-Noether's Theorem. Let $\rG$ be
$\mathrm{SL}_1(\rD)$. Since $\rG$ has no outer form, there is an
orientation $u$ between $\Psi$ and $\rG$. Since both $\Psi$ and
$\rG$ are anisotropic over $\rat$, we have
$u(\Delta^{\circ}(\Psi))\supseteq\Delta^{\circ}(\rG)$. However, at
place $v_1$ and $v_2$, the root datum $\Psi$ splits but $\rG$ is
anisotropic, so $\fE(\rG,\Psi,u)(\rat_{v_1})=\emptyset$. Therefore,
$\fE(\rG,\Psi,u)(\rat)=\emptyset$ and Theorem~\ref{1.10} does not
hold over $\rat$.
\end{example}
\end{subsection}
\begin{subsection}{Applications-the problems to embed an \'etale algebra in a central simple algebra with respect to involutions}
Let $K$ be a field, $(\rE,\sigma)$ be an \'etale $K$-algebra with
the involution $\sigma$, and $(\rA,\tau)$ be a central simple
algebra over $K$ with the involution $\tau$. Assume
$\sigma\mid_{K}=\tau\mid_{K}$. Let $k=K^{\sigma}$. From now on, we
assume that $k$ is a global field of characteristic different from
2. Let $\Omega_k$ be the set of all places of $k$. Fix a separable
closure $k^s$ of $k$. Let $\sG=\Gal(k^s/k)$ and
$\sG_v=\Gal(k^s_v/k_v)$ for $v\in\Omega_k$. Let
$\rT=\rU(\rE,\sigma)^{\circ}$, and $\rG=\rU(\rA,\tau)^{\circ}$. Note
that by the definition of $\rU(\rE,\sigma)^{\circ}$,
$\rT=\rR_{\rE^{\sigma}/k}(\rR_{\rE/\rE^{\sigma}}^{(1)}(\gm_m))$. We keep all the notation defined in section 1.3.

In the paper \cite{PR1}, Prasad and Rapinchuk consider the
local-global principle for the $K$-embedding from $(\rE,\sigma)$
into $(\rA,\tau)$. As we have mentioned in Theorem~\ref{1.4}, the
local-global principle for the existence of $k$-embeddings from
$(\rE,\sigma)$ into $(\rA,\tau)$ is equivalent to the local-global
principle for the existence of $k$-points of $\fE(\rG,\Psi)$. Here,
we will reduce the original problem to the existence of $k$-points
of oriented embedding functors, and prove that the local-global
principle holds in certain cases by computing the Shafarevich group
$\sha(k,\sico(\rT))$.
\begin{subsubsection}{Symplectic involutions}
For $\tau$ symplectic, $\Psi$ and $\rG$ are semisimple simply
connected of type $C_n$, which is the first case in
Theorem~\ref{2.12}, so we just restate the result as the following:
\begin{prop}~\label{2.4}
If $\tau$ is symplectic, then the local-global principle holds for
the existence of $K$-embeddings of $(\rE,\sigma)$ into $(\rA,\tau)$.
\end{prop}
\end{subsubsection}
\begin{subsubsection}{Orthogonal involutions}
Throughout this subsection, for an \'etale algebra $\rF$ over $K$,
we let $\rM_{\rF}$ be the character group of the torus
$\rR_{\rF/K}(\gm_m)$ and let $\rJ_{\rF}$ be the character group of
the torus $\rR^{(1)}_{\rF/K}(\gm_m)$. Note that $K=k$ when $\tau$ is
an orthogonal involution.

\begin{paragraph}{The case where the degree of A is odd}
Let us consider the case where $\tau$ is orthogonal, and
$\rA=\bM_{2n+1}(K)$. In this case, the corresponding group $\rG$ is
adjoint of type $B_n$, so there is no outer automorphisms. By
Theorem~\ref{1.4} and Proposition~\ref{2.3}, to prove the
local-global principle for the $K$-embeddings here, it suffices to
prove that $\sha^2(K,\sico(\rT))$ vanishes. Note that in this case,
$\rE=K\times\rE'$ and $\sigma$ acts trivially on the component $K$,
so
$\rT=\rR_{\rE'^{\sigma}/K}(\rR_{\rE'/\rE'^{\sigma}}^{(1)}(\gm_m))$.

Let $\rE'^{\sigma}=\overset{r}{\underset{i=1}{\prod}}{\rF_i}$, where
$\rF_i$ is a field over $K$ for all $i$. Let $d=(d_1,...,d_r)$ be an
element in $\rE'^{\sigma}$ such that
$\rE'=\rE'^{\sigma}[x]/(x^2-d)=\overset{r}{\underset{i=1}{\prod}}{\rF_i}[x]/(x^2-d_i)$.
Let $\rE_i={\rF_i}[x]/(x^2-d_i)$ for all $i$ and $\rE_{i,v}$ (resp.
$\rF_{i,v}$) be $\rE_i\otimes_K K_v$ (resp. $\rF_i\otimes_K K_v$)
for all $v\in\Omega_K$.

\begin{theo}~\label{2.5}
Suppose $\tau$ is orthogonal, and $\rA=\bM_{2n+1}(K)$. If there is a
place $v\in\Omega_K$ such that the following condition holds:
\begin{center}
for all $i$, $d_i\in\rF_i^{{\times}^{2}}$ if and only if
$d_i\in(\rF_{i,v})^{{\times}^{2}}$.
\end{center}
 Then the local-global
principle for the existence of $K$-embeddings from $(\rE,\sigma)$
into $(\rA,\tau)$ holds.
\end{theo}

We start with some calculations:
\begin{lemma}~\label{2.6}
$\sico(\rT)=\rR^{(1)}_{\rE'/K}(\gm_m)/\rR^{(1)}_{\rE'^{\sigma}/K}(\gm_m)$.
\end{lemma}
\begin{proof}
Consider the exact sequence over $\rE'^{\sigma}$:
$$\xymatrix {1\ar[r]&\gm_m\ar[r]&\rR_{\rE'/\rE'^{\sigma}}(\gm_m)\ar[r]&\rR^{(1)}_{\rE'/\rE'^{\sigma}}(\gm_m)\ar[r]& 1},$$
where the map from $\rR_{\rE'/\rE'^{\sigma}}(\gm_m)$ to
$\rR^{(1)}_{\rE'/\rE'^{\sigma}}(\gm_m)$ sends $x$ in
$\rR_{\rE'/\rE'^{\sigma}}(\gm_m)(R)$ to $x/\sigma(x)$, for any
$K$-algebra $R$. Let us take the Weil restriction of the above
sequence over $K$. Then we get the exact sequence:
\begin{center}
$(1)$ $\xymatrix
{1\ar[r]&\rR_{\rE'^{\sigma}/K}(\gm_m)\ar[r]&\rR_{\rE'/K}(\gm_m)\ar[r]&\rT\ar[r]&
1}.$
\end{center}
Let $\rM$ (resp. $\rP$) be the character group of $\rT$ (resp.
$\sico(\rT)$).

First, suppose that $(\rE',\sigma)$ is split. Then
$\Aut(\rE',\sigma)=(\ent/2\ent)^{n}\rtimes\rS_n$ and there is a
basis $\{e_i\}_i$ of $\rM$ such that the $\rS_n$-part of
$\Aut(\rE',\sigma)$ acts on $\{e_i\}_i$ by permuting the indices and
$(\ent/2\ent)^{n}$ acts on $\{e_i\}_i$ by change the sign of $e_i$.
In this case, $\rP=\rM+\frac{1}{2}(e_1+...+e_n)$. We choose a basis
$\{\varepsilon_i, \epsilon_i\}_{i=1}^n$ (resp. $\{h_i\}_{i=1}^n$) of
$\rM_{\rE'}$ (resp. $\rM_{\rE'^{\sigma}}$) on which
$\Aut(\rE',\sigma)$ acts as the following: $\rS_n$ permutes the
indices $i$ and $(\ent/2\ent)^{n}$ exchanges $\varepsilon_i,
\epsilon_i$ (resp. $(\ent/2\ent)^{n}$ acts trivially on $h_i$).

The we have the following exact sequence corresponding to $(1)$:
\begin{center}
$\xymatrix{ 0\ar[r]&\rM\ar[r]^{\imath}
 &\rM_{\rE'}\ar[r]^{\jmath} &\rM_{\rE'^{\sigma}}\ar[r]& 0}$,
\end{center}
where $\imath$ maps $e_i$ to $\varepsilon_i-\epsilon_i$ and $\jmath$
maps $\varepsilon_i$, $\epsilon_i$ to $h_i$. Consider the map
$\ol{\imath}$ from $\rM$ to $\rJ_{\rE'}$ induced by $\imath$. Then
$\ol{\imath}(e_1+...+e_n)=2(\ol{\varepsilon_1}+...+\ol{\varepsilon_n})$,
where $\ol{\varepsilon}_i$ is the image of ${\varepsilon_1}$ in
$\rJ_{\rE'}$. Hence $\ol{\imath}$ induces a map from $\rP$ to
$\rJ_{\rE'}$ and we have the following exact sequence:
\begin{center}
$(2)$ $\xymatrix{ 0\ar[r]&\rP\ar[r]
  &\rJ_{\rE'}
\ar[r] &\rJ_{\rE'^{\sigma}}\ar[r] & 0.}$
\end{center}
Since all the maps constructed are equivariant under
$\Aut(\rE',\sigma)$, we conclude
$\sico(\rT)=\rR^{(1)}_{\rE'/K}(\gm_m)/\rR^{(1)}_{\rE'^{\sigma}/K}(\gm_m)$.
\end{proof}

Now we use the above lemma to compute $\sha^2(K,\sico(\rT))$.
\begin{proof}[Proof of Theorem~\ref{2.5}]
Keep all the notations in Lemma~~\ref{2.6}.  By the Poitou-Tate
duality (ref.~\cite{NSW} Chap. VIII, Thm. 8.6.9), we have
$$\sha^2(K,\sico(\rT))\simeq\sha^1(K,\rP)^{\ast}.$$ Hence,  it is
enough to show $\sha^1(K,\rP)=0$.

 From the
exact sequence (2) in the proof of Lemma~\ref{2.6}, we derive the
commutative exact diagram:
\begin{center}

 $\xymatrix{
 & & 0 \ar[d] \ar[r] & 0 \ar[d]
  \\
 & 0\ar[d]\ar[r]& \ent \ar[d] \ar[r]^{\times 2} & \ent \ar[d]
  \\
  0\ar[r]&\rM^{\sG}\ar[d]\ar[r] &\rM_{\rE'}^{\sG}\ar[d] \ar[r] & \rM_{\rE'^{\sigma}}^{\sG} \ar[d]
  \\
0\ar[r]&\rP^{\sG}\ar[r]
  &\rJ_{\rE'}^{\sG}\ar[d]
\ar[r] &\rJ_{\rE'^{\sigma}}^{\sG}\ar[r]\ar[d]&
\rH^1(K,\rP)\ar[r]&\rH^1(K,\rJ_{\rE'})\ar[r]&...\\
&&\rH^{1}(K,\ent)=0& \rH^{1}(K,\ent)=0}$
\end{center}

Again,
$\sha^1(K,\rJ_{\rE'})=\sha^2(K,\rR^{(1)}_{\rE'/K}(\gm_m))^{\ast}$.
By Hilbert Theorem 90, we have that
$\rH^2(K,\rR^{(1)}_{\rE'/K}(\gm_m))$ injects into
$\rH^2(K,\rR_{\rE'/K}(\gm_m))$. However,
$\sha^2(K,\rR_{\rE'/K}(\gm_m))$ vanishes, so does
$\sha^1(K,\rJ_{\rE'})$. Let $x\in\sha^1(K,\rP)$. Since
$\sha^1(K,\rJ_{\rE'})=0$, we have $y\in\rJ_{\rE'^{\sigma}}^{\sG}$
mapped to $x$.

Let $\rI=\{1,2,...,r\}$. Let $\rI_1$ be the subset of $\rI$ such
that $i\in\rI_1$ if and only if $d_i\in {\rF_i^{\times}}^2$. Let
$\rI_2=\rI\setminus\rI_1$. Note that
$\rM_{\rE'^{\sigma}}^{\sG}=\overset{r}{\underset{i=1}{\bigoplus}}\rM_{\rF_i}^{\sG}$
and
$\rM_{\rE'}^{\sG}=\overset{r}{\underset{i=1}{\bigoplus}}\rM_{\rE_i}^{\sG}$.
For $i\in\rI_1$, $\rE_i\simeq\rF_i\times\rF_i$, so
$\rM_{\rE_i}^{\sG}\simeq\rM_{\rF_i}^{\sG}\bigoplus\rM_{\rF_i}^{\sG}$
and $\rM_{\rE_i}^{\sG}$ is mapped surjectively onto
$\rM_{\rF_i}^{\sG}$.

Let $\gamma_i$ be a basis of $\rM_{\rF_i}^{\sG}$. For $i\in\rI_2$,
we have the following observation:
\begin{lemma}\label{2.10}
For $i\in\rI_2$ and $y\in\rM_{\rF_i}^{\sG}$, $y$ is in the image of
$\rM_{\rE_i}^{\sG}$ if and only if the coefficient of $\gamma_i$ in
$y$ is even.
\end{lemma}
\begin{proof}[Proof of Lemma~\ref{2.10}]
Since $\rE_i$ is a field over $\rF_i$ with degree 2 for $i\in\rI_2$,
the module $\rM_{\rE_i}^{\sG}$ is of rank 1 and is generated by
$\underset{\epsilon_j,\varepsilon_j\in\rM_{\rE_i}}{\sum}(\epsilon_j+\varepsilon_j)$.
Since the element
$\underset{\epsilon_j,\varepsilon_j\in\rM_{\rE_i}}{\sum}(\epsilon_j+\varepsilon_j)$
is mapped to $2\gamma_i$ in $\rM_{\rF_i}^{\sG}$, the lemma then
follows.
\end{proof}

We return to the proof of Theorem~\ref{2.5}. Since
$\rM_{\rE'^{\sigma}}^{\sG}$ is mapped surjectively onto
$\rJ_{\rE'^{\sigma}}^{\sG}$, $\rJ_{\rE'^{\sigma}}^{\sG}$ is
generated by $\gamma_i$'s. Let $\ol{\gamma}_i$ be the image of
$\gamma_i$ in $\rJ_{\rE'^{\sigma}}$. Let
$y=\overset{r}{\underset{i=1}{\sum}}a_i\ol{\gamma}_i$. If for all
$i\in\rI_2$, the $a_i$'s have the same parity, then we can find
$z=\overset{r}{\underset{i=1}{\sum}}b_i{\gamma}_i$, which is a
lifting of y in $\rM_{\rE'^{\sigma}}^{\sG}$, such that $b_i$ is even
for any $i\in\rI_2$. Then by Lemma~\ref{2.10}, $z$ is in the image
of $\rM_{\rE'}^{\sG}$ and hence $y$ is in the image of
$\rJ_{\rE'}^{\sG}$. So it is enough to prove that for all
$i\in\rI_2$, the $a_i$'s have the same parity.

Now, let $v$ be a place of $K$ such that for all $i$,
$d_i\in\rF_i^{{\times}^{2}}$ if and only if
$d_i\in(\rF_{i,v})^{{\times}^{2}}$. Since $x$ is in $\sha^1(K,\rP)$,
$y$ is in the image $\rJ_{\rE'_v/K_v}^{\sG_v}$. For each
$i\in\rI_2$, since $d_i$ is not a square in $\rF_{i,v}$, there is
some $h_{j(i)}\in\rM_{\rF_i}$ such that there exists
$\tau_{j(i)}\in\sG_v$ which exchanges $\epsilon_{j(i)}$ and
$\varepsilon_{j(i)}$.  Therefore, for all $i\in\rI_2$, the
coefficients of $\ol{h}_{j(i)}$'s in the expression of $y$ have the
same parity. Since the coefficient of $\ol{h}_{j(i)}$ in $y$ is
$a_i$, we know that all $a_i$'s have the same parity for
$i\in\rI_2$. By Lemma~\ref{2.10}, $y$ is in the image of
$\rJ_{\rE'}^{\sG}$, which means $\sha^1(K,\rP)=0$.
\end{proof}

\begin{remark}~\label{2.8}
A special case of the above theorem is when there is a place $v$
such that $\sico(\rT)$ is anisotropic over $K_v$.  We now show that
$\sico(\rT)$ is anisotropic over $K_v$ implies all
$d_i\not\in{(\rF_{i,v})}^{{\times}^2}$. To see this, we note that in
our case here, $\sico(\rT)$ is anisotropic if and only if $\rT$ is
anisotropic. If there is $d_i\in{(\rF_{i,v})}^{{\times}^2}$, then
$\rM_{E_{i,v}/K_v}=\rM_{F_{i,v}/K_v}\oplus\rM_{F_{i,v}/K_v}$. Let
$\alpha$ be a nontrivial element in $\rM_{F_{i,v}/K_v}^{\sG_v}$.
Then $(\alpha,-\alpha)\in\rM_{E_{i,v}/K_v}^{\sG_v}$ and it is in the
image of $\rM$, which means $\rM^{\sG_v}$ is nontrivial and
contradicts to the condition that $\rT$ is anisotropic over $K_v$.
Therefore, $d_i\not\in{(\rF_{i,v}})^{{\times}^2}$ for all $i$.
\end{remark}
\end{paragraph}
\begin{paragraph}{The case where the degree of $\rA$ is even}
Throughout this paragraph, we let $\rA$ be $\bM_{2n}(K)$, or
$\bM_{n}(\rD)$ with orthogonal involution $\tau$, where $\rD$ is a
quaternion division algebra over $K$. In this case, the
corresponding group $\rG$ is semisimple of type $D_n$, and
$\ul{\Isomext}(\Psi,\rG)$ satisfies the local-global principle.

For $\rA$ satisfying one of the conditions in Theorem~\ref{2.7}, we
first show that $\fE(\rG,\Psi)(K_v)$ is nonempty implies that
$\fE(\rG,\Psi,u)(K_v)$ is nonempty for any orientation $u$. (see
Lemma~\ref{2.9}.) Then we prove that the local global principle
holds for the oriented embedding functor $\fE(\rG,\Psi,u)$. By
Theorem~\ref{1.4} and Proposition~\ref{1.15}, we get the
local-global principle for the existence of $K$-embeddings from
$(\rE,\sigma)$ into $(\rA,\tau)$.

We first fix some notations. Let
$\rE^{\sigma}=\overset{r}{\underset{i=1}{\prod}}{\rF_i}$, where the
$\rF_i$'s are  fields over $K$. Let $d=(d_1,...,d_r)$ be in
$\rE^{\sigma}$ and
$\rE=\rE^{\sigma}[x]/(x^2-d)=\overset{r}{\underset{i=1}{\prod}}{\rF_i}[x]/(x^2-d_i)$.
Let $\rE_i={\rF_i}[x]/(x^2-d_i)$, and $\rE_{i,v}$ (resp.
$\rF_{i,v}$) be $\rE_i\otimes_K K_v$ (resp. $\rF_i\otimes_K K_v$)
for all $v\in\Omega_K$.

\begin{theo}~\label{2.7}
Suppose that $\rA$ is equal to one of the following:
 \begin{itemize}
   \item [(1)] $\bM_{2n}(K)$, $n>1$.
   \item [(2)] $\bM_{2m+1}(\rD)$, where $\rD$ is a quaternion division algebra over $K$.
   \item [(3)] $\bM_{2m}(\rD)$, where $\rD$ is a quaternion division algebra over $K$, and at each place $v\in\Omega_K$, if $\rA$ is not split and the
   discriminant splits, then $\rE_v$ is not split over $\rE^{\sigma}_v$, i. e. $\rE_v\neq \rE^{\sigma}_v\times\rE^{\sigma}_v$.
 \end{itemize}
 If there is a place $v\in\Omega_K$
such that for all $i$, $d_i\in\rF_i^{{\times}^{2}}$ if and only if
$d_i\in(\rF_{i,v})^{{\times}^{2}}$, then the local-global principle
for the $K$-embedding of $(\rE,\sigma)$ into $(\rA,\tau)$ holds.
\end{theo}

First we prove the following lemma:
\begin{lemma}\label{2.9}
For $\rA$ satisfying one of the three conditions in
Theorem~\ref{2.7}, the existence of a $K_v$-point of $\fE(\rG,\Psi)$
implies the existence of a $K_v$-point of $\fE(\rG,\Psi,u)$ for any
$u\in\ul{\Isomext}(\Psi,\rG)(K_v)$.
\end{lemma}
\begin{proof}
Suppose that there is a $K_v$-point $f$ of $\fE(\rG,\Psi)$. By
Theorem~\ref{1.10}, there is an orientation $u'$ induced by $f$ such
that $u'(\Delta^\circ(\Psi_{K_v}))\supseteq
\Delta^\circ(\rG_{K_v})$.

According to the list of all possible Tits indices
(ref.~\cite{Tit}), if $\rA$ satisfies (1) or (2) in
Theorem~\ref{2.7}, then the Tits index of $\rG_{K_v}$ will be
symmetric under $\ul{\Autext}(\rG_{K_v})$. Therefore, for any
$u\in\ul{\Isomext}(\Psi,\rG)(K_v)$, we have that
$u(\Delta^\circ(\Psi_{K_v}))$ contains $\Delta^\circ(\rG_{K_v})$,
and again, by Theorem~\ref{1.10}, we have
$\fE(\rG,\Psi,u)(K_v)\neq\emptyset$.

Now assume that $\rA$ satisfies (3) in Theorem~\ref{2.7}. If over
$K_v$, $\rA$ is not split and the discriminant splits, then $\rG$ is
a non-split inner form over $K_v$. In this case, the possible Tits
indices of $\rG$ are symmetric except the following case:

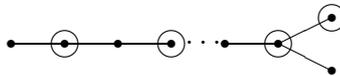
\begin{center}
\begin{picture}(120,20)
\put(00,0){\line(1,0){20}}\put(20,0){\line(1,0){20}}
\put(40,0){\line(1,0){20}} \put(65,0){\dots}
\put(80,0){\line(1,0){20}}
\put(100,0){\line(2,-1){20}}\put(100,0){\line(2,1){20}}
\put(0,0){\circle*{3}} \put(20,0){\circle*{3}}
\put(40,0){\circle*{3}} \put(60,0){\circle*{3}}
\put(80,0){\circle*{3}}\put(100,0){\circle*{3}}
\put(120,-10){\circle*{3}} \put(120,10){\circle*{3}}
\put(20,0){\circle{10}} \put(60,0){\circle{10}}
\put(100,0){\circle{10}} \put(120,10){\circle{10}}
\end{picture}
\end{center}
\smallskip
Suppose that $\Delta^{\circ}(\rG_{K_v})$ takes the above
nonsymmetric form. We will show that condition (3) in
Theorem~\ref{2.7} forces $\Delta^\circ(\Psi_{K_v})$ to be symmetric
under $\ul{\Autext}(\Psi_{K_v})$ in this case.

Consider the Dynkin diagram of $\Psi$:
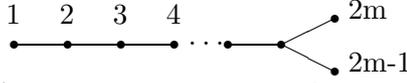
\begin{center}
\begin{picture}(120,20)
\put(00,0){\line(1,0){20}}\put(20,0){\line(1,0){20}}
\put(40,0){\line(1,0){20}} \put(65,0){\dots}
\put(80,0){\line(1,0){20}}
\put(100,0){\line(2,-1){20}}\put(100,0){\line(2,1){20}}
\put(0,0){\circle*{3}} \put(20,0){\circle*{3}}
\put(40,0){\circle*{3}} \put(60,0){\circle*{3}}
\put(80,0){\circle*{3}}\put(100,0){\circle*{3}}
\put(120,-10){\circle*{3}} \put(120,10){\circle*{3}}
\put(-3,8){\small{1}}\put(17,8){\small{2}} \put(37,8){\small{3}}
\put(57,8){\small{4}}\put(125,10){\small{2m}}\put(125,-10){\small{2m-1}}
\end{picture}
\end{center}
\smallskip
Suppose that $\rI=\Delta^\circ(\Psi_{K_v})$ is not symmetric under
$\ul{\Autext}(\Psi)$. Without loss of generality, we suppose that
the vertex 2m is not in $\rI$. Let $\rI'$ be the Dynkin subdiagram
with vertices 1,...,2m-1 which is of type $A_{2m-1}$. So
$\rI\subseteq\rI'$.

Since there is $f\in\fE(\rG,\Psi)(K_v)$, $\rG_{K_v}$ has a parabolic
subgroup $\rP_\rI$ with the type $\rI$ and $\rP_\rI$ contains
$f(\rT_{K_v})$ by Proposition~\ref{0.9}. Let $\rG_0$, $\Psi_0$ be
the split form of $\rG$ and $\Psi$ respectively
(Section.~\ref{s1.3.1.2}), and let $\rT_0$ be the split torus
determined by $\Psi_0$. Let $\rP_{0,\rI}$ be a parabolic subgroup of
$\rG_{0,K_v}$ with type $\rI$ and contains $\rT_{0,K_v}$.

Let $\sP_\rI$ (resp. $\sP_{0,\rI}$) be the subsheaf of the sheaf of
roots of $\Psi$ (resp. $\Psi_0$) determined by $\rP_\rI$ (resp.
$\rP_{0,\rI}$). Define $\rW_{0,\rI}=\rW(\Psi_{0,\rI})$ and
$\rW_{\rI}=\rW(\Psi_\rI)$ as we have done in the proof of
Theorem~\ref{1.10}. Define $\rW_{0,\rI'}=\rW(\Psi_{0,\rI'})$ in the
same way.

 Let $\Psi_0$ be
$(\rM_0,\rM^{\vee}_0,\rR_0,\rR^{\vee}_0)$, and $\{e_i\}_{i=1}^{2m}$
be a basis of $\rM_0$ such that $\rR_0$ is the set $\{\pm e_i\pm
e_j\}_{i<j}$, where the vertex $i$ corresponds to $e_i-e_{i+1}$ for
$i=1,..;,2m-1$, and the vertex $2m$ corresponds to
$e_{2m-1}+e_{2m}$. Let $\rS_n$ be the permutation group of $n$
elements. Then we have
$$\ul{\Aut}(\Psi_{0,K_v})(K_v)=(\ent/2\ent)^{2m}\rtimes\rS_{2m},$$
where $\rS_{2m}$ acts on $\rR_0$ by permuting the indices of
$\{e_i\}_{i=1}^{2m}$, and $(\ent/2\ent)^{2m}$ acts on $\rR_0$ by
exchanging the sign of $e_i$'s (~\cite{Bou}, Plan. IV). Under this
basis, $\rW_{0,\rI'}$ is just the permutation group of the set
$\{e_i\}_{i=1}^{2m}$. Therefore, the natural inclusion
$\imath_\rW:\rW_{0,\rI'} \ra\ul{\Aut}(\Psi_{0,K_v})$ sends
$w\in\rW_{0,\rI'}\simeq\rS_{2m}$ to
$(1,w)\in(\ent/2\ent)^{2m}\rtimes\rS_{2m}$.

Since  $\rG_{K_v}$ is an inner form of $\rG_{0,K_v}$,  there is an
orientation
$$\mu\in\ul{\Isomext}(\Psi_{0,K_v},\Phi(\rG_{K_v},f(\rT_{K_v})))(K_v).$$
The orientation $\mu$ together with $u'^{-1}$ gives an orientation
$\nu\in\ul{\Isomext}(\Psi_{0,K_v},\Psi_{K_v})(K_v)$.

We then define
$$\rQ=\ul{\Isomint}_{\nu}(\Psi_{0,K_v},\sP_{0,\rI};\Psi_{K_v},\sP_\rI)$$ as we have done in the proof of Proposition~\ref{1.10}.
Since $\rW_{0,\rI}\subseteq\rW_{0,\rI'}$, we can regard
$\rW_{0,\rI}$ as a subgroup of
$\{1\}\rtimes\rS_{2m}\subseteq\ul{\Aut}(\Psi_0)$ through
$\imath_\rW$. Since
$\Psi_{K_v}=\rQ\overset{\rW_{0,\rI}}{\wedge}\Psi_{0,K_v}$, by
Remark~\ref{1.5},
$(\rE_v,\sigma)\simeq\rQ\overset{\rW_{0,\rI}}\wedge(\rE_{0,v},\sigma_0)$.
Therefore, $\rE_v\simeq\rE'_v\times\rE'_v$ with $\sigma$ acts on
$\rE_v$ as the exchange of the two copies of $\rE'_v$, which
contradicts to the assumption (3) in Theorem~\ref{2.7}! Therefore,
$\rI$ is symmetric under $\ul{\Autext}(\Psi_{K_v})$ and we conclude
that $u(\Delta^\circ(\Psi_{K_v}))\supseteq \Delta^\circ(\rG_{K_v})$
for any orientation $u$. Again, by Theorem~\ref{1.10}, we have
$\fE(\rG,\Psi,u)(K_v)\neq\emptyset$, for any
$u\in\ul{\Isomext}(\Psi,\rG)(K_v)$.
\end{proof}

Next, we prove that the $\ul{\Autext}(\rG)$-torsor
$\ul{\Isomext}(\Psi,\rG)$ satisfies the local-global principle.
Namely,
\begin{lemma}\label{2.16}
Let $\rG$ (resp. $\Psi$) be the corresponding semisimple group
(resp. root datum) defined by $\rA$ (resp. $\rE$). If the
$\ul{\Autext}(\rG)$-torsor $\ul{\Isomext}(\Psi,\rG)$ has a
$K_v$-point at each place $v\in\Omega_K$, then
$\ul{\Isomext}(\Psi,\rG)$ has a $K$-point.
\end{lemma}
\begin{proof}
If $\rA$ is not equal to $\bM_{8}(K)$ or $\rA=\bM_{4}(\rD)$, then
$\ul{\Autext}(\rG)$ is $(\ent/2\ent)_K$, so the local-global
principle for $\ul{\Isomext}(\Psi,\rG)$ holds in this case.

For $\rG$ an inner form, the outer automorphism group
$\ul{\Autext}(\rG)$ is the symmetric group $\rS_3$. Therefore, to
prove the local-global principal for the $\rS_3$-torsor
$\ul{\Isomext}(\Psi,\rG)$, we only need to prove
$\sha^1(K,\rS_3)=0$. Consider the exact sequence:
$$ (1)\ \ 0\ra\ent/3\ent\ra\rS_3\ra\ent/2\ent\ra 0.$$
From the above exact sequence, we get the following exact sequence
$$0\ra\ent/3\ent\ra\rS_3\ra\ent/2\ent\ra\rH^{1}(K,\ent/3\ent)\ra\rH^{1}(K,\rS_3)\ra\rH^{1}(K,\ent/2\ent).$$
Since the map from $\rS_3$ ro $\ent/2\ent$ is surjective, we have
$$0\ra\rH^{1}(K,\ent/3\ent)\ra\rH^{1}(K,\rS_3)\ra\rH^{1}(K,\ent/2\ent).$$
However, the group $\sha^{1}(K,\ent/2\ent)=0$, so the set
$\sha^{1}(K,\rS_3)$ is in the image of $\rH^{1}(K,\ent/3\ent)$.

Recall that $\sG$ is the absolute Galois group of $K$. Note that
$\rH^{1}(K,\ent/3\ent)=\Hom_{gr}(\sG,\ent/3\ent)$, where
$\Hom_{gr}(\sG,\ent/3\ent)$ is the set of continuous homomorphisms
from $\sG$ to $\ent/3\ent$. Suppose
$\alpha\in\Hom_{gr}(K,\ent/3\ent)$ is mapped into
$\sha^{1}(K,\rS_3)$. Then since the symmetric group $\rS_n$ is
surjective to the group $\ent/2\ent$ for each place $v\in\Omega_K$,
we have
$$0\ra\rH^{1}(K_v,\ent/3\ent)\ra\rH^{1}(K_v,\rS_3)\ra\rH^{1}(K_v,\ent/2\ent).$$
Therefore, the homomorphism $\alpha$ is in $\sha^{1}(K,\ent/3\ent)$.
Now we claim that $\sha^{1}(K,\ent/3\ent)$ is trivial, i.e. for each
$\alpha\in\Hom_{gr}(K,\ent/3\ent)$, if $\alpha$ is in
$\sha^{1}(K,\ent/3\ent)$, then $\alpha$ is the trivial homomorphism.
Suppose that $\alpha$ is not the trivial homomorphism. Let $\sH$ be
the kernel of $\alpha$. Let $L=({K^s})^{\sH}$. Then $L$ is a Galois
extension of $K$ with Galois group $\ent/3\ent$ and we can regard it
as a $\ent/3\ent$-torsor. Since the homomorphism $\alpha$ is in
$\sha^{1}(K,\ent/3\ent)$, $L_v$ is split completely over $K_v$ for
each place $v\in\Omega_{K}$. This contradicts Chebotarev's density
Theorem! Therefore, $\alpha$ is the trivial homomorphism and
$\sha^{1}(K,\ent/3\ent)$ is trivial. Since $\sha^{1}(K,\rS_n)$ is in
the image of $\sha^{1}(K,\ent/3\ent)$, $\sha^{1}(K,\rS_n)$ is also
trivial.

For $\rA=\bM_{4}(\rD)$ and $\rG$  an outer form, let $L$ be the
splitting field of the discriminant of $\rA$. We choose a splitting
$z$ from $\ent/2\ent$ to $\rS_3$ and we twist the sequence $(1)$ by
$z$. Since $z$ acts on $\ent/3\ent$ as $-1$, we have the exact
sequence:
$$(2)\ \ 0\ra\rR^{(1)}_{L/K}(\ent/3\ent)\ra\ _{z}(\rS_3)\ra\ent/2\ent\ra 0,$$
where we regard $\ent/3\ent$ as a constant group scheme. Note that
$z$ is invariant under the twisting because $\ent/2\ent$ is
commutative. Therefore, the sequence still splits. Consider the
exact sequence derived from (2):
\begin{center}
$ 0\ra \rR^{(1)}_{L/K}(\ent/3\ent)(K)\ra\ _{z}(\rS_3)(K)\ra
(\ent/2\ent)(K)\ra \rH^{1}(K,\rR^{(1)}_{L/K}(\ent/3\ent))
\ra\rH^{1}(K,\ _{z}(\rS_3))
\ra\rH^{1}(K,\rR^{(1)}_{L/K}(\ent/2\ent)).$
\end{center}
Since the sequence (2) splits, $_{z}(\rS_3)(K)$ is mapped onto
$(\ent/2\ent)(K)$. Hence we have the exact sequence
$$0\ra\rH^{1}(K,\rR^{(1)}_{L/K}(\ent/3\ent))
\ra\rH^{1}(K,\ _{z}(\rS_3))
\ra\rH^{1}(K,\rR^{(1)}_{L/K}(\ent/2\ent)).$$  Since
$\sha^{1}(K,\ent/2\ent)=0$, the set $\sha^{1}(K,\ _{z}(\rS_3)(K))$
is in the image of  $\rH^{1}(K,\rR^{(1)}_{L/K}(\ent/3\ent)).$ Again,
because the exact sequence (2) splits, for each place
$v\in\Omega_K$, we have
$$0\ra\rH^{1}(K_v,\rR^{(1)}_{L_v/K_v}(\ent/3\ent))
\ra\rH^{1}(K_v,\ _{z}(\rS_3))
\ra\rH^{1}(K_v,\rR^{(1)}_{L_v/K_v}(\ent/2\ent)).$$ Therefore,
$\sha^{1}(K,\ _{z}(\rS_3)(K))$ is in the image of
$\sha^{1}(K,\rR^{(1)}_{L/K}(\ent/3\ent))$. Now, we only need to
prove $\sha^{1}(K,\rR^{(1)}_{L/K}(\ent/3\ent))=0$. By Shapiro's
Lemma, we have
$\sha^{1}(K,\rR_{L/K}(\ent/3\ent))=\sha^{1}(L,\ent/3\ent).$ As we
have proved above, the group $\sha^{1}(L,\ent/3\ent)=0$, so
$\sha^{1}(K,\rR_{L/K}(\ent/3\ent))=0$. Consider the following exact
sequence
$$0\ra\rR^{(1)}_{L/K}(\ent/3\ent)\ra\rR_{L/K}(\ent/3\ent)\xrightarrow{\Nr}\ent/3\ent\ra 0.$$
In our case, the norm map $\Nr$ from
$\rR_{L/K}(\ent/3\ent)(K)=\ent/3\ent$ to $\ent/3\ent$ is just the
multiplication by 2. Hence the norm map $\Nr$ is a surjective map
from $\rR_{L/K}(\ent/3\ent)(K)$ to $(\ent/3\ent)(K)$. Therefore, the
map from  $\rH^{1}(K,\rR^{(1)}_{L/K}(\ent/3\ent))$ to
$\rH^{1}(K,\rR_{L/K}(\ent/3\ent))$ is injective. Hence,
$\sha^{1}(K,\rR^{(1)}_{L/K}(\ent/3\ent))$ is also trivial.
Therefore, in both cases, the local-global principle for
$\ul{\Isomext}(\Psi,\rG)$ holds.

\end{proof}

Now we have all the ingredients to prove Theorem~\ref{2.7}.
\begin{proof}[Proof of Theorem~\ref{2.7}]
Suppose that $(\rE\otimes K_v,\sigma\otimes \id_{K_v})$ can be
embedded into $(\rA\otimes K_v,\tau\otimes\id_{K_v})$ over $K_v$ for
each $v\in\Omega_K$, i.e., $\rE(\rG,\Psi)(K_v)\neq\emptyset$ for
each $v\in\Omega_K$. Then we have
$\ul{\Isomext}(\Psi,\rG)(K_v)\neq\emptyset$ for each place $v$. By
Lemma~\ref{2.16}, we can fix an orientation $u$.

By Lemma~\ref{2.9}, the oriented embedding functor $\fE(\rG,\Psi,u)$
has a $K_v$-point for each $v\in\Omega_K$.  By
Proposition~\ref{2.3}, the only obstruction for $\fE(\rG,\Psi,u)$ to
satisfy the local-global principle lies in $\sha^2(K,\sico(\rT))$.
As the proof of Theorem~\ref{2.5} shows, $\sha^2(K,\sico(\rT))$
vanishes if there is a place $v\in\Omega_K$ such that for all $i$,
$d_i\in\rF_i^{{\times}^{2}}$ if and only if $d_i\in(\rF_{i}\otimes_K
K_v)^{{\times}^{2}}$. Therefore, the oriented embedding functor
$\rE(\rG,\Psi,u)$ satisfies the local-global principle in this case.
By Theorem~\ref{1.4} and Proposition~\ref{1.15}, the local-global
principle for the existence of $K$-embeddings from $(\rE,\sigma)$
into $(\rA,\tau)$ holds.
\end{proof}

In the following, we provide an example when the local-global
principle for the embedding functor fails.
\begin{example}\label{2.13}
Let K be $\rat(\sqrt{-1})$, and $\rF=K[x]/(x^2-3)$. Let
$\rE'=\rF\times\rF\times\rF$ and $\rE=\rE'\times\rE'$. Let $\sigma$
be the $K$-automorphism of $\rE$ which exchanges the two copies of
$\rE'$. Then $\sigma$ is an involution and $\rE^{\sigma}\simeq\rE'$.
With the notations defined in Section~\ref{s1.3}, we know that the
right $\ul{\Aut}(\rE_0,\sigma_0)$-torsor
$\ul{\Isom}((\rE_0,\sigma_0),(\rE,\sigma))$ defines a class in
$\rH^{1}(K, \rS_6)$, where $\rS_6$ is contained in the Weyl group of
$\Psi_0$ (\cite{Bou} Plan. IV).  Let $\Psi$ be the corresponding
root datum. Since $\Psi$ comes from a class of $\rH^{1}(K, \rS_6)$,
$\Psi$ is an inner form of $\Psi_0$.

Let us fix four places of $K$ such that $\rF$ is not split over
$K_v$. For example, we can take a place $v$ which corresponds to a
prime number of the form $7+12l$, where $l$ is a positive integer.
By Gauss reciprocity, $x^2-3$ is not split at $v$. Let $v_1$,...,
$v_4$ be the four places mentioned above. At these places, the
corresponding Tits index of $\Psi$ is the following
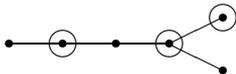
\begin{center}
\begin{picture}(120,20)
\put(00,0){\line(1,0){20}}\put(20,0){\line(1,0){20}}
\put(40,0){\line(1,0){20}}\put(60,0){\line(2,1){20}}\put(60,0){\line(2,-1){20}}
\put(0,0){\circle*{3}} \put(20,0){\circle*{3}}
\put(40,0){\circle*{3}} \put(60,0){\circle*{3}}
\put(80,10){\circle*{3}}\put(80,-10){\circle*{3}}
\put(20,0){\circle{10}} \put(60,0){\circle{10}}
\put(80,10){\circle{10}}
\end{picture}
\end{center}

Consider the following central isogeny:
\begin{center}
$\xymatrix@C=0.5cm{
  1 \ar[r] &  \mu_2\times\mu_2 \ar[r] &  \spin_6 \ar[r] & \pso_6 \ar[r] &
  1}$.
\end{center}
Since we have no real places, by~\cite{San} Corollary 4.5, we have
$$\rH^{1}(K,\pso_6)\xrightarrow{\sim}\rH^{2}(K,\mu_2\times\mu_2).$$
Also at each finite place $v$, we have
$$\rH^{1}(K_v,\pso_6)\xrightarrow{\sim}\rH^{2}(K_{v},\mu_2\times\mu_2)\simeq\ent/2\ent\times\ent/2\ent.$$
(ref.~\cite{K} Chap. IV, Thm. 1 and Thm. 2).
 Let $[\xi_i]$ be the class in $\rH^{1}(K_{v_i},\pso_6)$  corresponding to $(1,0)$ in $\ent/2\ent\times\ent/2\ent$, for
 i=1, 2. Let $[\xi_i]$ be the class in $\rH^{1}(K_{v_i},\pso_6)$ corresponding to $(0,1)$ in $\ent/2\ent\times\ent/2\ent$, for
 i=3, 4. For the other places $v\in\Omega_K\setminus\{v_1,\ v_2,\ v_3,\ v_4\}$, we let $[\xi_v]$ in $\rH^{1}(K_{v},\pso_6)$
correspond to $(0,0)\in\ent/2\ent\times\ent/2\ent$. By
Brauer-Hasse-Noether Theorem, we know that there exists a class
$[\xi]$ in $\rH^{1}(K,\pso_6)$ such that the image of $[\xi]$ in
$\rH^{1}(K_{v},\pso_6)$ is $[\xi_v]$ for each $v\in\Omega_K$.

Choose a cocycle $\zeta$ which represents the class $[\xi]$. Let
$\rG$ be the $K$-form of $\rG_0$ twisted by $\zeta$. Since $\rG$ and
$\Psi$ are  inner forms of $\rG_0$ and $\Psi_0$ respectively, we can
fix an orientation $u$ of $\Psi$ with respect to $\rG$. Without loss
of generality, we can choose the orientation $u$ such that
$u(\Delta^{\circ}(\Psi_{K_{v_1}}))\supseteq
(\Delta^{\circ}(\rG_{K_{v_1}}))$. Note that there is no orientation
$u'$ such that  $u'(\Delta^{\circ}(\Psi_{K_{v}}))$ contains
$\Delta^{\circ}(\rG_{K_{v}})$ for both $v=v_1$  and $v=v_3$.

For each place $v\in\Omega_K\setminus\{v_3,\ v_4\}$, we have
$u(\Delta^{\circ}(\Psi_{K_{v}}))\supseteq
\Delta^{\circ}(\rG_{K_{v}})$, so by Theorem~\ref{1.10}, there is a
$K_v$ point of $\fE(\rG,\Psi,u)$. On the one hand, for the place
$v\in\{v_3,\ v_4\}$, by Theorem~\ref{1.10}, $\fE(\rG,\Psi,u)(K_v)$
is empty. Therefore, the embedding functor $\fE(\rG,\Psi,u)$ has no
$K$-points. For the same reason, we conclude
$\fE(\rG,\Psi,u')(K)=\emptyset$ for the other orientation $u'$.
Hence $\fE(\rG,\Psi)$ has no $K$-point. However, at each place $v$,
we can always find an orientation
$u_v\in\ul{\Isomext}(\Psi,\rG)(K_v)$ such that
$u_v(\Delta^{\circ}(\Psi_{K_{v}}))\supseteq
\Delta^{\circ}(\rG_{K_{v}})$, so the embedding functor
$\fE(\rG,\Psi)$ has a $K_v$-point for each place $v$. Therefore, the
local-global principle fails in this case.

\end{example}
\end{paragraph}
\end{subsubsection}
\begin{subsubsection}{Involutions of the second kind}
In this section,  $\rA$ is of degree $n$ over $K$ and $\tau$ is of
the second kind. The corresponding reductive group $\rG$ is of type
$A_{n-1}$. In this case, $K$ and $k$ are no longer the same.

Recall that $i_{\rT}:\rR_{K/k}^{(1)}(\gm_{m,K})\ra\rT$ (resp.
$i_{\rG}:\rR_{K/k}^{(1)}(\gm_{m,K})\ra\rG$) denote the embedding
defined by the $K$-structure morphism of $\rE$ (resp. $\rA$). We
first interpret the $K$-morphism condition into an orientation.
Namely, we  show that the following are equivalent:
\begin{enumerate}
  \item  A $k$-embedding $f$ is a $K$-embedding.
  \item  $f\circ i_\rT=i_\rG$.
  \item  $f$ is a $k$-point of $\fE(\rG,\Psi,u)$ for
some particular orientation $u$
\end{enumerate}

Using the following lemma, we can concretely define the orientation
$u$ mentioned above.
\begin{lemma}\label{2.11}
Let $\rZ=\rR^{(1)}_{K/k}(\gm_m)$ and $\Psi_\rZ=\Phi(\rZ,\rZ)$. Then
\begin{itemize}
  \item [(1)] The natural homomorphism from $\ul{\Aut}(\Psi)$ to $\ul{\Aut}(\rad(\Psi))$ induces an isomorphism from
  $\ul{\Autext}(\Psi)$ to $\ul{\Aut}(\Psi_\rZ)$.
  \item [(2)] $\ul{\Isomext}(\Psi,\rG)$ is a trivial $\ul{\Aut}(\Psi_\rZ)$-torsor.
\end{itemize}
\end{lemma}
\begin{proof}
Let $j_\rT$  be the homomorphism from the character group of $\rT$
to the character group of $\rZ$ induced by $i_\rT$. Then
$j_\rT$ induces an isomorphism between $\rad(\Psi)$ and $\Psi_\rZ$
and we have a canonical way to identify $\ul{\Aut}(\rad(\Psi))$ and
$\ul{\Aut}(\Psi_\rZ)$. Consider the natural morphism from
$\ul{\Aut}(\Psi)$ to $\ul{\Aut}(\rad(\Psi))$. Since the Weyl group
acts trivially on $\ul{\Aut}(\rad(\Psi))$, we have a natural
morphism $\eta$ from $\ul{\Autext}(\Psi)$ to
$\ul{\Aut}(\rad(\Psi))$. Note that since $\rZ$ is a torus of
dimension one, $\ul{\Aut}(\Psi_\rZ)\simeq\ent/2\ent$. Hence,
$\ul{\Aut}(\rad(\Psi))\simeq\ent/2\ent$.

To prove $\eta$ is an isomorphism, we only need to check it over
$k^{s}$, so we can assume that $\Psi$ is split and of type
$(\rM,\rM^{\vee},\rR,\rR^{\vee})$.

We first prove the injectivity of $\eta$. By the definition of
$\Psi$, we can find a basis $\{e_i\}_{i=1,...,n}$ of $\rM$ such that
$\Delta=\{e_i-e_{i+1}\}_{i=1,...,n-1}$ is a system of simple roots
of $\rR$ (ref.~\cite{Bou}, Plan. I). By Proposition~\ref{0.5},
$\Autext(\Psi)\simeq\rE_\Delta(\Psi)$. Let $h\in\rE_\Delta(\Psi)$
and suppose that $h$ acts on $\rad(\Psi)$ trivially. We claim that
$h$ acts on $\Psi$ trivially.

To see this, we note that $h$ induces an isomorphism on the Dynkin
diagram, so $h$ can only act on $\Delta$ trivially or exchange
$e_i-e_{i+1}$ with $e_{n-i}-e_{n-i+1}$.

Let $h(e_1)=\overset{n}{\underset{i=1}{\sum}} a_ie_i$.

Suppose that $h$ exchanges $e_i-e_{i+1}$ with $e_{n-i}-e_{n-i+1}$.
Since
\begin{center}
$\alpha^{\vee}(e_1)=h(\alpha)^{\vee}(h(e_1))$ for all
$\alpha^{\vee}\in\rR^{\vee}$, \end{center} we have
\begin{align*}
a_{n-1}&=a_n+1, \\
a_i&=a_{n-1},\ i=1,...,n-1.
\end{align*}
Besides, $h$ acts on $\rad(\Psi)$ trivially, so
$e_1-h(e_1)=\overset{n}{\underset{i=1}{\sum}} b_i(e_i-e_{i+1})$. By
summing up the coefficients, we have
$\overset{n}{\underset{i=1}{\sum}} a_i=1$ and hence $na_{n-1}=2$.
Since $a_i$'s are integers, the only possibility is $n=2$ and
$a_1=1$. In this case, $h(e_1)=e_1$ and $h(e_1-e_2)=e_1-e_2$, so $h$
is identity.

Now suppose that $h$ acts on $\Delta$ trivially. Then by the same
reasoning, we have\begin{align*}
a_1&=a_2+1, \\
a_2&=a_i,\ i=2,...,n.
\end{align*}
Besides, $h$ acts on $\rad(\Psi)$ trivially, so
$\overset{n}{\underset{i=1}{\sum}} a_i=1$. Therefore, $a_1=1$ and
$a_i=0$ for $i\neq 1$, which means $h$ is the identity.

This proves that $\eta$ is injective.

On the other hand, since the $-1$ map on $\Psi$ induces the $-1$ map
on $\rad(\Psi)$ and $\ul{\Aut}(\rad(\Psi))\simeq\ent/2\ent$, we have
$\eta$ is surjective and hence an isomorphism.

To prove (2), we choose a maximal torus $\rT'$ of $\rG$. Note that
since $K$ is in the center of $\rA$, $i_\rG(\rZ)$ is in the center
of $\rG$ and hence $i_\rG(\rZ)$ is in $\rT'$. Let $j_\rG$  be the
map between character groups of $\rT'$ and $\rZ$ induced by $i_\rG$.
Let $\Psi_\rG=\Phi(\rG,\rT')$. Then we have a natural morphism from
$\ul{\Isom}(\Psi,\Psi_\rG)$ to
$\ul{\Isom}(\rad(\Psi),\rad(\Psi_\rG))$. Since the Weyl group
$\rW(\Psi)$ acts trivially on $\rad(\Psi)$, the above morphism
induces an morphism from $\ul{\Isomext}(\Psi,\Psi_\rG)$ to
$\ul{\Isom}(\rad(\Psi),\rad(\Psi_\rG))$, which is an isomorphism.
Through $j_\rT$ and $j_\rG$, we have an isomorphism  from
$\ul{\Isom}(\rad(\Psi),\rad(\Psi_\rG))$ to $\ul{\Aut}(\Psi_\rZ)$,
which sends $f^{\Psi}$ to $j_\rG\circ f^{\Psi}\circ j_\rT^{-1}$.
Therefore, we have
$$\zeta:\ul{\Isomext}(\Psi,\Psi_\rG)\ra\ul{\Aut}(\Psi_\rZ).$$ Since
$\ul{\Autext}(\Psi)\simeq\ul{\Aut}(\Psi_\rZ)$ and $\zeta$ is
compatible with the $\ul{\Aut}(\Psi_\rZ)$-action, $\zeta$ is an
isomorphism between $\ul{\Aut}(\Psi_\rZ)$-principal homogeneous
spaces. Since there is a canonical isomorphism from
$\ul{\Isomext}(\Psi,\Psi_\rG)$ to $\ul{\Isomext}(\Psi,\rG)$, the
result then follows.
\end{proof}

With the  notations defined in the above lemma, we let
$u\in\ul{\Isomext}(\Psi,\rG)(k)$ be $\zeta^{-1}(1)$. Then for a
$k$-embedding $f$, we see that $f\circ i_\rT=i_\rG$ if and only if
$f$ is a $k$-point of $\fE(\rG,\Psi,u).$ Hence again, we can reduce
the embedding problem to the existence of rational points of
$f\in\fE(\rG,\Psi,u)$ and reformulate Prasad-Rapinchuk's Theorem as
the following (\cite{PR1}, Thm. 4.1):

\begin{theo}
Suppose that $\tau$ is an involution of the second type. If $\rE$ is
a field, then the local-global principle for the $K$-embeddings from
$(\rE,\sigma)$ to $(\rA,\tau)$ holds.
\end{theo}
\begin{proof}
By Lemma~\ref{2.11}, we can fix an orientation $u$ such that $f$ is
a $k$-point of $\fE(\rG,\Psi,u)$ if and only if $f$ is a
$K$-embedding. By Remark~\ref{1.12}, $\fE(\rG,\Psi,u)(k_v)$ is
nonempty if and only if $\fE(\rG,\Psi)(k_v)$ is nonempty. Hence, it
suffices to show that the local-global principal for
$\fE(\rG,\Psi,u)$ holds. By Theorem~\ref{2.2}, we only need to show
that $\sha^2(k,\sico(\rT))$ vanishes. Consider the exact sequence
\begin{center}
 $\xymatrix{
0\ar[r]&\sico(\rT)\ar[r]
  &\rR_{\rE^{\sigma}/k}(\rR^{(1)}_{\rE/\rE^{\sigma}}(\gm_m))
\ar[r] &\rR^{(1)}_{K/k}(\gm_m)\ar[r]& 0}$,
\end{center}
where we derive the long exact sequence:
\begin{center}
 $\xymatrix{
\ar[r]& \rH^1(k,\sico(\rT)\ar[r])\ar[r]&
\rH^{1}(k,\rR_{\rE^{\sigma}/k}(\rR^{(1)}_{\rE/\rE^{\sigma}}(\gm_m)))
\ar[r]& \rH^{1}(k,\rR^{(1)}_{K/k}(\gm_m)\ar[r])& \\
\ar[r]&\rH^{2}(k,\sico(\rT))\ar[r]&
\rH^{2}(k,\rR_{\rE^{\sigma}/k}(\rR^{(1)}_{\rE/\rE^{\sigma}}(\gm_m)))...}$.
\end{center}
Since
$\sha^{2}(k,\rR_{\rE^{\sigma}/k}(\rR^{(1)}_{\rE/\rE^{\sigma}}(\gm_m)))=\sha^{2}(\rE^{\sigma},\rR^{(1)}_{\rE/\rE^{\sigma}}(\gm_m))=0$,
we know that $\sha^{2}(k,\sico(\rT))$ is in the image of
$\rH^{1}(k,\rR^{(1)}_{K/k}(\gm_m))=k^{\times}/\Nr_{K/k}(K^{\times})$,
where $\Nr_{K/k}$ denotes the norm map from $K$ to $k$. Let $x$ be
an element of $k^{\times}$ and suppose that $x$ is mapped to
$\sha^{2}(k,\sico(\rT))$. At each place $v\in\Omega_k$, let $x_v$ be
the image of $x$ in $k_v$. Since $x$ is mapped to
$\sha^{2}(k,\sico(\rT))$,  $x$ belongs to $k^{\times}\bigcap
\Nr_{\rE^{\sigma}/k}(\rI_{\rE^{\sigma}})\Nr_{K/k}(\rI_K)$, where
$\rI_{\rE^{\sigma}}$ and $\rI_{K}$ are id\`ele groups of
$\rE^{\sigma}$ and $K$ respectively. By Hasse multinorm principle
(ref.~\cite{PlR} Prop. 6.11), x belongs to
$\Nr_{\rE^{\sigma}/k}({\rE^{\sigma}})\Nr_{K/k}(K)$,  so $x$ is in
the image of
$\rH^{1}(k,\rR_{\rE^{\sigma}/k}(\rR^{(1)}_{\rE/\rE^{\sigma}}(\gm_m)))$.
Hence $x$ is mapped to 0 in $\sha^{2}(k,\sico(\rT))$, which implies
$\sha^{2}(k,\sico(\rT))=0$. The theorem then follows.
\end{proof}
\end{subsubsection}
\end{subsection}
\end{section}
\subsection*{Acknowledgements}
Thanks to Brian Conrad,  Philippe Gille, and Boris Kunyavskii for
their precious suggestions and comments.

\addcontentsline{toc}{section}{References}
\bibliographystyle{alpha}

\end{document}